\documentclass[a4paper, 10pt, twoside, notitlepage]{amsart}

\usepackage{comment}
\usepackage[utf8]{inputenc}
\usepackage{color}
\usepackage{amsmath} 
\usepackage{amssymb} 
\usepackage{amsthm}
\usepackage{geometry}
\usepackage{graphicx}
\usepackage{esint}
\usepackage[colorlinks=true,linkcolor=blue]{hyperref}
\usepackage{tikz}
\usetikzlibrary{patterns} 
\usepackage{xfrac}
\usepackage{soul}

\theoremstyle{plain}
\newtheorem{thm}{Theorem}
\newtheorem{prop}{Proposition}[section]
\newtheorem{lem}[prop]{Lemma}
\newtheorem{cor}[prop]{Corollary}

\newtheorem{defi}[prop]{Definition}
\newtheorem{rmk}[prop]{Remark}

\newtheorem{example}[prop]{Example}

\newcommand {\R} {\mathbb{R}} 
 \newcommand {\N} {\mathbb{N}}
\newcommand {\C} {\mathbb{C}} 

\newcommand {\Ss} {\mathbb{S}}
\newcommand {\p} {\partial}

\newcommand {\D} {\Delta}

\newcommand {\supp} {\text{supp}}
\newcommand {\diam} {\text{diam}}

\DeclareMathOperator {\dist} {dist}

\DeclareMathOperator {\sign} {sgn}

\title[Nonlocal Ellipticity and Gauges]{On the Calder\'on problem for nonlocal Schr\"odinger equations with homogeneous, directionally antilocal principal symbols}

\author{Giovanni Covi}
\address{Institut für Angewandte Mathematik, Ruprecht-Karls-Universität Heidelberg, Im Neuenheimer Feld 205, 69120 Heidelberg, Germany}
\email{giovanni.covi@uni-heidelberg.de}
\author{Mar\'ia \'Angeles Garc\'ia-Ferrero}
\address{BCAM – Basque Center for Applied Mathematics, Alameda de Mazarredo 14, 48009 Bilbao, Spain}
\email{mgarcia@bcamath.org}
\author{Angkana Rüland}
\address{Institut für Angewandte Mathematik, Ruprecht-Karls-Universität Heidelberg, Im Neuenheimer Feld 205, 69120 Heidelberg, Germany}
\email{Angkana.Rueland@uni-heidelberg.de}

\begin{document}
\begin{abstract}
In this article we consider direct and inverse problems for $\alpha$-stable, elliptic nonlocal operators whose kernels are possibly only supported on cones and which satisfy the structural condition of \emph{directional antilocality} as introduced in \cite{I86}. We consider the Dirichlet problem for these operators on the ``domain of dependence of the operator'' and in several, adapted function spaces. This formulation allows one to avoid natural ``gauges'' which would else have to be considered in the study of the associated inverse problems. Exploiting the directional antilocality of these operators we complement the investigation of the \emph{direct problem} with infinite data and single measurement uniqueness results for the associated \emph{inverse problems}. Here, due to the only directional antilocality, new geometric conditions arise on the measurement domains. We discuss both the setting of symmetric and a particular class of non-symmetric nonlocal elliptic operators, and contrast the corresponding results for the direct and inverse problems. In particular for only ``one-sided operators'' new phenomena emerge both in the direct and inverse problems: For instance, it is possible to study the problem in data spaces involving local and nonlocal data, the unique continuation property may not hold in general and further restrictions on the measurement set for the inverse problem arise.
\end{abstract}

\maketitle

\tableofcontents
\addtocontents{toc}{\setcounter{tocdepth}{1}}

\section{Introduction}
\label{sec:intro}

Nonlocal elliptic operators like the fractional Laplacian arise in various settings in physics \cite{DGLZ12, Er02, GL97, La00, MK00, ZD10}, engineering \cite{GO08}, mathematical finance \cite{AB88, Le04, Sc03}, ecology \cite{Hum10, MV18, RR09} and turbulent fluid dynamics \cite{Co06, DG13}, among others (also check the survey \cite{BV16}). Compared to their local counterparts, they display striking novel phenomena including their boundary regularity \cite{G15,RS16} and very strong rigidity and flexibility properties \cite{GSU16, GRSU18, RS17, Rue20}. The latter have major implications on the fractional Calder\'on problem, an associated inverse problem, which had been introduced in \cite{GSU16} for the fractional Laplacian. Exploiting nonlocality, for this operator and closely related operators modelled on it, partial data uniqueness, stability and recovery results could be proved for the associated inverse problems  which are not known in the same generality for the corresponding classical local counterparts (see, for instance, \cite{GSU16,GRSU18,Rue20, RS17}, the survey articles \cite{S17, Rue18} and the references in Section ~\ref{sec:lit} below).

Motivated by applications of such nonlocal elliptic operators and the rigidity and flexibility properties of the fractional Laplacian, in this article we study direct and inverse problems of Calder\'on type for nonlocal elliptic operators, which, in general, \emph{differ} from the fractional Laplacian and the fractional Calder\'on problem in their rigidity and flexibility properties. Focusing on specific classes of generators of stable processes which, for instance, arise in central limit theorems, phyics and economics \cite{ST17, JW94, W84, KRSB09, B12}, in this article we show the following properties: 
\begin{itemize}
\item In spite of their nonlocality and ellipticity, the specific, \emph{non-isotropic geometries} of the operators are already reflected in our formulation and the properties of the associated direct problems. This leads to a formulation of the problem in a suitable ``domain of dependence'', see Sections ~\ref{sec:direct} and ~\ref{sec:well_posedAsp}.
\item These operators enjoy much weaker rigidity and flexibility properties than the fractional Laplacian in that the global Runge approximation property and even the weak unique continuation property may fail for certain (one-sided) examples of these operators, see Sections ~\ref{sec:Runge_approx_first} and ~\ref{sec:Runge_approx_one_sided}.
\item The domain of dependence structure gives rise to natural and at least partially necessary \emph{geometric restrictions} for the formulation and derivation of the uniqueness results for the associated nonlocal Calder\'on type inverse problem. In general, these may enjoy only substantially weaker properties than the analogous problems for the fractional Laplacian, see Theorems ~\ref{thm:inf_meas}, ~\ref{thm:single_meas} and Section ~\ref{sec:inv_ex}.
\item The notion of \emph{directional antilocality} as introduced in \cite{I86,I88,I89} can partially compensate for this and provide certain replacements of the rigidity and flexibility of the fractional Laplacian, see the discussion in Sections ~\ref{sec:inv_first} and ~\ref{sec:directional_anti}.
\end{itemize}
We discuss these properties both for a family of \emph{symmetric} and a \emph{model family of non-symmetric, nonlocal elliptic operators}. The latter are of particular theoretical interest since they allow for the simultaneous prescription of \emph{local} and \emph{nonlocal} boundary data, which leads to new effects in the direct and the inverse problem formulation and results.

\subsection{The direct problem}

\label{sec:formulations}

In the sequel, as a model setting of a family of \emph{symmetric} nonlocal elliptic operators we consider a subclass of elliptic nonlocal operators which, from a stochastic point of view, are generators of $2s$-stable L\'evy processes. These and related operators naturally arise as specific examples in generalized central limit theorems \cite{ST17} but are also related to linearizations of nonlinear, nonlocal operators \cite{BHRV17,IS21}. These operators and their associated stochastic processes have been intensively studied in the probability, potential theory and regularity theory communities \cite{BS05, B09,BC10,S10, KW18, RS16,DRSSV20, Z86, SZF95}. 
Analytically, (in their strong formulations) the specific class of model operators which we consider here is of the form
\begin{align}\label{eq:L}
L u(x):= \int\limits_{\R^n}\big(2u(x)-u(x+y)-u(x-y)\big)\frac{a(y/|y|)}{|y|^{n+2s}}dy, 
\end{align}
where $s\in(0,1)$ and the kernel $a$ satisfies the properties \textnormal{(\hyperref[ass:A1]{A1})}-\textnormal{(\hyperref[ass:A3]{A3})} below. In order to stress that $L$ is a (nonlocal) \emph{differential} operator we also use the notation $L(D)$.

In all our considerations on this family of symmetric nonlocal elliptic operators, the integral kernel, determined by the function $a: \Ss^{n-1}\rightarrow \R$, satisfies the following properties: 
\begin{itemize}
\item[(A1)]\label{ass:A1} 
the kernel $a$ gives rise to an elliptic integro-differential operator in the sense that the symbol $L(\xi)$ of $L(D)$ satisfies $L(\xi):= c_s\int\limits_{\Ss^{n-1}}|\xi \cdot \theta|^{2s} a(\theta) d\theta >0$ for all $\xi \in \R^n \backslash \{0\}$, $s\in (0,1)$, $a\in L^1(\Ss^{n-1})$ (see Lemma ~\ref{lem:symbol} for the symbol computations),
\item[(A2)]\label{ass:A2} 
the function $a$ is symmetric and non-negative, i.e. $a(\theta)=a(-\theta)$ for all $\theta\in\Ss^{n-1}$ and 
$a(\theta)\geq 0$,
\item[(A3)]\label{ass:A3} 
there exists a convex, open, non-empty cone $\mathcal{C} \subset \R^{n}\backslash \{0\}$, $\mathcal{C} \neq \emptyset$ such that $a(\theta)\neq 0$ if and only if $\theta\in (-\mathcal C\cup\mathcal{C})\cap \Ss^{n-1}$.
\end{itemize}
We emphasize that these operators, in particular, satisfy the  $s$-transmission condition (see \cite{G15}, however the regularity conditions from there are violated) and have been studied in terms of their higher Sobolev and H\"older regularity properties as special examples in \cite{KRS14,RS16,CK20}. For simplicity, in this article we restrict our attention to the setting $s\in (0,1)$ and to kernels $d\mu:= a(\theta)d\theta$ being absolutely continuous with respect to the Hausdorff measure on the sphere. The case $a=1$ a.e. corresponds to the fractional Laplacian; in the sequel we will however mainly be interested in settings in which $\overline{-\mathcal C\cup \mathcal C} \cap \Ss^{n-1} \subsetneq \Ss^{n-1}$ for which only substantially weaker rigidity and flexibility conditions hold than for the setting of the fractional Laplacian.

Before turning to Calder\'on type inverse problems for this class of operators, we first consider the Dirichlet problem associated with them. Due to the nonlocality of the operators this may at first be formulated as
\begin{align}
\label{eq:Dirichlet1}
\begin{split}
\big(L(D) + q(x)\big) u & = 0 \mbox{ in } \Omega,\\
u & = f \mbox{ on } \Omega_e,
\end{split}
\end{align}
where $\Omega \subset \R^n$ is an open, bounded domain, $\Omega_e:= \R^n \backslash \overline{\Omega}$ and $q$ is in a suitable function space.
However, this first formulation leads to obvious ``gauges'' in that there are infinitely many solutions whose restriction to $\Omega$ vanishes although $f\neq 0$. Heading towards the inverse problem, it will be convenient to deal with these ``gauges'' already in the formulation of the direct problem. Indeed, this will be reflected in viewing the Dirichlet problem not as a problem in the whole space but in the ``domain of dependence" of the operator (see Figure ~\ref{figure:setting_intro}), which is given by
\begin{align}
\label{eq:domain_dep}
\mathcal{C}(\Omega):=(\Omega+ \mathcal{C})\cup (\Omega - \mathcal{C})=\{x+t\theta\in \R^n: \ x\in \Omega,\ t>0,\  \theta \in (-\mathcal C\cup\mathcal C)\cap\Ss^{n-1}\}.
\end{align}
As a consequence, both the direct and the inverse problems are still nonlocal problems, however with a more \emph{mild, only directionally nonlocal dependence} if compared to the analogous problems for the fractional Laplacian. Following \cite{I86, I88, I89, K88} (see  \cite{L82} for the isotropic case), we will work with the stronger notion of \emph{directionally antilocal} operators (see Definition ~\ref{defi:anti} below).

In view of the inverse problem, in the sequel, we contrast the first formulation from \eqref{eq:Dirichlet1} with the better adapted form \eqref{eq:Dirichlet2} from below. This illustrates the ``lesser degree of nonlocality'' of the operators under consideration when compared with the fractional Laplacian and provides the starting point for the discussion of the rigidity and flexibility properties of these operators and the investigation of the associated inverse problems.

\subsubsection{A first formulation of the direct problem, well-posedness and first consequences for the inverse problem}
\label{sec:naive}

Let us first consider the Dirichlet problem in the form of \eqref{eq:Dirichlet1} (and its associated weak form) as a problem posed in $\R^n$. Its weak form is obtained through testing and a symmetrization procedure: More precisely, we consider the bilinear form
\begin{align}\label{eq:bilR}
\tilde{B}_q(u,v):=\frac 1 2 \int\limits_{\R^n}\int\limits_{\R^n} \big(u(x)-u(y)\big)\big(v(x)-v(y)\big) \frac{a\big(\frac{x-y}{|x-y|}\big)}{|x-y|^{n+2s}}dydx + (q u,v)_{L^2(\Omega)},
\end{align}
and, for $q$ in a suitable function space and $f\in H^{s}(\Omega_e)$, we say that $u$ is a weak solution to \eqref{eq:Dirichlet1} if
\begin{align*}
\tilde{B}_q(u,v) = 0 \; \mbox{ for all } v\in \widetilde{H}^s(\Omega)
\end{align*}
and $u-f\in \widetilde{H}^{s}(\Omega)$. We refer to Section ~\ref{sec:spaces} for precise definitions of these and related function spaces.
Following the well-posedness results of \cite{FKV15} (which considers exterior data in weaker spaces), the energy arguments from \cite{GSU16} or the pseudodifferential framework from \cite{G15}, one obtains that this formulation of the problem \eqref{eq:Dirichlet1} is solvable as a mapping from $H^s(\Omega_e)$ to $H^s(\R^n)$. As a consequence, as in the case of the fractional Laplacian, one could be tempted to define the inverse problem through the associated Dirichlet-to-Neumann map given in terms of $\tilde B_q(\cdot,\cdot)$, which formally is
\begin{align}
\label{eq:DtNnaive}
\bar{\Lambda}_q : H^s(\Omega_e) \rightarrow H^{-s}(\Omega_e), \; f \mapsto \bar{\Lambda}_q(f) := L(D)u|_{\Omega_e}.
\end{align}

However, in view of the inverse problem and also the ``domain of dependence structure'' of the operator $L$ (given through the kernel $a$), the formulation \eqref{eq:Dirichlet1} and also \eqref{eq:DtNnaive} have the obvious caveat that even for $q=0$ there is an infinite dimensional set of boundary data such that the restriction to $\Omega$ of the solution to the problem \eqref{eq:Dirichlet1} yields the zero function: Indeed, any  solution with boundary datum $f\in \widetilde{H}^{s}(\Omega_e)$ with $\supp(f)\subset \R^n \backslash \overline{\mathcal{C}(\Omega)}$ is of this type. Here $\mathcal{C}(\Omega)$ denotes the domain of dependence of $L$ given $\Omega$ defined in \eqref{eq:domain_dep} (see Figure ~\ref{figure:setting_intro}).
 Compared to the setting of the fractional Calder\'on problem, in this whole space formulation, a natural ``gauge'' enters this first formulation of the direct and inverse problems through ``the domain of dependence of the operator $L$''. As a consequence, in this formulation there is no hope of deducing as general partial data or single measurement results as in \cite{GSU16, GRSU18, RS17, Rue20}, (see Lemma ~\ref{lem:necessity} for explicit examples of this).

 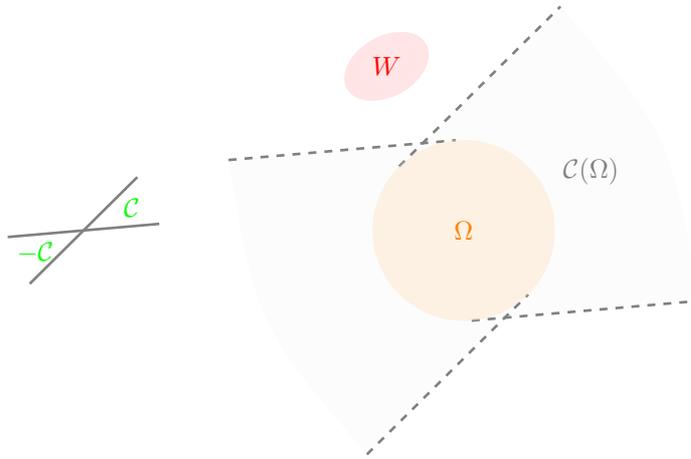
\begin{figure}[t]
\begin{tikzpicture}

\pgfmathsetmacro{\angleone}{5}
\pgfmathsetmacro{\angletwo}{45}

\pgfmathsetmacro{\anglemid}{(\angleone+\angletwo)/2}
\pgfmathsetmacro{\anglediff}{abs(\angleone-\angletwo)/2}

\pgfmathsetmacro{\midx}{1.4* cos(\anglemid)}
\pgfmathsetmacro{\midy}{1.4* sin(\anglemid)}

\pgfmathsetmacro{\thetax}{cos(\anglemid)}
\pgfmathsetmacro{\thetay}{sin(\anglemid)}
\pgfmathsetmacro{\rtheta}{sin(\anglediff)}

\pgfmathsetmacro{\angonex}{cos(\angleone-10)}
\pgfmathsetmacro{\angoney}{sin(\angleone-10)}

\begin{scope}[xshift=-5cm, scale=.5]

\node[green] at (\midx,\midy) {$\mathcal C$};
\node[green] at (-\midx,-\midy) {$-\mathcal C$};

\draw[black!50,line width=1pt, domain=-2:2] plot ({\x*cos(\angleone)},{\x * sin(\angleone)});
\draw[black!50,line width=1pt, domain=-2:2] plot ({\x*cos(\angletwo)},{\x * sin(\angletwo)});

\end{scope}

\begin{scope}[scale=.6]

\pgfmathsetmacro{\r}{2}

\pgfmathsetmacro{\xx}{5}
\pgfmathsetmacro{\int}{(-\r*sin(\angletwo+90)+\r*sin(\angleone+90)+sin(\angleone)/cos(\angleone)*\r*(cos(\angletwo+90)-cos(\angleone+90)))/(sin(\angletwo)-sin(\angleone)*cos(\angletwo)/cos(\angleone))}

\fill[black!10, opacity=.1]
({\r*cos(\angletwo+90)+\xx*cos(\angletwo)},{\r*sin(\angletwo+90)+\xx * sin(\angletwo)}) 
--({\r*cos(\angletwo+90)+\int*cos(\angletwo)},{\r*sin(\angletwo+90)+\int * sin(\angletwo)})
 --({\r*cos(\angleone+90)-\xx*cos(\angleone)},{\r*sin(\angleone+90)-\xx * sin(\angleone)})
 .. controls ({-\xx*\thetax},{-\xx*\thetay}) ..({\r*cos(\angletwo-90)-\xx*cos(\angletwo)},{\r*sin(\angletwo-90)-\xx * sin(\angletwo)})
--({\r*cos(\angletwo-90)-\int*cos(\angletwo)},{\r*sin(\angletwo-90)-\int * sin(\angletwo)})
--({\r*cos(\angleone-90)+\xx*cos(\angleone)},{\r*sin(\angleone-90)+\xx * sin(\angleone)}) 
.. controls ({\xx*\thetax},{\xx*\thetay}) ..
({\r*cos(\angletwo+90)+\xx*cos(\angletwo)},{\r*sin(\angletwo+90)+\xx * sin(\angletwo)});

\draw[black!50,dashed,line width=1pt, domain=0:\xx] plot ({\r*cos(\angletwo+90)+\x*cos(\angletwo)},{\r*sin(\angletwo+90)+\x * sin(\angletwo)});
\draw[black!50,dashed,line width=1pt, domain=-\xx:0] plot ({\r*cos(\angleone+90)+\x*cos(\angleone)},{\r*sin(\angleone+90)+\x * sin(\angleone)});
\draw[black!50,dashed,line width=1pt, domain=-\xx:0] plot ({\r*cos(\angletwo-90)+\x*cos(\angletwo)},{\r*sin(\angletwo-90)+\x * sin(\angletwo)});
\draw[black!50,dashed,line width=1pt, domain=0:\xx] plot ({\r*cos(\angleone-90)+\x*cos(\angleone)},{\r*sin(\angleone-90)+\x * sin(\angleone)});

\fill[orange!20, opacity=.5] (0,0) circle [radius=\r,];
\node[orange] at (0,0) {$\Omega$};

\node[black!50] at ({(\r+.2)*\midx}, {(\r+.2)*\midy}) {$\mathcal C(\Omega)$};

\pgfmathsetmacro{\Wx}{(2*\r)* cos(\anglemid+90)}
\pgfmathsetmacro{\Wy}{(2*\r)* sin(\anglemid+90)}
\pgfmathsetmacro{\Wrx}{\r/2}
\pgfmathsetmacro{\Wry}{\r/3}

\fill[red!50, opacity=.2] (\Wx,\Wy) ellipse [x radius=\Wrx cm,y radius=\Wry cm, rotate=30];
\node[red] at (\Wx,\Wy) {$W$};

\end{scope}
\end{tikzpicture}
\caption{Domain of dependence of $L$ given $\Omega$. The domain $\mathcal{C}(\Omega)$ includes the set $\Omega$ (orange) and the grey cone emanating from its points. For $q=0$ and any  exterior data $f$ supported in the domain $W$, the solution restricted to \eqref{eq:Dirichlet1} vanishes. This in particular shows that without taking this ``obvious'' geometry of the operator $L(D)$ into account in the formulations of the direct and inverse problems, no single (nor infinite data) measurement uniqueness results can hold.}
\label{figure:setting_intro}
\end{figure}

This discussion hence strongly suggests to include the \emph{geometry of the operator $L$} with its domain of dependence structure into the formulations of both the direct and inverse problems.

\subsubsection{An alternative, more adapted,  ``domain-of-dependence'' formulation of the direct problem, well-posedness and the inverse problem}
\label{sec:better}

In the sequel, rather than viewing the Dirichlet problem as a problem on $\R^n$, for a given open, bounded set $\Omega$, taking the domain of dependence structure of the operators into account, we will regard it as a problem on the set $\mathcal{C}(\Omega)\subset \R^n$ from \eqref{eq:domain_dep}. In its strong form this corresponds to

\begin{align}
\label{eq:Dirichlet2}
\begin{split}
\big(L(D) + q\big) u & = 0 \mbox{ in } \Omega,\\
u & = f \mbox{ on } \mathcal{C}(\Omega)\backslash \overline{\Omega}.
\end{split}
\end{align}
Here, as above, $\mathcal{C}(\Omega)$ denotes the ``domain of dependence of $L$ given $\Omega$''. We emphasize that we do \emph{not} prescribe data in $\Omega_e \backslash \mathcal{C}(\overline \Omega)$ in \eqref{eq:Dirichlet2}, since the operator does not ``see'' this part of the complement. This avoids the obvious ``gauges'' or ``degeneracies'' of the first formulation from \eqref{eq:Dirichlet1} and the discussion in the previous section.
Using suitable bilinear forms associated to the elliptic operator $L(D)$ (or more precisely, suitable Dirichlet forms), we will study weak versions of this problem as a map from $H^s(\mathcal{C}(\Omega)\backslash \overline{\Omega})$ to $H^{s}(\mathcal{C}(\Omega))$. While there are many possible choices of associated Dirichlet forms, we will focus on two different, naturally arising bilinear forms. In particular, we will prove the problem's well-posedness outside of a (discrete) set of countably many eigenvalues in different functional settings. The choice of the Dirichlet form strongly influences the inverse problem and its natural measurement setting (see the discussion in Section ~\ref{sec:DtN_symm} below).

As a first weak manifestation of \eqref{eq:Dirichlet2}, using a weak formulation for \eqref{eq:Dirichlet2}, we rely on the following adapted bilinear form
\begin{align}
\label{eq:bil}
B_q(u,v):= \frac 1 2 \int\limits_{\mathcal{C}(\Omega)}\int\limits_{\mathcal{C}(\Omega)} \big(u(x)-u(y)\big)\big(v(x)-v(y)\big) \frac{a\big(\frac{x-y}{|x-y|}\big)}{|x-y|^{n+2s}}dydx+ (q u,v)_{L^2(\Omega)}.
\end{align}
This directly restricts the domain which is ``visible'' for the operator to the cone $\mathcal{C}(\Omega)$. It is reminiscent of bilinear forms used in the connection with censored processes \cite{B03}, see also \cite{KW18}. Using the structure of $B_q$ we will deduce well-posedness in very weak function spaces modelled on the setting from \cite{FKV15} (which includes the Sobolev spaces $H^s(\mathcal{C}(\Omega)\backslash \overline{\Omega})$ in particular).

By virtue of these well-posedness results, it is possible to define the Poisson operator $P_q$ associated with the operator $L(D)$:
\begin{align}
\label{eq:Poisson}
P_q : H^s(\mathcal{C}(\Omega)\backslash \overline{\Omega}) \rightarrow H^{s}(\mathcal{C}(\Omega)), \ f \mapsto u.
\end{align}
Here $u$ is the solution to \eqref{eq:Dirichlet2} associated with the boundary data $f$ if zero is not a Dirichlet eigenvalue of the operator $L(D)+ q$.

Using the bilinear form \eqref{eq:bil} will further allow us to define a ``Dirichlet-to-Neumann map"
\begin{align}
\label{eq:DtN0}
\Lambda_q: H^s(\mathcal{C}(\Omega)\backslash \overline{\Omega}) \rightarrow H^{-s}(\mathcal{C}(\Omega)\backslash \overline{\Omega}), \ f \mapsto \Lambda_q(f),
\end{align}
and to study the associated, Calder\'on type inverse problem. 

As a second weak manifestation of the problem \eqref{eq:Dirichlet2}, we will make use of the slightly ``more global" bilinear form from \eqref{eq:bilR}. While this bilinear form is defined \emph{globally} and does not directly refer to the domain of dependence $\mathcal{C}(\Omega)$,
in order to account for the fact that we seek to solve \eqref{eq:Dirichlet2} instead of \eqref{eq:Dirichlet1}, we only consider data $f\in \widetilde H^s(\mathcal{C}(\Omega)\backslash \overline{\Omega})$. This allows us to exploit the connection of \eqref{eq:bilR} and the Fourier transform which in turn leads to very ``symmetric'' results for the inverse problem. Analogously to \eqref{eq:DtN0} and \eqref{eq:DtNnaive}, it is then possible to define a natural (and explicit) ``Dirichlet-to-Neumann'' operator for this functional setting (if zero is not a Dirichlet eigenvalue)
\begin{align}
\label{eq:DtN1}
\tilde{\Lambda}_q: \widetilde{H}^s(\mathcal{C}(\Omega)\backslash \overline{\Omega}) \rightarrow H^{-s}(\mathcal{C}(\Omega)\backslash \overline{\Omega}), \ f \mapsto \Lambda_q(f) = L(D)u|_{\mathcal{C}(\Omega)\backslash \overline{\Omega}}.
\end{align}
We note that, in general, the two Dirichlet-to-Neumann maps \eqref{eq:DtN0} and \eqref{eq:DtN1} \emph{differ} (see Remark ~\ref{rmk:inequalDtN}) and thus also the measurements for the Calder\'on inverse problems do \emph{not} coincide.

In addition to the well-posedness results, a further key structural ingredient replacing the strong rigidity and flexibility properties of the fractional Laplacian is needed in order to address the inverse problem in the following. To this end, we restrict our attention to operators which satisfy a \emph{directional antilocality} condition, as introduced in \cite{I86, I88, I89, K88} and which builds on \cite{L82}. We will discuss this notion next.

\subsection{Rigidity in the form of antilocality in cones and examples}
\label{sec:inverse}

The properties of the operator $L(D)$ and the above discussion illustrate that the redundant information of \eqref{eq:Dirichlet1} should not be considered and instead the direct and inverse problems should be considered in the framework from \eqref{eq:Dirichlet2} (with possibly different choices of Dirichlet forms). 

However, in order for $\Lambda_q$ (or $\tilde{\Lambda}_q$) to share certain properties of the inverse problem of the fractional Laplacian, we consider a final structural ingredient which is analogous to the strong global uniqueness properties of the fractional Laplacian. Following the article \cite{L82} as well as the results from \cite{I86, I88, I89, K88}, which build up on this, we consider a suitable notion of \emph{antilocality}. This notion itself originated from Reeh-Schlieder theorems in quantum physics \cite{RS61,V93, SVW02} and the study of the fractional Laplacian \cite{HJ12,R38}  and plays a major role in the analysis of the fractional Calder\'on problem with its surprisingly strong uniqueness, stability and single measurement properties \cite{GSU16,  RS17,GRSU18, Rue20} (see Section ~\ref{sec:lit} for further references on this). Due to the only directional domain of dependence of the operators $L$ in general, we here consider operators satisfying an \emph{antilocality condition in cones}:

\begin{defi}[Antilocality in cones, \cite{I89}, Definition 2.2]
\label{defi:anti}
Let $L: C_0^{\infty}(\R^n) \rightarrow C^{\infty}(\R^n)$ be a linear operator and let $\Gamma \subset \R^n \backslash \{0\}$ be a convex cone.  The operator $L$ is $\Gamma$-antilocal, if the following implication holds:
If  $f\in C_c^{\infty}(\R^n)$ and $f=L(D)f = 0$ in an open, non-empty subset $U \subset \R^n$, then $f = 0$ in $U+ \Gamma$.
\end{defi}

We remark that this definition can be generalized to more complicated geometries such as the union of disconnected convex cones.

Let us discuss the notion of directional antilocality in more detail: Just as the isotropic antilocality properties of the fractional Laplacian, also \emph{directional antilocality properties} of an operator give rise to \emph{strong rigidity properties}. In a sense, they can be viewed as a type of \emph{global unique continuation property} -- however with the major difference that the information is not propagated by the global validity of an equation but already the \emph{local} information $f=0=L(D) f$ suffices to deduce \emph{global} information. As one can easily show using the only directional domain of dependence of the operators under consideration, antilocality in cones does not imply global antilocality in general (for the operators from \eqref{eq:L} satisfying \textnormal{(\hyperref[ass:A1]{A1})}-\textnormal{(\hyperref[ass:A3]{A3})} from above, the considerations from above illustrate that clearly global antilocality does not hold and at best directional antilocality in the cones $-\mathcal C \cup \mathcal{C}$ could be valid, see also Figure ~\ref{figure:setting_intro}).

Examples of only directionally antilocal operators arise already in one dimension (for instance, in the description of the first hitting time of a Brownian particle, see \cite[Section 2.1]{W01}) and will play an important model role in our discussion below.

\begin{example}[A 1D operator which is antilocal to the right/left, \cite{I86}]
\label{ex:model1}
Following \cite{I86} and \break \cite{I88}, we consider the operators 
\begin{align}\label{eq:example1D}
A_{p}^s(D)f(x) =
\begin{cases}
\displaystyle \int\limits_{-\infty}^{\infty}\big(f(x)-f(x+y)\big)  \nu^s_p(y)dy& \mbox{ if } s\in\big(0,\frac 1 2),\\
\displaystyle \int\limits_{-\infty}^{\infty}\big(f(x)-f(x+y)+f'(x)ye^{-\frac{|y|}{e}}\big)\nu^s_p(y)dy & \mbox{ if } s=\frac 1 2,\\
\displaystyle \int\limits_{-\infty}^{\infty}\big(f(x)-f(x+y)+yf'(x)\big) \nu^s_p(y)dy & \mbox{ if } s\in\big(\frac 1 2,1),
\end{cases}
\end{align}
with $p\in [0,1]$ and
\begin{align*}
    \nu^s_p(y):=\frac{p \chi_{\R_{-}}(y) + (1-p) \chi_{\R_{+}}(y)
}{|y|^{1+2s}},
\end{align*}
 where $ \chi_{\R_{-}}(y)$ and $ \chi_{\R_{+}}(y)$ denote the characteristic functions of $\R_{-}$ and $\R_{+}$, respectively.
By the results of \cite{I86, I88}, these operators are antilocal if $p\in (0,1)$ and antilocal to the left/right if $p=1$, or $p=0$, respectively. We remark that for  $s=\frac 1 2$, \cite{I88} considers only $A^{\sfrac 12}_{\sfrac 12}$. Moreover, Ishikawa  uses $\sin y$ instead of $y e^{-\frac y e}$. In spite of this difference it is possible to directly prove one-sided antilocality for $A^{\sfrac 12}_0, A^{\sfrac 12}_1$ by the method of moments,  see Lemma ~\ref{lem:As0AntMoments} below.
The number $e$ in the exponent for the operators from \eqref{eq:example1D} in the case $s= \frac{1}{2}$ is chosen in order to avoid an additional constant in the symbol of the antisymmetric part, given below for any dimension.
 We further highlight that these operators do not satisfy the symmetry condition from \textnormal{(\hyperref[ass:A2]{A2})}. Hence, compared to the operators satisfying \textnormal{(\hyperref[ass:A1]{A1})}-\textnormal{(\hyperref[ass:A3]{A3})}, new interesting phenomena arise, which are discussed in Section ~\ref{sec:exs1} below (see also Section ~\ref{sec:inv2}).
\end{example}

Moreover, it is known that certain two-dimensional operators are also antilocal in certain associated cones:

\begin{example}[2D operators which are antilocal in cones, \cite{I89}]
\label{ex:model2}
Let $\Gamma \subset \R^2\backslash \{0\}$ be an open, non-empty, convex cone. Following \cite{I89}, we consider the two-dimensional operators whose symbol is given by
\begin{align}\label{eq:example2Dsymb}
A_{0,\Gamma}^s(\xi) :=
-\Gamma(2s)\int\limits_{\Gamma\cap\Ss^{n-1}}|\xi\cdot \theta|^{2s} \big(e^{-i\pi s} \chi_{\Gamma_+(\xi)}(\theta) + e^{i\pi s} \chi_{\Gamma_-(\xi)}(\theta)\big ) d\theta,
\end{align}
where $\chi_{\Gamma_\pm}$ denotes the characteristic function of $\Gamma_{\pm}(\xi):= \Gamma \cap \{\theta\in \Ss^{1}: \pm(\xi\cdot \theta)>0\}$.  If $s \in \big(0,\frac 12\big)\cup\big(\frac 1 2,1\big)$, it is proven in \cite[Theorem 2.4]{I89} that these operators are $\Gamma$-antilocal.
\end{example}

These and related stable processes are described and characterized in \cite[Theorem 2.1]{K73}.  The operators defined by  \eqref{eq:example2Dsymb} correspond to those with measures given by the characteristic function of $\Gamma\cap\Ss^{n-1}$. 
The analogue of these operators for the case $s=\frac 12$  is given by
\begin{align*}
    A_{0,\Gamma}^{\sfrac 12}(\xi) := \int\limits_{\Gamma\cap\Ss^{n-1}}\Big(\frac{\pi}{2}|\xi\cdot\theta|-i \xi\cdot\theta\log(|\xi\cdot\theta|)\Big) d\theta,
\end{align*}
see \cite[Theorem 2.1]{K73}.
This corresponds to the following operator (with the trivial generalization to one-dimension, see Lemma ~\ref{lem:AspgammaSymb}):
\begin{align*}
    A^{\sfrac 12}_{0,\Gamma}(D) f(x)=\int_{\R^n} \big(f(x+y)-f(x)-\nabla f(x)\cdot y e^{-\frac{|y|}{e}}\big) \frac{\chi_\Gamma(y)}{|y|^{n+1}} dy.
\end{align*}

The operator classes of Examples ~\ref{ex:model1} and ~\ref{ex:model2} will be discussed in more detail in Section ~\ref{sec:exs1} below. 

While at the moment we do not know whether all the operators from \eqref{eq:L} satisfying the conditions \textnormal{(\hyperref[ass:A1]{A1})}-\textnormal{(\hyperref[ass:A3]{A3})} are directionally antilocal in the sense of Definition ~\ref{defi:anti}, we show that they satisfy a weaker type of antilocality property:
Recalling the argument which had been given in \cite[Corollary 5.2]{RS17b} and which in turn relied on \cite[Lemma 3.5.4]{Isakov11}, it is known than many operators are in a \emph{weak sense antilocal} in that, if $f\in \widetilde{H}^r(B_R)$ for some $r\in \R$ and $L(D)f=0$ in $\R^n \backslash B_R(0)$ for some large open ball $B_R(0)\subset \R^n$, then $f \equiv 0$. In \cite{RS17b} this was proved for rather large classes of pseudodifferential operators with local and nonlocal contributions. We note that this in particular includes our operators -- even if the problems may only be considered in some subset $\mathcal{C}(\Omega)\subset \R^n$:

\begin{prop}[Exterior data]
\label{prop:exterior}
Let $r,s\in \R$ and $L(D): H^{r}(\R^n) \rightarrow H^{r-2s}(\R^n)$ be a linear operator. Assume further that either 
\begin{itemize}
\item[$(i)$] $L(D)$ is $-\mathcal C\cup \mathcal C$-antilocal for some convex, open, non-empty cone $\mathcal C \subset \R^n\backslash\{0\}$, or
\item[$(ii)$]  $s\in (0,1)$, $L(D)$ is of the form \eqref{eq:L} with $a$ satisfying the conditions \textnormal{(\hyperref[ass:A1]{A1})}-\textnormal{(\hyperref[ass:A3]{A3})}.
\end{itemize}
Let $\Omega\subset\R^n$ be an open, bounded set. If $L(D)f=0$ on $\mathcal{C}(\Omega)\backslash \overline{\Omega}$ for some $f\in \widetilde H^{r}(\Omega)$, then $f \equiv 0$ in $\Omega$.
\end{prop}

We remark that, as in the case of the fractional Laplacian, it would also have been possible to consider local and nonlocal combinations of these operators and to even add further pseudodifferential contributions (see \cite[Corollary 5.2]{RS17b}).

For the fractional Laplacian (see \cite[Proposition 2.3]{GRSU18}), which enjoys isotropic antilocality properties, it was proved \cite{GSU16, GRSU18} that \emph{rigidity} -- in the form of antilocality -- and \emph{flexibility} -- in the form of Runge approximation properties -- are dual. This provided the key tool in order to investigate the inverse problem and to deduce properties for it which remain open for its classical local counterpart. 

We next show that a similar duality result also holds for directionally antilocal operators: Directional antilocality together with the well-posedness of the adjoint equation is dual to Runge approximation properties. Thus, directional antilocality should be viewed as the key mechanism for studying the associated inverse problems.

\begin{thm}[Duality between (directional) antilocality and Runge approximation]
\label{thm:dual}
Let $\mathcal C\subset\R^n\backslash\{0\}$ be an open, non-empty, convex cone and let $\Omega \subset \R^n$ be open, non-empty and bounded. Assume that $W\subset \Omega_e\cap  \mathcal{C}(\Omega)$. Let $q\in L^{\infty}(\Omega)$ and let $L(D) + q$ be a  self-adjoint elliptic operator of order $s\in (0,1)$ for which the Dirichlet problem \eqref{eq:Dirichlet2} is well-posed and gives rise to a well-defined Poisson operator, mapping from $H^{s}(\mathcal{C}(\Omega) \backslash \overline{\Omega})$ to ${H}^s(\mathcal C(\Omega))$.
Then, the following results are equivalent:
\begin{itemize}
\item[$(a)$] The set $\mathcal{R}:=\{P_q(f)|_\Omega: \ f\in C_c^{\infty}(W)\}$ is dense in $\widetilde H^s(\Omega)$.
\item[$(b)$] If $w\in \widetilde H^s(\Omega)$ (weakly) solves
\begin{align}
\label{eq:dual}
\begin{split}
\big(L(D)+q\big)w & = v  \mbox{ in } \Omega,\\
w& = 0 \mbox{ in } \mathcal{C}(\Omega) \backslash \overline{\Omega},
\end{split}
\end{align}
for some $v\in H^{-s}(\Omega)$
and $L(D)w = 0 $ in $W$, then $v \equiv 0$ and, hence, $w\equiv 0$ in $\mathcal{C}(\Omega)$.
\end{itemize} 
\end{thm}

Hence, the strong rigidity properties of the antilocality condition in cones directly transfers to strong flexibility in the form of Runge approximation properties. In contrast to the setting of the isotropic fractional Laplacian, we emphasize that, in the only \emph{directionally} antilocal framework, the verification of these conditions for our operators requires the \emph{geometric conditions that} $W \subset \mathcal{C}(\Omega)$ and that also $\Omega \subset \mathcal{C}(W):= (W-\mathcal{C})\cup (W+\mathcal{C})$ (see Theorem \ref{thm:Runge}).
In Lemma ~\ref{lem:necessity} we will show that these conditions are indeed \emph{necessary} in order to deduce the Runge approximation property for a $-\mathcal{C}\cup\mathcal C$ antilocal, elliptic operator, and that these thus reflect the geometry of the operator. 

\begin{rmk}[Applicability to the exterior problem for \eqref{eq:L}]
\label{rmk:L_ops_ext}
We remark that the if $L(D)$ is of the form  \eqref{eq:L} with $a$ satisfying  \textnormal{(\hyperref[ass:A1]{A1})}-\textnormal{(\hyperref[ass:A3]{A3})} together with the additional conditions that $\Omega$ is Lipschitz and  $q \in L^{\infty}(\Omega)$ is such that zero is not a Dirichlet eigenvalue of $L(D) + q$ imply the well-posedness hypotheses in Theorem ~\ref{thm:dual} (see Section ~\ref{sec:wellpos}).
In the particular case that we consider  the (conical) exterior  of $\Omega$, i.e. $W= \mathcal{C}(\Omega)\backslash \overline{\Omega}$, this duality  holds by virtue of Proposition ~\ref{prop:exterior}, even though we did not prove the full $-\mathcal C\cup \mathcal{C}$-antilocality of these operators.
\end{rmk}

\subsection{The inverse problem}
\label{sec:inv_full}
As consequences of the above structural discussion, we next deduce properties of the associated inverse problems. We split this into two parts: First, we present some of the main results for the inverse problems in the setting of the symmetric operators satisfying the conditions \textnormal{(\hyperref[ass:A1]{A1})}-\textnormal{(\hyperref[ass:A3]{A3})} from above, and then we illustrate new phenomena arising in the study of the inverse problems associated with the operators from Examples ~\ref{ex:model1} and ~\ref{ex:model2}.

\subsubsection{The inverse problem for symmetric operators satisfying the conditions \textnormal{(\hyperref[ass:A1]{A1})}-\textnormal{(\hyperref[ass:A3]{A3})}}
\label{sec:inv1}

We first focus on the partial data inverse problem and its uniqueness properties for the operators satisfying \textnormal{(\hyperref[ass:A1]{A1})}-\textnormal{(\hyperref[ass:A3]{A3})} above. In the next section, we will in parallel collect results for the non-selfadjoint operators from Examples ~\ref{ex:model1} and ~\ref{ex:model2}. We begin by providing Runge approximation results under the condition that the operator is directionally antilocal and under suitable resulting geometric assumptions.

\begin{thm}[Runge approximation]
\label{thm:Runge}
Let $\Omega \subset \R^n$ be open, non-empty,  bounded and let $q\in L^\infty(\Omega).$
Let $L(D)$ be the operator in \eqref{eq:L}
of order $s\in (0,1)$ with $a$ satisfying the conditions \textnormal{(\hyperref[ass:A1]{A1})}-\textnormal{(\hyperref[ass:A3]{A3})} from above. Suppose that $q$ is such that the Dirichlet problem \eqref{eq:Dirichlet2} is well-posed and gives rise to a well-defined Poisson operator, mapping from $H^{s}(\mathcal{C}(\Omega) \backslash \overline{\Omega})$ to ${H}^s(\mathcal C(\Omega))$.
Assume in addition that  $L(D)$ is $-\mathcal C\cup \mathcal C$-antilocal. Consider an open set $W$ such that $W \subset \mathcal{C}(\Omega)$ and $\Omega \subset \mathcal{C}(W)$.
Then, the set
\begin{align*}
\mathcal{R}:= \big\{u= P_q f|_{\Omega}:\ f\in C_c^{\infty}(W) \big\}
\end{align*}
is dense in $\widetilde H^s(\Omega)$.
\end{thm}

Let us discuss the imposed conditions of the theorem: First of all, solvability is assumed in order to define a suitable Dirichlet-to-Neumann operator (this could be relaxed to a Cauchy data setting). Under mild conditions on the potential and in various function spaces, in Section ~\ref{sec:direct} we prove that this condition is satisfied for our class of operators obeying the conditions \textnormal{(\hyperref[ass:A1]{A1})}-\textnormal{(\hyperref[ass:A3]{A3})} from above. Due to the domain of dependence structure of $L(D)$, the geometric condition that $W\subset \mathcal{C}(\Omega)$ is natural and part of our well-posedness theory. The key structural condition is that of antilocality, which allows us to invoke Theorem ~\ref{thm:dual}. However, we emphasize that the only \emph{directional} antilocality condition again gives rise to an \emph{additional geometric condition}. In addition to the assumption that $W \subset \mathcal{C}(\Omega)$ which stems from our well-posedness considerations, we also impose that $\Omega \subset \mathcal{C}(W)$. This allows us to invoke the assumed $-\mathcal C\cup\mathcal{C}$-antilocality of the operators in order to infer the desired density condition. In Lemma ~\ref{lem:necessity} below, we discuss examples illustrating that such additional geometric conditions are indeed \emph{necessary}.

While we have formulated Theorem ~\ref{thm:Runge} for an abstract class of operators and have imposed directional antilocality, we remark that by virtue of Proposition ~\ref{prop:exterior} for the ``full data'' exterior problem in which $W= \mathcal{C}(\Omega)\backslash \overline{\Omega}$,  the assumptions of Theorem ~\ref{thm:Runge} are all satisfied. 
Hence, in this setting, for our operators from \eqref{eq:L} with \textnormal{(\hyperref[ass:A1]{A1})}-\textnormal{(\hyperref[ass:A3]{A3})} (with mild conditions on the potential $q$) the result from Theorem ~\ref{thm:Runge} is valid.
The same observation holds for Theorems  ~\ref{thm:inf_meas}  and ~\ref{thm:single_meas} below if 
$W_2=\mathcal{C}(\Omega)\backslash \overline{\Omega}$.

Relying on the Runge approximation result, as in the fractional Calder\'on problem, we obtain an associated infinite data measurement result:

\begin{thm}[Infinite measurement uniqueness  under geometric restrictions]
\label{thm:inf_meas}
Let $\Omega \subset \R^n$ be open, non-empty, bounded and let $q_1, q_2\in L^{\infty}(\Omega)$.
Let $L(D)$ be the operator in \eqref{eq:L}
of order $s\in (0,1)$ with $a$ satisfying the conditions \textnormal{(\hyperref[ass:A1]{A1})}-\textnormal{(\hyperref[ass:A3]{A3})}. Suppose that  $q_1, q_2$ are such that the Dirichlet problem \eqref{eq:Dirichlet2} with the potential $q_j$, $j\in \{1,2\}$, is well-posed and gives rise to a well-defined Poisson operator, mapping from $H^{s}(\mathcal{C}(\Omega) \backslash \overline{\Omega})$ to ${H}^s(\mathcal C(\Omega))$. Assume that  $L(D)$ is $-\mathcal C\cup\mathcal C$-antilocal. Consider open, non-empty  sets $W_1, W_2$ such that $W_j \subset \mathcal{C}(\Omega)$ and $\Omega\subset\mathcal{C}(W_j)$ for $j=1,2$.
Assume that
\begin{align*}
\Lambda_{q_1}(f)|_{W_2} = \Lambda_{q_2}(f)|_{W_2}\; \mbox{ for } f\in C_c^{\infty}(W_1),
\end{align*}
then $q_1 = q_2$.
\end{thm}

\begin{rmk} \label{rmk:inf_meas_tilde}
While the two Dirichlet-to-Neumann maps from \eqref{eq:DtN0}, \eqref{eq:DtN1} do \emph{not} coincide in general, we remark that an analogous statement as in Theorems ~\ref{thm:Runge} and ~\ref{thm:inf_meas} also holds for $\tilde \Lambda_{q_1}(f)|_{W_2}=\tilde \Lambda_{q_2}(f)|_{W_2}$.
Further we stress that analogously to \cite{RS17} it would also have been possible to work with potentials in critical multiplier spaces. Since the structural conditions on the principal symbols are our main focus in this article, we have opted not to include this here.
\end{rmk}

Similarly as in the fractional Calder\'on problem, also the infinite, partial data measurement nonlocal Calder\'on problems studied here are always formally overdetermined inverse problems and the corresponding single measurement problems are always formally determined. Thus, there is at least formal reason to consider single measurement uniqueness results. Exploiting the $-\mathcal C\cup \mathcal{C}$-antilocality together with the weak unique continuation property (Proposition ~\ref{prop:WUCP}, which here is a consequence of the two-sided antilocality property of the operator), it is then indeed possible to prove such single measurement results.

\begin{thm}[Single measurement uniqueness  under geometric restrictions]
\label{thm:single_meas}
Let $\Omega \subset \R^n$ be open, non-empty, bounded, $C^1$ regular  and let $q\in C^{0}(\Omega)$.
Let $L(D)$ be the operator in \eqref{eq:L}
of order $s\in (0,1)$ with $a$ satisfying the conditions \textnormal{(\hyperref[ass:A1]{A1})}-\textnormal{(\hyperref[ass:A3]{A3})}. Suppose that $q$ is such that the Dirichlet problem \eqref{eq:Dirichlet2}  is well-posed and gives rise to a well-defined Poisson operator, mapping from $H^{s}(\mathcal{C}(\Omega) \backslash \overline{\Omega})$ to ${H}^s(\mathcal C(\Omega))$. Assume that  $L(D)$ is $-\mathcal C\cup \mathcal C$-antilocal. Consider  open, non-empty sets $W_1, W_2 \subset \mathcal{C}(\Omega)$ such that $\Omega \subset \mathcal{C}(W_2)$.
Then $f\in \widetilde H^s(W_1)\backslash \{0\}$ and $\tilde \Lambda_q(f)|_{W_2}$ determine $q$ uniquely.
\end{thm}

We emphasize that in both Theorems ~\ref{thm:inf_meas} and ~\ref{thm:single_meas} the choice $W_1=W_2$ is admissible.

\begin{rmk}\label{rmk:single_meas_tilde}
The unique determination also follows from the knowledge of $f\in \widetilde H^s(W_1)\backslash\{0\}$ and $\Lambda_q(f)|_{W_2}$ under the additional geometric conditions that $ W_1\cap W_2=\emptyset$ or $\mathcal C(W_1\cap W_2)\subset \mathcal C(\Omega)$ (see Lemma ~\ref{lem:idtwoDtN}). 
\end{rmk}

We remark that this result is essentially in parallel to the single measurement results for the fractional Calder\'on problem. The geometric conditions on $W_1, W_2$ are only (very mild) consequences of the domain of dependence structure of the problem and of the directional antilocality.

\subsubsection{The inverse problem for the model operators from Examples ~\ref{ex:model1} and  ~\ref{ex:model2}}
\label{sec:inv2}
In Section ~\ref{sec:exs1}, we contrast the results from the symmetric setting from the previous section with the results for the inverse problem for the explicit examples of not necessarily symmetric operators from Examples ~\ref{ex:model1} and  ~\ref{ex:model2} (generalized to any $p\in[0,1]$ and any dimension in \eqref{eq:Aspgamma}). Here new phenomena arise both for the direct and the inverse problems:
\begin{itemize}
\item Already in the formulation of the direct problem, the cases $p\in\{0,1\}$ give rise to new data spaces in that both local and nonlocal contributions can be considered (see Section ~\ref{sec:well_posedAsp}).
\item For $p\in\{0,1\}$ the operators are not antilocal in two-sided cones but only antilocal in one-sided cones, which in one dimension turns into antilocality to the left and right, respectively (see Section ~\ref{sec:directional_anti}).
\item While the weak unique continuation property holds for $n=1$ (as a consequence of well-posedness; see Lemma ~\ref{lem:1Dwucp}), for $n\geq 2$ and $p\in\{0,1\}$ this fails in general for $s\in (0,\frac{1}{2})$, see Lemma ~\ref{lem:lackWUCP} (the case $p=0$ is defined in \eqref{eq:example2Dsymb}, in parallel to the one-dimensional setting we will consider a whole family parametrized by $p\in [0,1]$). While the operators are thus nonlocal and elliptic and ``have constant coefficients'', they are substantially less rigid than the fractional Laplacian or local, constant coefficient elliptic partial differential operators. The only one-sided antilocality is strongly reflected in this result.
\item Due to the only one-sided domain of dependence for the operators with $p\in \{0,1\}$ in Examples ~\ref{ex:model1} and  ~\ref{ex:model2}, new restrictions have to be imposed on the data in the inverse problems. In particular, restricted to the domain of dependence, the associated Dirichlet-to-Neumann maps do not carry information on the operator (but only on the data). Thus, for instance, in the single measurement results, data have to be taken in the ``opposite'' cone in order to infer non-trivial information (see Section ~\ref{sec:inv_ex}).
\end{itemize}

Moreover, in Section ~\ref{sec:exs2} we illustrate that many ``natural'' nonlocal elliptic operators which are defined as sums of certain rigid nonlocal operators do not enjoy arbitrarily strong antilocality properties but only satisfy weaker forms of antilocality and rigidity, for which geometric constraints have to be imposed.

\subsection{Relation to the literature on Calder\'on type inverse problems}
\label{sec:lit}
The problem under investigation in this article should be viewed as a generalization of the fractional Calder\'on problem in which the geometry of the problem plays a more prominent role and which displays weaker rigidity and flexibility properties than the fractional Calder\'on problem. 

The study of the fractional Calder\'on problem has been a very active field in the past years: After its introduction in \cite{GSU16}, in which the partial data infinite measurement result at $L^{\infty}$ coefficient regularity was proved, many facets have been addressed. This includes the study of uniqueness and (in-)stability in the low regularity regime \cite{RS17, RS18}, qualitative and quantitative single measurement results and reconstruction \cite{GRSU18, Rue20}, inversion methods by monotonicity \cite{HL19, HL20}, nonlinear problems \cite{LL20, LO20}, uniqueness in the presence of anisotropic background metrics \cite{GLX17}, the study of the magnetic problem and lower order perturbations \cite{C20, L20,L20a, CLR20, BGU21}, stability in the presence of apriori information \cite{RS19} and the presence of Liouville type transforms in these settings \cite{C20a}. 
The study of higher order analogues of these problems and parabolic settings was initiated in \cite{GFR19, C20b, CMRU20, LLR20}.

Compared to the local ``classical'' Calder\'on problem \cite{U09}, the fractional problem displays rather striking uniqueness, stability and reconstruction properties mirroring the strong rigidity and flexibility properties of the fractional Laplacian \cite{S17,Rue18}. Formally, this is indicated by the strong overdeterminedness of the problem.
Crucial inputs in the above results on the inverse problems consisted of the unique continuation property of the fractional Laplacian \cite{FF14, Seo, Rue15, Y17,GFR19} and the dual Runge approximation properties for fractional Laplacian \cite{DSV14, GSU16, RS17, Rue19}. Subsequently, some of these properties have been extended to larger classes of nonlocal operators
\cite{DSV16, RS17b, GFR20}. Moreover, connections between local and nonlocal Calder\'on type problems have been established in \cite{CR20}. A nonlocal problem for the fractional Laplacian with a lesser degree of overdeterminedness was recently introduced and studied in \cite{G20} where the author makes use of the theory developed in \cite{G15}.

As initiated in \cite{GFR20}, it is the purpose of this article to transfer some of the results for the specific problem of the fractional Laplacian to more general nonlocal elliptic operators and to extract the decisive nonlocal features entering in this discussion, such as the operators' antilocality. In particular, our study of the operators as in \eqref{eq:L} satisfying \textnormal{(\hyperref[ass:A1]{A1})}-\textnormal{(\hyperref[ass:A3]{A3})} illustrates the relevance of the notion of antilocality and indicates that geometric conditions may enter for these only directionally antilocal operators.

\subsection{Outline of the article}
The remainder of the article is structured as follows: In Section ~\ref{sec:pre}, we begin by collecting notation and discussing some auxiliary results on some of the function spaces which we will be considering in the article. Building on this, in Section ~\ref{sec:direct} we discuss the well-posedness properties of the direct problem, emphasizing the role of the geometry of the operator only ``seeing" certain regions. Here we focus on the symmetric operators from \textnormal{(\hyperref[ass:A1]{A1})}-\textnormal{(\hyperref[ass:A3]{A3})}. Given these results, in Section \ref{sec:inv_first} we study the antilocality properties from exterior conical domains of these operators and, exploiting these, provide the results on the general inverse problems. In Sections ~\ref{sec:exs1} and ~\ref{sec:exs2}  we complement this by discussing further examples, including the ones from above, and point out additional features which may arise due to the lack of symmetry. The appendices contain various symbol computations, the investigation of an alternative bilinear form for the operators from Section \ref{sec:exs1} and some geometric facts.

\subsection*{Acknowledgements}
Giovanni Covi was supported by an Alexander-von-Humboldt postdoctoral fellowship. María Ángeles García-Ferrero was supported  by the Spanish MINECO through Juan de la Cierva fellowship FJC2019-039681-I, by the Spanish State Research Agency through BCAM Severo Ochoa excellence accreditation SEV-2017-0718 and by the Basque Government through the  BERC Programme 2018-2021.
Angkana Rüland was supported by the Deutsche Forschungsgemeinschaft (DFG, German Research Foundation) under Germany’s Excellence Strategy EXC-2181/1 - 390900948 (the Heidelberg STRUCTURES Cluster of Excellence).

\section{Preliminaries}
\label{sec:pre}

\subsection{Notation}

We denote by $\Gamma$ a  convex, open, non-empty cone in $\R^n$, i.e. $\Gamma\subset \R^n\backslash\{0\}$ is  a convex set such that $t x\in\Gamma$ for all $x\in\Gamma$ and $t>0$. Further, $\mathcal C$ also denotes a convex, open, non-empty cone in $\R^n$, however, we reserve this notation to settings when we consider  operators  with symmetric kernels supported in the two-sided cone $-\mathcal C\cup\mathcal C$. In that case, we define for a set $\Omega \subset \R^n$ the domain of dependence (relative to an operator whose kernel is supported in the two-sided cone $-\mathcal{C}\cup\mathcal{C}$) as the two-sided cone
\begin{align*}
    \mathcal C(\Omega)&:=\Omega+(-\mathcal C\cup\mathcal C)=\{x+t\theta\in \R^n: \ x\in \Omega,\ t>0,\  \theta \in(-\mathcal C\cup\mathcal C)\cap\Ss^{n-1}\}.
\end{align*}

For any open set $\Omega\subset \R^n$, we denote the interior of its complement by $ \Omega_e :=\R^n\backslash \overline{\Omega}.$ 
In addition, given any function $u:\Omega\to\R$, we define its extension by zero by
\begin{align}\label{eq:extension}
    \mathcal E_{\Omega} u:=\begin{cases}
    u & \mbox{ in }\ \Omega,\\
    0 & \mbox{ in }\ \Omega_e.
    \end{cases}
\end{align}
The notation $u|_\Omega$ denotes the restriction of a function to $\Omega$, but sometimes -- if no confusion may arise -- it will also be interpreted as a global solution  which is zero outside $\overline{\Omega}$.

In one dimension we further use the notation
\begin{align*}
\R_+:=\{x\in \R: \ x>0\} \; \mbox{ and }\; \R_-:=\{x\in \R: \ x<0\}
\end{align*}
for the right and left half lines.
Moreover, in the following we say that an open set $\Omega \subset \R^n$ is \emph{a differentiable domain} if its boundary can locally be written as a differentiable (not necessarily continuously differentiable) function.

Finally, the Fourier transform is denoted by
\begin{align*}
\hat u(\xi) =\mathcal F u(\xi)=\int_{\R^n} u(x)e^{-ix\cdot\xi}dx.
\end{align*}

\subsection{Function spaces}\label{sec:spaces}

In the next sections, we will discuss the direct and inverse problems associated with the operators described in the introduction in various function spaces. While ``standard" Sobolev spaces and their symmetry properties in $\Omega$ and $\mathcal{C}(\Omega)\backslash \overline{\Omega}$ prove to be rather convenient for the inverse problem, the direct problem can be formulated at lower regularity in less symmetric spaces which take into account the geometry of the operators and their nonlocal character (see \cite{FKV15}). In the sequel, we recall and define both classes of function spaces and deduce a number of auxiliary properties which will be heavily exploited in the following sections.

\subsubsection{Sobolev spaces}
\label{sec:Sob}

First, we recall the definitions of the Sobolev spaces which are relevant for us. For $s\in\R$, the whole space Sobolev spaces are denoted by 
\begin{align*}
H^s(\R^n):=\big\{u\in \mathcal{S'}(\R^n): \|(1+|\cdot|^2)^{\frac s 2}\mathcal F u \|_{L^2(\R^n)} < \infty\big\},
\end{align*}
and their homogeneous versions by 
\begin{align*}
\dot H^s(\R^n):=\big\{u\in \mathcal{S'}(\R^n): \||\cdot|^s\mathcal F u \|_{L^2(\R^n)} < \infty\big\}.
\end{align*}
Associated with them, we further define
\begin{align*}
    \|u\|_{H^s(\R^n)}&:=\|(1+|\cdot|^2)^{\frac s 2}\mathcal F u \|_{L^2(\R^n)},\\ 
    \|u\|_{\dot H^s(\R^n)}&:=\||\cdot|^{s}\mathcal F u \|_{L^2(\R^n)}.
\end{align*}

Given an open set $\Omega\subset \R^n$, we define
\begin{align*}
H^s(\Omega)&:=\big\{u|_\Omega: u\in H^s(\R^n)\big\}, \mbox{ equipped with the quotient topology},\\
\widetilde H^s(\Omega)&:=\mbox{ closure of } C^\infty_c(\Omega) \mbox{ in } H^s(\R^n),\\
H^s_{\overline{\Omega}}&:=\{u\in H^s(\R^n): \supp{(u)}\subseteq \overline{\Omega}\}.
\end{align*}
If $\Omega$ is an open, bounded Lipschitz domain, the following identifications hold:
\begin{align}\label{eq:Hsid}
\begin{split}
    \big(H^s(\Omega)\big)^*=\widetilde H^{-s}(\Omega), \ \ \big(\widetilde H^s(\Omega)\big)^*= H^{-s}(\Omega),& \quad s\in\R,
    \\
    H^s_{\overline{\Omega}}=\widetilde H^s(\Omega),& \quad s\in\R,
    \\
    \widetilde H^s(\Omega)=H^s(\Omega),& \quad s\in \Big(0,\frac 12\Big),\\
    \widetilde H^s(\Omega)=\big\{u\in H^s(\Omega): u|_{\p\Omega}=0\big\},& \quad s\in\Big(\frac 12,1\Big).
\end{split}
\end{align}
These identifications also hold in $\Omega+\mathcal C$, $\mathcal C(\Omega)$ and their open complements, with $\Omega$ a bounded Lipschitz domain and $\mathcal C$ an open convex cone (see Lemma ~\ref{lem:HsLip}). 

We use $\langle \cdot, \cdot \rangle$ to denote the corresponding duality pairings.

\subsubsection{Asymmetric Sobolev type spaces}
\label{sec:Vsa}

Next, we introduce further spaces, which are modelled on the spaces from \cite{FKV15} and which are tailored to the exterior boundary value problem \eqref{eq:Dirichlet1}. In particular, they allow us to deal with this problem at rather low regularity in the exterior domain.
Let $s\in (0,1)$, assume that $a:\Ss^{n-1}\to [0,\infty)$ satisfies the assumptions \textnormal{(\hyperref[ass:A1]{A1})}-\textnormal{(\hyperref[ass:A3]{A3})}, and let $\Omega \subset\R^n$ be  an open set.
We define 
\begin{align*}
V^s(\Omega, a)&:=\Big \{u: \mathcal C(\Omega) \rightarrow \R: \ u|_{\Omega} \in L^2(\Omega), \ 
 \big(u(x)-u(y)\big)\frac{a^{\frac{1}{2}}\big(\frac{x-y}{|x-y|}\big)}{|x-y|^{\frac{n}{2}+s}} \in L^2\big(\Omega \times \mathcal C(\Omega)\big)\Big\},
\end{align*}
endowed with the norm
\begin{align*}
    \|u\|_{V^s(\Omega,a)}^2:=\|u\|_{L^2(\Omega)}^2+[u,u]_{V^s(\Omega, a)},
\end{align*}
where
\begin{align*}
[u,v]_{V^s(\Omega, a)}:= \int\limits_{\Omega}\int\limits_{\mathcal C(\Omega)} \big(u(x)-u(y)\big)\big(v(x)-v(y)\big) \frac{a\big(\frac{x-y}{|x-y|}\big)}{|x-y|^{n+2s}} dx dy.
\end{align*}

Assumption ~\textnormal{(\hyperref[ass:A1]{A1})} implies the following statements for $s\in (0,1)$ (see Lemmas ~\ref{lem:HVOm} and ~\ref{lem:HVR} in the following subsection for their proofs)
\begin{align}
    H^s(\mathcal C(\Omega))&\subseteq V^s(\Omega, a), \label{eq:incl}\\
   H^s(\R^n)&= V^s(\R^n, a). \label{eq:equiv}
\end{align}
We emphasize that, in general, the first inclusion is strict, due to the asymmetric definition of the spaces $V^s(\Omega, a)$ with regards to $\Omega$ and $\mathcal{C}(\Omega)\backslash \overline{\Omega}$: While the space $V^s(\Omega,a)$ imposes a Sobolev regularity type control in $\Omega$, it only provides very weak regularity conditions on $\mathcal{C}(\Omega)\backslash \overline{\Omega}$.

In addition, if $\Omega$ is bounded, the following Poincar\'e inequality holds for  any $u\in \widetilde H^s(\Omega)$:
\begin{align}
\label{eq:Poincare}
    \|u\|_{L^2(\Omega)}^2\leq C [u,u]_{V^s(\R^n,a)}=C\|L(D)^{\frac 1 2}u\|_{L^2(\R^n)}^2.
\end{align}

Finally, we introduce the following space, in which the exterior data in \eqref{eq:Dirichlet2} will be considered:
\begin{align}
\label{eq:ext_data}
\begin{split}
    V^s_e(\Omega,a)&:=\{u|_{\mathcal C(\Omega)\backslash \overline{\Omega}}: u\in V^s(\Omega,a)\}, \mbox{ equipped with the quotient topology}.
\end{split}
\end{align}

\subsubsection{Auxiliary results on the function spaces}

In this final subsection we prove some auxiliary results on the relationship between the spaces defined in Sections ~\ref{sec:Sob} and ~\ref{sec:Vsa}. They will play an important role in our definition of weak solutions in Section ~\ref{sec:direct}.

We begin by proving the validity of \eqref{eq:Hsid} for a slightly more general class of domains than the one considered in \cite[Defintion 3.28]{McLean}. We emphasize that we only drop the condition that the boundary $\p\Omega$ is compact. This implies that the boundary can be unbounded, but we assume that the unbounded parts must be given by a union of finitely many Lipschitz graphs (up to rigid motions). Then, we still have a finite covering of $\p\Omega$ by (up to rigid motion) Lipschitz graphs.

\begin{lem}\label{lem:HsLip}
Let $\Omega\subset \R^n$ be an open Lipschitz set such that there exist finite families $\{W_j\}_{j=1}^N$, $\{\Omega_j\}_{j=1}^N$ satisfying that  $W_j, \Omega_j \subset \R^n$ and
\begin{itemize}
    \item[(i)] $\p\Omega\subset\bigcup_{j=1}^N W_j$,
    \item[(ii)] $\Omega_j$ is a Lipschitz hypograph (up to a rigid motion) for any  $j\in\{1,\dots,N\}$,
    \item[(iii)] $W_j\cap\Omega=W_j\cap \Omega_j$ for any  $j\in\{1,\dots,N\}$.
\end{itemize}
Then the identifications \eqref{eq:Hsid} hold.
\end{lem}

\begin{proof}
The proof of each identification does not differ from  those in \cite[Chapter 3]{McLean}, where it is additionally assumed  that  $\p\Omega$ is compact.
Indeed, we can find $\Omega_0 \Subset \Omega$ such that $\overline{\Omega}\subset \bigcup_{j=0}^N \Omega_{j}$ and consider a partition of unity for $\overline{\Omega}$ subordinated to $\{\Omega_j\}_{j=0}^N$ (see for instance the proof of \cite[Theorem 3.29]{McLean}).
After this reduction, the proof only exploits the same ingredients as in \cite{McLean}, including \eqref{eq:Hsid}  for Lipschitz hypographs.
\end{proof}

We next deduce an extension result which will ensure that for our notion of weak solution (see Definition ~\ref{def:weaksol1}) to the inhomogeneous interior problem \eqref{eq:inhom_int} it is possible to extend a solution canonically from a function on $\mathcal{C}(\Omega)$ to a function on $\R^n$:

\begin{lem}\label{lem:zeroext}
Let $s\in (0,1)$, let $\mathcal C\subset \R^n\backslash\{0\}$ be an open, non-empty, convex cone and let $\Omega\subset \R^n$ be an open,  bounded  Lipschitz domain. Then 
\begin{align*}
   \widetilde H^s(\Omega)=\big\{\mathcal E_{\mathcal C(\Omega)} u: \ u\in H^s(\mathcal C(\Omega)),\; u=0 \mbox{ in } \mathcal C(\Omega)\backslash\overline{\Omega}\big\},
\end{align*}
with $\mathcal E_{\mathcal C(\Omega)}$  as in \eqref{eq:extension}.
\end{lem}

\begin{proof}
The inclusion $\subset$ is immediate by definition. 
The inclusion $\supset$ follows easily in the case  $\Omega\Subset \mathcal C(\Omega)$, which holds for differentiable domains (see Lemma ~\ref{lem:C1dom}). Nevertheless, the inclusion $\Omega\Subset \mathcal C(\Omega)$   may fail for some Lipschitz but not differentiable  domains (e.g. when considering certain triangles, see  the right of Figure ~\ref{figure:extensionlemma}). 
In general, the inclusion $\supset$ can be proved as follows:

Let  $\tilde u\in H^s(\R^n)$  with  $\tilde u|_{\mathcal C(\Omega)}=u$. This means that $\tilde u=0$ in $\mathcal C(\Omega)\backslash\overline{\Omega}$. Let $R>0$ be such that $\Omega\subset  B_R$  and let  $\eta$ be a smooth cut-off function supported in $B_{2R}$ such that $\eta=1$ in $B_R$. 
Since $\eta \tilde u\in H^s(\R^n)$ and $\supp (\eta \tilde u)\subset \overline U$ with $U=B_{2R}\cap\big(\Omega\cup \mathcal C(\Omega)_e\big)$, then $\eta \tilde u\in H^s_{\overline U}$. 
Furthermore, since  $U$ is a  bounded Lipschitz subset (see Figure ~\ref{figure:extensionlemma} and the remarks from Section ~\ref{sec:Sob}),  
$\eta \tilde u\in \widetilde H^s(U)$.
By definition,  for any $\epsilon>0$ there is $g\in C^\infty_c(U)$ such that $\|\eta \tilde u-g\|_{H^s(\R^n)}<\epsilon$. 
Since $\mathcal H^{n-1}\big(\overline{\Omega}\cap (\R^n\backslash\mathcal C(\Omega))\big)=0$ (see Lemma ~\ref{lem:Lipdom}), we have $\mathcal E_\Omega g\in C^\infty_c(\Omega)$ and 
$$\|\mathcal E_\Omega(\eta \tilde u)-\mathcal E_\Omega g\|_{H^s(\R^n)}<\epsilon.$$ 
Therefore, $\mathcal E_\Omega(\eta \tilde u)\in \widetilde H^s(\Omega)$.
The conclusion then follows by noticing that 
$\mathcal E_\Omega(\eta \tilde u)=\mathcal E_\Omega u=\mathcal E_{\mathcal C(\Omega)}u$.
\end{proof}

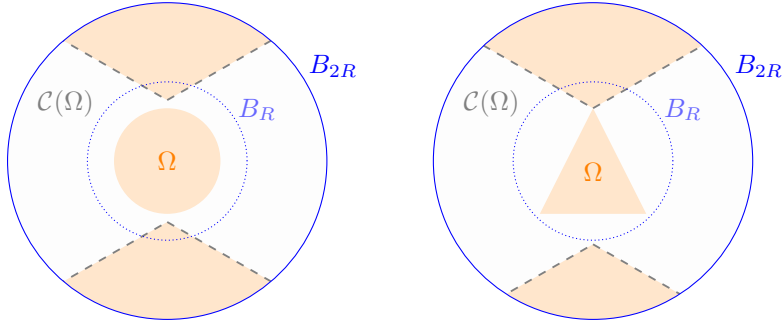
\begin{figure}[t]
\begin{tikzpicture}

\pgfmathsetmacro{\anglesym}{30}
\pgfmathsetmacro{\angleone}{-\anglesym}
\pgfmathsetmacro{\angletwo}{\anglesym}
\pgfmathsetmacro{\angletwox}{cos(\angletwo)}
\pgfmathsetmacro{\angletwoy}{sin(\angletwo)}

\begin{scope}[scale=.7]

\pgfmathsetmacro{\r}{1}
\pgfmathsetmacro{\R}{1.5}

\pgfmathsetmacro{\xx}{6}
\pgfmathsetmacro{\int}{(-\r*sin(\angletwo+90)+\r*sin(\angleone+90)+sin(\angleone)/cos(\angleone)*\r*(cos(\angletwo+90)-cos(\angleone+90)))/(sin(\angletwo)-sin(\angleone)*cos(\angletwo)/cos(\angleone))}

\pgfmathsetmacro{\sep}{\r/cos(\anglesym)}

\begin{scope}
\clip (0,0) circle (2*\R);

\fill[orange!20] (0,0) circle (2*\R);

\pgfmathsetmacro{\Wx}{1.7}
\pgfmathsetmacro{\Wy}{0}
\pgfmathsetmacro{\rr}{.5}

\fill[black!1]
({\r*cos(\angletwo+90)+\xx*cos(\angletwo)},{\r*sin(\angletwo+90)+\xx * sin(\angletwo)}) 
--({\r*cos(\angletwo+90)+\int*cos(\angletwo)},{\r*sin(\angletwo+90)+\int * sin(\angletwo)})
 --({\r*cos(\angleone+90)-\xx*cos(\angleone)},{\r*sin(\angleone+90)-\xx * sin(\angleone)})
 .. controls (-6.7,\Wy) ..({\r*cos(\angletwo-90)-\xx*cos(\angletwo)},{\r*sin(\angletwo-90)-\xx * sin(\angletwo)})
--({\r*cos(\angletwo-90)-\int*cos(\angletwo)},{\r*sin(\angletwo-90)-\int * sin(\angletwo)})
--({\r*cos(\angleone-90)+\xx*cos(\angleone)},{\r*sin(\angleone-90)+\xx * sin(\angleone)}) 
.. controls (6.7,\Wy) ..
({\r*cos(\angletwo+90)+\xx*cos(\angletwo)},{\r*sin(\angletwo+90)+\xx * sin(\angletwo)});

\draw[black!50,dashed,line width=.7pt, domain=-0:\xx] plot ({\x*cos(\angletwo)},{\sep+\x * sin(\angletwo)});

\draw[black!50,dashed,line width=.7pt, domain=-\xx:0] plot ({\x*cos(\angleone)},{\sep+\x * sin(\angleone)});

\draw[black!50,dashed,line width=.7pt, domain=-\xx:0] plot ({\x*cos(\angletwo)},{-\sep+\x * sin(\angletwo)});

\draw[black!50,dashed,line width=.7pt, domain=-0:\xx] plot ({\x*cos(\angleone)},{-\sep+\x * sin(\angleone)});

\fill[orange!20] (0,0) circle [radius=\r,];
\node[orange] at (0,0) {$\Omega$};

\node[black!50] at ({-2.2*(\r)*\angletwox}, {2.2*(\r)*\angletwoy}) {$\mathcal C(\Omega)$};

\draw[blue, densely dotted](0,0) circle (\R);
\node[blue!60] at ({1.3*(\R)*\angletwox}, {1.3*(\R)*\angletwoy}) {$B_{R}$};

\end{scope}

\draw[blue](0,0) circle (2*\R);
\node[blue] at ({2.4*(\R)*\angletwox}, {2.4*(\R)*\angletwoy}) {$B_{2R}$};
\end{scope}

\begin{scope}[scale=.7, xshift=8cm]

\pgfmathsetmacro{\r}{1}
\pgfmathsetmacro{\R}{1.5}

\pgfmathsetmacro{\xx}{6}
\pgfmathsetmacro{\int}{(-\r*sin(\angletwo+90)+\r*sin(\angleone+90)+sin(\angleone)/cos(\angleone)*\r*(cos(\angletwo+90)-cos(\angleone+90)))/(sin(\angletwo)-sin(\angleone)*cos(\angletwo)/cos(\angleone))}

\pgfmathsetmacro{\sep}{1+1*tan(\anglesym)}

\begin{scope}
\clip (0,0) circle (2*\R);

\fill[orange!20] (0,0) circle (2*\R);

\pgfmathsetmacro{\Wx}{1.7}
\pgfmathsetmacro{\Wy}{0}
\pgfmathsetmacro{\rr}{.5}

\fill[black!1]
({\xx*cos(\angletwo)},{1+\xx * sin(\angletwo)}) 
--(0,1)
 --({-\xx*cos(\angleone)},{1-\xx * sin(\angleone)})
 .. controls (-6.7,\Wy) ..({-\xx*cos(\angletwo)},{-\sep-\xx * sin(\angletwo)})
--(0,-\sep)
--({\xx*cos(\angleone)},{-\sep+\xx * sin(\angleone)}) 
.. controls (6.7,\Wy) ..
({\r*cos(\angletwo+90)+\xx*cos(\angletwo)},{\r*sin(\angletwo+90)+\xx * sin(\angletwo)});

\draw[black!50,dashed,line width=.7pt, domain=-0:\xx] plot ({\x*cos(\angletwo)},{1+\x * sin(\angletwo)});

\draw[black!50,dashed,line width=.7pt, domain=-\xx:0] plot ({\x*cos(\angleone)},{1+\x * sin(\angleone)});

\draw[black!50,dashed,line width=.7pt, domain=-\xx:0] plot ({\x*cos(\angletwo)},{-\sep+\x * sin(\angletwo)});

\draw[black!50,dashed,line width=.7pt, domain=-0:\xx] plot ({\x*cos(\angleone)},{-\sep+\x * sin(\angleone)});

\fill[orange!20] (0,1)--(-1,-1)--(1,-1) -- cycle;
\node[orange] at (0,-.2) {$\Omega$};

\node[black!50] at ({-2.2*(\r)*\angletwox}, {2.2*(\r)*\angletwoy}) {$\mathcal C(\Omega)$};

\draw[blue, densely dotted](0,0) circle (\R);
\node[blue!60] at ({1.3*(\R)*\angletwox}, {1.3*(\R)*\angletwoy}) {$B_{R}$};

\end{scope}

\draw[blue](0,0) circle (2*\R);
\node[blue] at ({2.4*(\R)*\angletwox}, {2.4*(\R)*\angletwoy}) {$B_{2R}$};
\end{scope}

\end{tikzpicture}
\caption{Illustration of the subset $U=B_{2R}\cap\big(\Omega\cup\mathcal C (\Omega)_e\big)$ (in orange) in Lemma ~\ref{lem:zeroext}. In the first case $\Omega\Subset\mathcal C(\Omega)$, so the inclusion $\supset$ follows immediately. In the second case, the slightly more general argument from the proof of Lemma ~\ref{lem:zeroext} is necessary.}
\label{figure:extensionlemma}
\end{figure}

Next we provide the proof for the identity \eqref{eq:equiv}:

\begin{lem}\label{lem:HVR} Let $s\in(0,1)$ and assume that the function $a:\Ss^{n-1}\rightarrow \R$ satisfies the hypotheses \textnormal{(\hyperref[ass:A1]{A1})}-\textnormal{(\hyperref[ass:A3]{A3})}. Then
\begin{align*}
     H^s(\R^n)&= V^s(\R^n,a).
\end{align*}
In particular, for any $u\in H^s(\R^n)$
\begin{align}\label{eq:HVR}
    [u,u]_{V^s(\R^n, a)}=2\|L(D)^{\frac 1 2} u\|_{L^2(\R^n)}^2\leq C\|u\|_{\dot H^s(\R^n)}^2.
\end{align}
\end{lem}

\begin{proof}
Both $H^s(\mathbb R^n)$ and $V^s(\mathbb R^n,a)$ are contained in $L^2(\mathbb R^n)$. Thus it suffices to show that the norms of $H^s(\mathbb R^n)$ and $V^s(\mathbb R^n,a)$ are equivalent. By \eqref{eq:bilR} and its Fourier transform (see Lemma ~\ref{lem:bilR_Fourier} and Lemma ~\ref{lem:symbol}), it follows that
\begin{align*}
    [u,u]_{V^s(\R^n, a)}&=2\tilde B_0(u,u)=2\|L(D)^{\frac 1 2}u\|_{L^2(\R^n)}^2.
\end{align*}
In addition, \textnormal{(\hyperref[ass:A1]{A1})} implies that there is $\lambda\in(1, \infty)$ such that 
\begin{align*}
    \lambda^{-1}\leq L(\omega)\leq \lambda, \quad \mbox{for all } \omega\in\Ss^{n-1}.
\end{align*}
This means that
\begin{align*}
    \lambda^{-1}|\xi|^{2s}\leq L(\xi)\leq \lambda|\xi|^{2s}, \quad \mbox{for all } \xi\in\R^n,
\end{align*}
and therefore 
\begin{align*}
    \lambda^{-1} \|u\|_{\dot H^s(\R^n)}^2\leq \frac 12 [u,u]_{V^s(\R^n, a)}\leq  \lambda\|u\|_{\dot H^s(\R^n)}^2.
\end{align*}
\end{proof}

Further, we present the argument for \eqref{eq:incl}:

\begin{lem}\label{lem:HVOm} Let $s\in(0,1)$,  $\Omega\subset\mathbb R^n$ be an open set, and assume that the function $a:\Ss^{n-1}\rightarrow \R$  satisfies the assumptions \textnormal{(\hyperref[ass:A1]{A1})}-\textnormal{(\hyperref[ass:A3]{A3})}. Then
\begin{align*}
    H^s(\mathcal C(\Omega))&\subseteq V^s(\Omega, a).
\end{align*}
\end{lem}

\begin{proof}
Let  $u\in H^s(\mathcal C(\Omega))$ and $\tilde u\in H^s(\R^n)$ be such that $\tilde u|_{\mathcal C(\Omega)}=u$. 
Then, by Lemma ~\ref{lem:HVR}, 
\begin{align*}
    [u,u]_{V^s(\Omega,  a)}\leq [\tilde u, \tilde u]_{V^s(\R^n, a)}\leq C\|\tilde u\|_{\dot H^s(\R^n)}^2.
\end{align*}
Taking the infimum among all possible  $\tilde u$ satisfying the above properties, it holds
\begin{align}\label{eq:HVOm}
   [u,u]_{V^s(\Omega,  a)}\leq C\|u\|_{H^s(\mathcal{C}(\Omega))}^2
\end{align}
 and therefore the claimed inclusion follows.
\end{proof}

\subsection{The bilinear forms}
\label{sec:bilin}

A common theme in our discussion of the direct problem, which has been also already indicated above, is the choice of a notion of a solution to the nonlocal equation \eqref{eq:Dirichlet2}. This can be viewed as the choice of a suitable Dirichlet form. In the previous sections, we considered two natural function spaces associated with the operator $L$ which are closely related to the two bilinear forms \eqref{eq:bilR}, \eqref{eq:bil} from the introduction. In our discussion of the direct problem, we will observe that these two bilinear forms give rise to the same solution in $\Omega$ if suitable boundary conditions are imposed. However, in general, the bilinear forms do not agree on arbitrary function spaces and lead to different Dirichlet-to-Neumann maps (see, in particular, the discussion in Section ~\ref{sec:DtN_symm}) and thus to different measurements for the associated inverse problems. In many places of the article we will thus develop our results for $B_q$ and $\tilde{B}_q$ in parallel.
Throughout  this section, we assume $\Omega\subset \R^n$ is a bounded, open set. 

We begin by discussing the boundedness properties of the bilinear forms. To this end, we first note that the bilinear form \eqref{eq:bilR} is globally related to the operator $L(D)$ through a Fourier characterization:

\begin{lem}
\label{lem:bilR_Fourier}
Let $s\in (0,1)$, let $\tilde{B}_q$ be as in \eqref{eq:bilR} with the function $a:\Ss^{n-1}\rightarrow \R$  satisfying the assumptions \textnormal{(\hyperref[ass:A1]{A1})}-\textnormal{(\hyperref[ass:A3]{A3})} and let $q\in L^\infty(\Omega)$. Then for $u,v\in H^{s}(\R^n)$ it holds that
\begin{align}
\label{eq:bil_Fourier}
\tilde B_q(u,v)
&=\int\limits_{\R^n} \big(L(D)^{\frac 12} u \big)\big(L(D)^{\frac 12} v\big) dx + (qu,v)_{L^2(\Omega)}.
\end{align}
In particular, 
\begin{align*}
    |\tilde{B}_q(u, v)|
     &\leq C\|u\|_{H^s(\R^n)}\|v\|_{H^s(\R^n)},
\end{align*}
where $C>0$ depends on $\|q\|_{L^{\infty}(\Omega)}$ and $a$.
\end{lem}

\begin{proof}
The identity \eqref{eq:bil_Fourier} follows from Fourier-transforming the bilinear form \eqref{eq:bilR} (see also Lemma ~\ref{lem:symbol} for the computation of the symbol). Indeed, by Plancherel's theorem, for $u, v\in C^\infty_c(\R^n)$  
\begin{align*}
    \tilde B_0(u,v)
    &=\big(L(D)u, v\big)_{L^2(\R^n)}
    =\big(L(\xi)\hat u, \hat v\big)_{L^2(\R^n)}
    \\&=\big(L(\xi)^{\frac12}\hat u, L(\xi)^{\frac12}\hat v\big)_{L^2(\R^n)}
    =\big(L(D)^{\frac12} u, L(D)^{\frac12} v\big)_{L^2(\R^n)}.
\end{align*}

Thus, the claimed estimate is a direct consequence of the representation \eqref{eq:bil_Fourier}, an application of the Cauchy-Schwarz inequality and the estimate \eqref{eq:HVR}.
\end{proof}

While the boundedness of the bilinear form $\tilde{B}_q$ in Sobolev spaces directly follows from its Fourier characterization, the corresponding bound for the bilinear form $B_q$ which we will use for the analysis of the Dirichlet problem \eqref{eq:Dirichlet2} requires slightly more care. 

\begin{lem}\label{lem:bilcont}
Let $s\in (0,1)$, let  $B_q$ be as in \eqref{eq:bil}  with the function $a:\Ss^{n-1}\rightarrow \R$ satisfying the assumptions \textnormal{(\hyperref[ass:A1]{A1})}-\textnormal{(\hyperref[ass:A3]{A3})} and let $q\in L^\infty(\Omega)$. Then for any $u,v\in H^s(\mathcal C(\Omega))$
\begin{align*}
    |B_q(u, v)|
     &\leq C\|u\|_{H^s(\mathcal C(\Omega))}\|v\|_{H^s(\mathcal C(\Omega))},
\end{align*}
where the constant $C>0$ depends on $\|q\|_{L^{\infty}(\Omega)}$ and $a$.
In addition, for any $u\in V^s(\Omega,a)$ and $v\in \widetilde H^s(\Omega)$
\begin{align*}
    |B_q(u, v)|
     &\leq C\|u\|_{V^s(\Omega, a)}\|v\|_{H^s(\R^n)}.
\end{align*}
\end{lem}

\begin{proof}
Let $u, v \in H^s(\mathcal C(\Omega))$  and let  $\tilde u, \tilde v \in H^s(\R^n)$ be such that $\tilde u|_{\mathcal C(\Omega)}=u$, $\tilde v|_{\mathcal C(\Omega)}=v$.
By the H\"older inequality
\begin{align*}
    |2 B_0(u, v)|
    &\leq  \left(\int_{\mathcal C(\Omega)}\!\int_{\mathcal C(\Omega)}\!\!\! \big(u(x)\!-\!u(y)\big)^2 k_a(x\!-\!y)dx dy\!\right)^{\frac 12}\!\!
    \left(\int_{\mathcal C(\Omega)}\!\int_{\mathcal C(\Omega)} \!\!\!\big(v(x)\!-\!v(y)\big)^2 k_a(x\!-\!y)dx dy\!\right)^{\frac 12}
    \\& \leq \left(\int_{\R^n}\int_{\R^n} \big(\tilde u(x)\!-\!\tilde u(y)\big)^2 k_a(x\!-\!y)dx dy\!\right)^{\frac 12}\!\!
    \left(\int_{\R^n}\int_{\R^n}\big(\tilde v(x)\!-\!\tilde v(y)\big)^2 k_a(x\!-\!y)dx dy\!\right)^{\frac 12}\\
    &\leq [\tilde u,\tilde u]_{V^s(\R^n,a)}^{\frac 1 2}[\tilde v,\tilde v]_{V^s(\R^n,a)}^{\frac 12},
\end{align*}
where
\begin{align}
\label{eq:kernel_a}
   k_a(z) :=\frac{a\big(\frac{z}{|z|}\big)}{|z|^{n+2s}}.
\end{align}
Applying  \eqref{eq:HVR}, we infer
\begin{align*}
    |B_0(u, v)|
    &\leq C\|\tilde u\|_{H^s(\R^n)}\|\tilde v\|_{H^s(\R^n)}.
\end{align*}
Taking the infimum among all possible $\tilde u, \tilde v$ and including the zeroth order contribution originating from the potential, we hence arrive at
\begin{align*}
    |B_q(u, v)|
    &\leq C\|u\|_{H^s(\mathcal C(\Omega))}\|v\|_{H^s(\mathcal C(\Omega))} +\|q\|_{L^\infty(\Omega)}\|u\|_{L^2(\Omega)}\|v\|_{L^2(\Omega)}
    \\ &\leq C\|u\|_{H^s(\mathcal C(\Omega))}\|v\|_{H^s(\mathcal C(\Omega))}.
\end{align*}

In order to prove the second claim, let  $u\in V^s(\Omega,a)$ and $v\in \widetilde H^s(\Omega)$. Splitting the integral and taking into account the support of $v$ in $\Omega$, we can write 
\begin{align*}
    2B_0(u,v)=I_1+I_2,
\end{align*}
where
\begin{align*}
    I_1
    &:=\int_\Omega\int_{\Omega}\big(u(x)-u(y)\big)\big(v(x)-v(y)\big)k_a(x-y)dxdy,
    \\
    I_2
    &:=\int_\Omega\int_{\mathcal C(\Omega)\backslash\overline{\Omega}}\big(u(x)-u(y)\big)\big(v(x)-v(y)\big)k_a(x-y)dxdy
    \\&\qquad+\int_{\mathcal C(\Omega)\backslash\overline{\Omega}}\int_{\mathcal C(\Omega)}\big(u(x)-u(y)\big)\big(v(x)-v(y)\big)k_a(x-y)dxdy
    \\&\ =2\int_\Omega\int_{\mathcal C(\Omega)\backslash\overline{\Omega}}\big(u(x)-u(y)\big)\big(v(x)-v(y)\big)k_a(x-y)dxdy.
\end{align*}
In the last step we used the symmetry of the kernel $k_a$.
Therefore, by \eqref{eq:HVR}, for $j\in\{1,2\}$
\begin{align*}
   |I_j|
   &\leq 2[u,u]_{V^s(\Omega,a)}^{\frac 1 2}[v,v]_{V^s(\Omega,a)}^{\frac 1 2}
   \leq 2 [u,u]_{V^s(\Omega,a)}^{\frac 1 2}[v,v]_{V^s(\R^n,a)}^{\frac 1 2}
   \leq C[u ,u]_{V^s(\Omega,a)}^{\frac 12}\|v\|_{H^s(\R^n)}.
\end{align*}
Combining this with the contribution from the potential yields the desired estimate.
\end{proof}

Upon concluding this section, we present some first comparisons between the bilinear forms $B_q$ and $\tilde{B}_q$. 

\begin{lem}
\label{lem:comp}
Let $s\in (0,1)$ and let $\tilde B_q$ and ${B}_q$ be as in \eqref{eq:bilR} and in \eqref{eq:bil}, respectively. Let $u,v\in H^s(\R^n)$.
If in addition $u \in\widetilde H^s(\Omega)$ or $v \in\widetilde H^s(\Omega)$, it holds that $\tilde B_q(u,v)= B_q(u,v)$.
\end{lem}

\begin{proof}
By the symmetry of both $B_q$ and $\tilde{B}_q$, it is enough to prove that $B_q(u,v)=\tilde B_q(u,v)$ for any $u \in\widetilde H^s(\Omega)$ and  $v\in H^s(\R^n)$.
We start by noticing that if $u$ is supported in $\overline{\Omega}$, then 
\begin{align*}
    \big(u(x)-u(y)\big) a\left(\frac{x-y}{|x-y|}\right)\neq 0
\end{align*}
at most in $\mathcal C(\overline \Omega)\times \mathcal C(\overline \Omega)$. 
Therefore, we infer
\begin{align*}
    \tilde B_0(u,v)
    &=\frac12\int_{\R^n}\int_{\R^n}  \big(u(x)-u(y)\big) a\Big(\frac{x-y}{|x-y|}\Big) \frac{v(x)-v(y)}{|x-y|^{n+2s}} dydx\\
    &=\frac 12\int_{\mathcal C(\Omega)}\int_{\mathcal C(\Omega)}  \big(u(x)-u(y)\big) a\Big(\frac{x-y}{|x-y|}\Big) \frac{v(x)-v(y)}{|x-y|^{n+2s}} dydx=B_0(u,v).
\end{align*}
Since the potential term is the same for both forms, the proof is complete.
\end{proof}

\begin{rmk}
\label{rmk:inequal}
For a general function $u\in H^s(\R^n)$ the equality between the bilinear forms may fail. 
Indeed, let  $W\Subset \mathcal C(\Omega)\backslash\overline{\Omega}$ be bounded, open and such that $\mathcal C(W)\not \subset \mathcal C(\Omega)$ and consider $f\in \widetilde H^s(W)$. Then $\tilde B_q(f,f)\neq B_q(f,f)$.
In fact, arguing as in the proof of Lemma ~\ref{lem:comp},
\begin{align*}
    \tilde B_0(f,f)
    &=\frac12\int_{\mathcal C(W)}\int_{\mathcal C (W)} \big(f(x)-f(y)\big)^2\frac{a\big(\frac{x-y}{|x-y|}\big)}{|x-y|^{n+2s}}dydx,\\
    B_0(f,f)
    &=\frac12\int_{\mathcal C(\Omega)\cap\mathcal C(W)}\int_{\mathcal C(\Omega)\cap\mathcal C (W)} \big(f(x)-f(y)\big)^2\frac{a\big(\frac{x-y}{|x-y|}\big)}{|x-y|^{n+2s}}dydx.
\end{align*}
Since $\mathcal C(W)\not \subset \mathcal C(\Omega)$, the integrals do not agree. Indeed, taking into account the support of $f$ (and the fact that $W \subset \mathcal{C}(\Omega)$, which hence implies that $f(x)=0$ for $x\in \mathcal{C}(W)\backslash \mathcal{C}(\Omega)$) and the symmetry of the kernel, we can write the difference as follows 
\begin{align*}
    \tilde B_0(f,f)-B_0(f,f)
    & = \frac{1}{2} \int\limits_{\mathcal C(W)\backslash\mathcal C ( \Omega)} \int\limits_{\mathcal{C}(W)} f^2(y) \frac{a\big(\frac{x-y}{|x-y|}\big)}{|x-y|^{n+2s}}dydx\\
& \quad    + \frac{1}{2} \int\limits_{\mathcal{C}(\Omega)\cap\mathcal C(W)} \int\limits_{\mathcal{C}(W)\backslash \mathcal{C}(\Omega)} f^2(x) \frac{a\big(\frac{x-y}{|x-y|}\big)}{|x-y|^{n+2s}}dydx\\
    &=\int_{W}\int_{\mathcal C(W)\backslash\mathcal C (\overline \Omega)} f^2(x) \frac{a\big(\frac{x-y}{|x-y|}\big)}{|x-y|^{n+2s}}dydx >0,
\end{align*}
where the integral in the last line is finite since $\dist\big(W, \mathcal C(W)\backslash\mathcal C (\overline{\Omega})\big)>0$.  
\end{rmk}

Further comparison results between the two forms e.g. as acting on solutions or as inducing associated Dirichlet-to-Neumann maps will be discussed below.

\subsection{Antilocality in cones} 

We conclude the discussion of the preliminary results with an antilocality statement for more general function spaces under the assumption that the operator under consideration is antilocal in the corresponding cones (see Definition ~\ref{defi:anti}).

\begin{lem}\label{lem:antilocHr}
Let $\Gamma\subset \R^n\backslash\{0\}$ be a convex,  non-empty, open cone.
Let $r,s\in\R$ and let $L(D): H^r(\R^n)\to H^{r-2s}(\R^n)$ be a linear operator which is $\Gamma$-antilocal in the sense of Definition ~\ref{defi:anti}.  Let $f\in H^r(\R^n)$ and assume $f=0=L(D)f$ in  an open subset $U\subset \R^n$. Then $f=0$ in $U+\Gamma$.
\end{lem}

\begin{rmk}
\label{rmk:reg_moll}
We remark that by the linearity of the operators and the openness of the vanishing assumption, only extremely mild regularity conditions are required in Lemma ~\ref{lem:antilocHr}. We will implicitly make use of similar mollification arguments in various places of the article without explicit references to this.
\end{rmk}

\begin{proof}
In order to apply the $\Gamma$-antilocality of Definition ~\ref{defi:anti}, we need to approximate  $f\in \widetilde H^r(U)$ by smooth compactly supported functions.
Let  $\eta$ be a standard mollifier compactly supported in the unit ball and let $\eta_j(x):=j^n \eta (jx)$ for $j\in \N$. We define $f_j := \eta_j\ast f \in C^\infty_c(\mathbb R^n)$ and $ U_j:= \{x\in U : d(x,\partial U)>\frac 1 j\}$.
In addition, we notice that 
\begin{align*}
    L(D) f_j & = \mathcal{F}^{-1}\big(\hat L(\xi)f_j(\xi)\big)= \mathcal{F}^{-1}\big(L(\xi)\hat\eta_j(\xi)\hat f(\xi)\big)= \big(\eta_j \ast (L(D) f)\big). 
\end{align*}
Then,  $f=0=L(D)f$ in $U$ implies that $f_j=0=L(D)f_j$ in $U_j$. 
By the $\Gamma$-antilocality, since $f_j\in C^\infty_c(\R^n)$, we conclude that $f_j=0$ in $U_j+\Gamma$.

It is clear that for any $x\in U+\Gamma$ there exists $N\in\mathbb N$ so large that $x\in U_j+\Gamma$ for all $j\geq N$. Thus $f_j(x)=0$ for all $j\geq N$, which in turn implies $f(x)=0$.
This finally implies $f=0$ in $U+\Gamma$.
\end{proof}

\section{The Direct Problem}
\label{sec:direct}
This section is devoted to the study of the direct problem \eqref{eq:Dirichlet2} with $L(D)$  of the form \eqref{eq:L} and with $a$ satisfying the conditions \textnormal{(\hyperref[ass:A1]{A1})}-\textnormal{(\hyperref[ass:A3]{A3})}. On the one hand, we will discuss this in function spaces modelled on \cite{FKV15} which allow for very weak regularity conditions on the data. On the other hand, we will also consider it in Sobolev spaces. The latter allows us to provide a more ``symmetric" definition of the Dirichlet-to-Neumann operator (see Proposition ~\ref{prop:DtNtilde}).
From now on, we assume $\Omega\subset\R^n$ is an open, bounded Lipschitz domain and $s\in(0,1)$.

\subsection{Well-posedness result for the direct problem in Sobolev spaces}
\label{sec:wellpos}
We first prove solvability (outside of the spectrum) for the inhomogenous problem with interior source and zero exterior data:
\begin{align}
\label{eq:inhom_int}
\begin{split}
\big(L(D) + q\big) u & = g \; \mbox{ in } \Omega,\\
u & = 0 \;\mbox{ in } \mathcal{C}(\Omega)\backslash \overline{\Omega},
\end{split}
\end{align}
with $g\in \big(\widetilde H^s(\Omega)\big)^{\ast}=H^{-s}(\Omega)$. 
As discussed in Section ~\ref{sec:better}, the exterior condition is only prescribed in the domain of dependence of $L$.
Similarly, $u$ is defined as a function on $\mathcal C(\Omega)$ with suitable regularity conditions.
Nevertheless, by Lemma ~\ref{lem:zeroext} and by the imposed boundary conditions, no generality is lost if we consider $u$  defined in $\R^n$ and vanishing also outside $\overline{\mathcal C(\Omega)}$. 
In the sequel we will introduce and discuss a solution notion based on the ``more local" bilinear form $B_q$ in the spaces $V^s(\Omega,a)$ which are associated with this bilinear form. The corresponding results for $\tilde{B}_q$ can be deduced analogously; we comment on this in the context of Sobolev spaces in Lemma ~\ref{lem:idtwosols} below.

Relying on our above discussion, we first present the definition of a (weak) solution based on the bilinear form $B_q$:

\begin{defi}\label{def:weaksol1} 
Let $s\in (0,1)$ and let $B_q(\cdot,\cdot)$ be the bilinear form from \eqref{eq:bil}.
Given  $g\in H^{-s}(\Omega)$, a function $u\in \widetilde H^s(\Omega)$ is a \emph{(weak) solution to \eqref{eq:inhom_int} (based on  $B_q$)} if
\begin{align}
\label{eq:weakeq}
B_q(u,v)=\langle g,v \rangle \  \mbox{ for all } v\in \widetilde H^s(\Omega).
\end{align}
\end{defi}

Using this, we prove the well-posedness result for the interior source problem:

\begin{prop}[Well-posedness, no exterior data]
\label{prop:inhom_int}
Let $s\in (0,1)$,  $\Omega\subset\R^n$ be a bounded, Lipschitz open  set, $L(D)$  be as in \eqref{eq:L} with  $a$ satisfying \textnormal{(\hyperref[ass:A1]{A1})}-\textnormal{(\hyperref[ass:A3]{A3})} and  $q\in L^\infty(\Omega)$.
There exists a countable set $\Sigma_q\subset(-\|q_-\|_{L^\infty(\Omega)},\infty)$ such that if $\lambda\notin\Sigma_q$, then for any $g\in H^{-s}(\Omega)$  there is a unique solution  
$u\in \widetilde H^s(\Omega)$ of 
\begin{align}\label{eq:inhom_int_lambda}
\begin{split}
\big(L(D) + q-\lambda\big) u & = g \;\mbox{ in } \Omega,\\
u & = 0 \; \mbox{ in } \mathcal C(\Omega)\backslash\overline{\Omega}.
\end{split}
\end{align}
In addition, the solution satisfies
\begin{align}\label{eq:inhom_int_est}
    \|u\|_{H^s(\R^n)}\leq C\|g\|_{H^{-s}(\Omega)}.
\end{align}
\end{prop}

\begin{proof}
\textit{Step 1: Solvability for $L(D)+q+\gamma$.} Let $\gamma=\|q_-\|_{L^\infty(\Omega)}$, where $q_-(x):=\min\{q(x),0\}$. Then $B_q+\gamma$ is a coercive continuous bilinear form on $\widetilde H^s(\Omega)$.
Indeed, by H\"older's inequality, Lemmas ~\ref{lem:comp}, ~\ref{lem:bilR_Fourier} and  ~\ref{lem:HVR}, for any $u, v\in \widetilde H^s(\Omega)$
\begin{align*}
    |B_q(u,v)|
    &\leq 
    \frac{1}{2}\|L(D)^{\frac 1 2}u\|_{L^2(\R^n)}\|L(D)^{\frac 1 2}v\|_{L^2(\R^n)}
    +\|q\|_{L^\infty(\Omega)}\|u\|_{L^2(\Omega)}\|v\|_{L^2(\Omega)}
    \\&\leq C\|u\|_{H^s(\R^n)}\|v\|_{H^s(\R^n)}.
\end{align*} 
The coercivity follows from the Poincar\'e inequality \eqref{eq:Poincare}:
\begin{align*}
    B_q(u,u)+\gamma(u,u)&\geq\frac{1}{2}\|L(D)^{\frac 1 2}u\|_{L^2(\R^n)}^2
    \\&\geq C\Big(\|L(D)^{\frac 1 2}u\|_{L^2(\R^n)}^2+\|u\|_{L^2(\Omega)}^2\Big)
    \geq C\|u\|^2_{H^s(\R^n)}.
\end{align*}
Then there is a unique  $u=:Kg\in \widetilde H^s(\Omega)$ satisfying
\begin{align*}
    B_q(u,v)+\gamma(u,v)_\Omega=\langle g,v\rangle \ \mbox{ for all } v\in \widetilde H^s(\Omega),
\end{align*}
and in addition
\begin{align*}
    \|u\|_{H^s(\R^n)}\leq C\|g\|_{H^{-s}(\Omega)}.
\end{align*}

\emph{Step 2: Fredholm alternative.} Let us consider the operator $K: g\mapsto u$. We have already seen that $K: H^{-s}(\Omega)\to \widetilde H^s(\Omega)$  is bounded.
Further, the embedding $\widetilde H^s(\Omega)\hookrightarrow L^2(\Omega)$ is compact.
Therefore, $K$ is a compact, self-adjoint operator when viewed as a mapping from $L^2(\Omega)\to L^2(\Omega)$. The conclusion hence follows by the spectral theorem and the Fredholm alternative. 
\end{proof}

In a second step, we obtain the well-posedness of the exterior value problem \eqref{eq:Dirichlet2} by reducing it to the problem \eqref{eq:inhom_int}.
In order to highlight the fact that the nonlocal problem allows for very weak data spaces and to clarify their connection with the bilinear form $B_q$, here we work with the  spaces defined in Section ~\ref{sec:Vsa}, similarly as in \cite{FKV15}.

\begin{defi}\label{def:weaksol2}
Let $s\in (0,1)$ and let $B_q(\cdot,\cdot)$ be the bilinear form from \eqref{eq:bil}.
Given  $f\in V^s_e(\Omega,a)$, $u\in V^s(\Omega,a)$ is a \emph{(weak) solution to \eqref{eq:Dirichlet2}  (based  on $B_q$)} if
\begin{align*}
\begin{split}
   &B_q(u,v)=0 \ \mbox{ for all } v\in\widetilde H^s(\Omega),
   \\
   &\mbox{and }\ \mathcal E_{\mathcal C(\Omega)}(u-\tilde f)\in \widetilde H^s(\Omega) \
   \mbox{ for any } \tilde f\in V^s(\Omega,a) \mbox{ with }\tilde f|_{\mathcal C(\Omega)\backslash \overline{\Omega}}=f.
\end{split}
\end{align*} 
\end{defi}

With this in mind, we address the exterior data well-posedness result:

\begin{prop}[Well-posedness, exterior data]
\label{prop:exterior_data}
Let $s\in (0,1)$,  $\Omega\subset\R^n$ be a bounded, Lipschitz open  set, $L(D)$  be as in \eqref{eq:L} with  $a$ satisfying \textnormal{(\hyperref[ass:A1]{A1})}-\textnormal{(\hyperref[ass:A3]{A3})} and  $q\in L^\infty(\Omega)$.
There exists a countable set $\Sigma_q \subset(-\|q_-\|_{L^\infty(\Omega)},\infty)$ such that if $\lambda\notin\Sigma_q$, then for any $f\in V^s_e(\Omega, a)$  there is a unique solution  
$u\in V^s(\Omega, a)$ of 
\begin{align}\label{eq:Dirichlet2lambda}
\begin{split}
    \big(L(D) + q-\lambda\big)u & = 0 \mbox{ in } \Omega,\\
u & = f \mbox{ in } \mathcal{C}(\Omega)\backslash \overline{\Omega}.
\end{split}
\end{align}
Moreover, it satisfies
\begin{align}\label{eq:exterior_data_est}
    \|u\|_{V^s(\Omega,a)}\leq C\|f\|_{V^s_e(\Omega,a)}.
\end{align}
\end{prop}

\begin{proof}[Proof of Proposition ~\ref{prop:exterior_data}]
We reduce the exterior setting to the one from Proposition ~\ref{prop:inhom_int}:
Let  $u:= w|_{\mathcal C(\Omega)}+\tilde f\in V^s(\Omega, a)$, where $\tilde f\in V^s(\Omega,a)$ with $\tilde f|_{\mathcal C(\Omega)\backslash \overline{\Omega}}=f$ and $w\in H^s(\R^n)$. Then  we seek to construct $w\in \widetilde H^s(\Omega)$ satisfying \eqref{eq:inhom_int_lambda} with $g =- \big(L(D)+q-\lambda\big)\tilde f|_\Omega$.
It hence suffices to prove that $g \in H^{-s}(\Omega)$, which reduces to proving that $L(D)\tilde f\in H^{-s}(\Omega)$ (interpreted weakly) if $\tilde f\in V^s(\Omega,a)$. 
By the second part of Lemma ~\ref{lem:bilcont}, for any $v\in \widetilde H^s(\Omega)$
\begin{align*}
   | \langle L(D)\tilde f, v\rangle| = |B_0(\tilde f,v)|
    \leq C\|\tilde f\|_{V^s(\Omega,a)}\|v\|_{H^s(\R^n)}.
\end{align*}
This then proves that $L(D)\tilde f \in H^{-s}(\Omega)$ and therefore $\|g\|_{H^{-s}(\Omega)}\leq C\|\tilde f\|_{V^s(\Omega,a)}$.

Finally, by Proposition ~\ref{prop:inhom_int}, if $\lambda\notin \Sigma_q$, there is a unique solution $w\in \widetilde H^s(\Omega)$ of \eqref{eq:inhom_int} with $g=-\big(L(D)+q-\lambda\big)\tilde f|_\Omega$ and it satisfies
\begin{align*}
    \|w\|_{V^s(\Omega,a)}
    &\leq \|w\|_{V^s(\R^n,a)}
    \leq C \|w\|_{H^s(\R^n)}
    \leq C\|g\|_{H^{-s}(\Omega)}
    \leq C\|\tilde f\|_{V^s(\Omega,a)}.
\end{align*}
This implies
\begin{align*}
    \|u\|_{V^s(\Omega,a)}
    &\leq \|\tilde f\|_{V^s(\Omega,a)}+\|w\|_{V^s(\Omega,a)}
    \leq C\|\tilde f\|_{V^s(\Omega,a)}.
\end{align*}
Taking the infimum among all possible $\tilde f$ then leads to \eqref{eq:exterior_data_est}.
\end{proof}

Next, as a direct corollary, we state the well-posedness for \eqref{eq:Dirichlet2} in the more restrictive function spaces $H^s(\mathcal C(\Omega))$, which  will be convenient for later applications to the associated inverse problems -- in particular for defining the corresponding Dirichlet-to-Neumann maps. 
In order to note the advantages of this, we highlight that functions $u \in H^s(\mathcal C(\Omega))$ satisfy  the symmetric regularity condition 
\begin{align}\label{eq:symregcond}
    \big(u(x)-u(y)\big) \frac{a^{\frac 1 2}\big(\frac{x-y}{|x-y|}\big)}{|x-y|^{\frac n2+s}}\in L^2(\mathcal C(\Omega)\times\mathcal C(\Omega)),
\end{align} 
which follows as in the proof of  Lemma ~\ref{lem:HVOm}. Contrary to the function spaces from Proposition ~\ref{prop:exterior_data} we thus also require $H^s$ Sobolev regularity for the data in the exterior domain.
We emphasize that the symmetry in $x$ and $y$ will allow us to define the Dirichlet-to-Neumann operator by means of our bilinear form $B_q(\cdot,\cdot)$ in a ``symmetric way" in the next subsection.

\begin{cor}
\label{cor:Hsexist}
Under the same conditions as in Proposition ~\ref{prop:exterior_data}, if $\lambda\notin\Sigma_q$ and  $f\in H^s(\mathcal C(\Omega)\backslash\overline{\Omega})$,  the  unique (weak) solution to \eqref{eq:Dirichlet2lambda} (based on $B_q$) satisfies
$u\in H^s(\mathcal C(\Omega))$ 
and 
\begin{align}\label{eq:exterior_data_est_2}
    \|u\|_{H^s(\mathcal C(\Omega))}\leq C\|f\|_{H^s(\mathcal C(\Omega)\backslash\overline{\Omega})}.
\end{align}
\end{cor}

\begin{proof}
Let  $\tilde u:=w+\tilde f\in H^s(\R^n)$, 
where $\tilde f\in H^s(\R^n)$ 
with $\tilde f|_{\mathcal C(\Omega)\backslash\overline{\Omega}}=f$ 
and $w\in \widetilde H^s(\Omega)$ satisfies \eqref{eq:inhom_int} 
with $g=-\big(L(D)+q-\lambda\big)\tilde f|_\Omega$.
By Lemma ~\ref{lem:HVR}, it holds
\begin{align*}
    [\tilde f, \tilde f]_{V^s(\Omega,a)}\leq [\tilde f, \tilde f]_{V^s(\R^n,a)}\leq C\|\tilde f\|_{H^s(\R^n)}.
\end{align*}
Therefore, arguing as in the proof of Proposition ~\ref{prop:exterior_data},
we infer $g\in H^{-s}(\Omega)$ and $\|g\|_{H^{-s}(\Omega)}\leq C\|\tilde f\|_{H^s(\R^n)}$. By virtue of Proposition ~\ref{prop:inhom_int}, if $\lambda\notin \Sigma_q$, then  $w\in \widetilde H^s(\Omega)$ exists, is unique and satisfies 
\begin{align*}
    \|w\|_{H^s(\R^n)}
    \leq C\|g\|_{H^{-s}(\Omega)}
    \leq C\|\tilde f\|_{H^s(\R^n)}.
\end{align*}
Defining $u:=\tilde u|_{\mathcal C(\Omega)}$ results in  
\begin{align*}
    \|u\|_{H^s(\mathcal C(\Omega))}\leq \|\tilde u\|_{H^s(\R^n)}
    \leq C\|\tilde f\|_{H^s(\R^n)}
\end{align*}
The estimate \eqref{eq:exterior_data_est_2}   follows by  taking the infimum among all possible $\tilde f$.
\end{proof}

In particular, the corollary allows us to introduce a well-defined Poisson operator $P_q$ as already claimed in \eqref{eq:Poisson}:

\begin{defi}[Poisson operator]
\label{def:Poisson}
Let $s\in (0,1)$, $\Omega$, $L(D)$ and $q$ be as above. Let $0\notin \Sigma_q$. Then, the \emph{Poisson operator} is given as the mapping
\begin{align}
\label{eq:Poisson1}
P_q : H^s(\mathcal{C}(\Omega)\backslash \overline{\Omega}) \rightarrow H^{s}(\mathcal{C}(\Omega)), \ f \mapsto u=:P_q f,
\end{align}
where $u$ denotes the solution to \eqref{eq:Dirichlet2} from Corollary ~\ref{cor:Hsexist} (with $\lambda=0$).
\end{defi}

After discussing the existence and uniqueness of solutions to \eqref{eq:Dirichlet2} associated with the bilinear form $B_q$, we discuss analogous results for the bilinear form $\tilde{B}_q$ and relate the corresponding solution notions.

To this end, we begin by defining weak solutions to \eqref{eq:Dirichlet2lambda} based on $\tilde{B}_q$:

\begin{defi}
\label{eq:defi_weak_Bq}
Let $s\in (0,1)$ and let $\tilde B_q(\cdot,\cdot)$ be the bilinear form from \eqref{eq:bilR}. Given $f\in H^s(\mathcal C(\Omega)\backslash\overline{\Omega} )$,  A function $u \in H^s(\mathcal{C}(\Omega))$ is a \emph{(weak) solution to \eqref{eq:Dirichlet2} (based on $\tilde{B}_q$)} if the following properties hold
\begin{align}\label{eq:defsoltB}
\begin{split}
   & \tilde B_q(\tilde u,v)=0 \ \mbox{ for all } v\in\widetilde H^s(\Omega) \mbox { and any } \tilde u\in H^s(\R^n) \mbox{ with }  \tilde u|_{\mathcal C(\Omega)}=u, \\
 &  \mbox{and }\ \mathcal E_{\mathcal C(\Omega)}(u-\tilde f)\in \widetilde H^s(\Omega)\ \mbox{ for any  } \tilde f\in H^s(\mathcal C(\Omega)) \mbox{ with }   \tilde f|_{\mathcal C(\Omega)\backslash\overline{\Omega}}=f.
\end{split}
\end{align}
\end{defi}

We remark that since the bilinear form $\tilde{B}_q$ is a more nonlocal object, the definition of a solution of \eqref{eq:Dirichlet2lambda} is slightly more involved for data $f\in H^s(\mathcal C(\Omega)\backslash\overline{\Omega} )$ compared to the one for $B_q$. This is due to the fact that the bilinear form $\tilde{B}_q$ (in contrast to the bilinear form $B_q$) requires globally defined inputs (and not only inputs localized on the domain of dependence $\mathcal{C}(\Omega)$). It is thus necessary to work with suitable extensions as in the definition stated above.

\begin{rmk}
\label{rmk:zerodata1}
Let us collect the following observations on Definition \ref{eq:defi_weak_Bq}:
\begin{itemize}
\item In contrast to asking the validity of the equation
\begin{align*}
\tilde B_q(\tilde u,v)=0 \ \mbox{ for all } v\in\widetilde H^s(\Omega) \mbox { and \emph{all} } \tilde u\in H^s(\R^n) \mbox{ with }  \tilde u|_{\mathcal C(\Omega)}=u,
\end{align*}
it would also have sufficed to require it 
\begin{align*}
\mbox{ for all } v\in\widetilde H^s(\Omega) \mbox { and \emph{some} } \tilde u\in H^s(\R^n) \mbox{ with }  \tilde u|_{\mathcal C(\Omega)}=u.
\end{align*}
Indeed, this follows by noting that for $v\in \widetilde H^s(\Omega)$ the bilinear form $B_q(\tilde u,v)=0$ is independent of the choice of the extension $\tilde{u}$ of $u$.
\item If data $f\in \widetilde{H}^s(\mathcal C(\Omega)\backslash\overline{\Omega} )$ are considered, the solution notion from Definition \ref{eq:defi_weak_Bq} becomes very intuitive, with the zero extension being the canonical choice since $\mathcal{E}_{\mathcal{C}(\Omega)}(u) \in \widetilde{H}^s(\mathcal C(\Omega))$. Indeed, this follows from the fact that, by definition, $\mathcal{E}_{\mathcal{C}(\Omega)}(u-{f}) \in \widetilde{H}^s(\Omega)$.
\item Compared to restricting our data to the space $\widetilde{H}^s(\mathcal C(\Omega)\backslash\overline{\Omega} )$, the advantage of the more general Definition ~\ref{eq:defi_weak_Bq} however is that we do not necessarily have to consider data supported in $\mathcal{C}(\Omega)$, but may prescribe information up to $\partial \mathcal{C}(\Omega)$. Since the bilinear form $B_q$ is by definition restricted to $\mathcal{C}(\Omega)$, this did not pose any difficulties in the setting of the definition of solutions to the ``more local'' bilinear form $B_q$.
\end{itemize}
\end{rmk}

With this notion and the properties of the bilinear forms from Section ~\ref{sec:pre}, we obtain that the two solution notions agree in $\mathcal{C}(\Omega)$:

\begin{lem}\label{lem:idtwosols}
Let $s\in (0,1)$, $f\in H^s(\mathcal C(\Omega)\backslash\overline{\Omega} )$ and let $\lambda \notin \Sigma_q$. Then there exists a unique weak solution to \eqref{eq:Dirichlet2lambda} based on  $\tilde{B}_q$,  i.e. it satisfies the conditions in Definition ~\ref{eq:defi_weak_Bq} with overall potential $q-\lambda$.   
Moreover, let $u\in H^s(\mathcal C(\Omega))$ be the solution to \eqref{eq:Dirichlet2lambda} from Corollary ~\ref{cor:Hsexist}, based on  the bilinear form  ${B}_q$.  Then $u$  agrees with the unique (weak) solution  of \eqref{eq:Dirichlet2lambda} based on the bilinear form $\tilde{B}_q$.
\end{lem}

\begin{proof}
We start by showing the existence of a unique $\text{u}\in H^s(\mathcal C(\Omega))$ satisfying \eqref{eq:defsoltB} with overall potential $q-\lambda$ for $\lambda\notin \Sigma_q$. As in the  proof of Corollary ~\ref{prop:exterior_data},  we reduce this problem to the one from Proposition ~\ref{cor:Hsexist}.
Firstly, we notice that the notion of solution in Definition ~\ref{def:weaksol1} and Proposition ~\ref{prop:inhom_int}  can be also written in terms of the bilinear form \eqref{eq:bilR} according to Lemma ~\ref{lem:comp}. 
This in particular implies that no unique solution exists if $\lambda\in\Sigma_q$, i.e.  $\tilde{\Sigma}_q = \Sigma_q$.

Secondly, for $\lambda \notin \Sigma_q$ we define $\tilde{ \text{u}}=\text{w}+\tilde f\in H^s(\R^n)$, where  $\tilde f\in H^s(\R^n)$ is such that $\tilde f|_{\mathcal C(\Omega)\backslash\overline{\Omega}}=f$
and $\text{w}\in \widetilde H^s(\Omega)$ solves \eqref{eq:inhom_int_lambda} with $\text{g}=-\big(L(D)+q-\lambda\big)\tilde f|_\Omega$. 
As in the  proof of Corollary  ~\ref{cor:Hsexist}, we infer that $\text{g}\in H^{-s}(\Omega)$.
If $\lambda\notin\Sigma_q$, then there exists a unique $\text{w}$. 
Finally, taking the infimum among all possible $\tilde f$ and $\text{u}:=\tilde{\text{u}}|_{\mathcal C(\Omega)}$, the existence and uniqueness of $\text{u}\in H^s(\mathcal C(\Omega))$ follows.

The equality  $\text{u}=u$, where $u$ is the solution of Corollary ~\ref{cor:Hsexist}, is immediate by comparing the previous construction with the one of $u$ in Corollary ~\ref{cor:Hsexist}. Indeed, for the same extension $\tilde f$, we have $\text{g}=g$. Since $\lambda\notin\Sigma_q$, then $\text{w}=w$.

Lastly, we notice that $\tilde B_q(\tilde u,v)=0$ for all possible extensions $\tilde u\in H^s(\R^n)$ with $\tilde u|_{\mathcal C(\Omega)}=u$.
Indeed, let $\tilde u_1, \tilde u_2\in H^s(\R^n)$ such that $\tilde u_j|_{\mathcal C(\Omega)}=u$  and let $v\in\widetilde H^s(\Omega)$. By Lemma ~\ref{lem:comp}
\begin{align*}
    \tilde B_q(\tilde u_1-\tilde u_2, v)=
    B_q(\tilde u_1-\tilde u_2, v)=0,
\end{align*}
where the last equality follows by the definition in \eqref{eq:bil} and the fact that  $\tilde u_1-\tilde u_2=0$ in $\mathcal C(\Omega)$.
\end{proof}

\subsection{The Dirichlet-to-Neumann operator}
\label{sec:DtN_symm}

Heading towards the discussion of the associated inverse problems, we next define the Dirichlet-to-Neumann map  for the problem \eqref{eq:Dirichlet2} based on the two bilinear forms $B_q$ and $\tilde{B}_q$. 
Roughly speaking, for $\tilde{B}_q$ it is defined as the restriction  of the operator applied to the solution to the ``significant region" in the exterior of $\Omega$ (Proposition ~\ref{prop:DtNtilde}). Under suitable geometric assumptions this is also the case for $B_q$ (Lemma ~\ref{lem:idtwoDtN}), however in general the two definitions do not agree (Remark ~\ref{rmk:inequalDtN}). In this section, due to the more convenient symmetry properties, we will only discuss the Sobolev space setting. Again, we will begin by discussing the setting based on the bilinear form $B_q$.

\begin{prop}[The Dirichlet-to-Neumann operator $\Lambda_q$]
\label{prop:DtN}
Let $s\in (0,1)$,  $\Omega\subset\R^n$ be a bounded, Lipschitz open  set, $L(D)$  be as in \eqref{eq:L} with  $a$ satisfying \textnormal{(\hyperref[ass:A1]{A1})}-\textnormal{(\hyperref[ass:A3]{A3})} and  $q\in L^\infty(\Omega)$ be such that $0\notin\Sigma_q$. 
There exists a well-defined, bounded Dirichlet-to-Neumann operator 
\begin{align*}
   \Lambda_q: H^s(\mathcal C (\Omega)\backslash\overline{\Omega}) \to \big(H^s(\mathcal C (\Omega)\backslash\overline{\Omega})\big)^\ast : \ f\mapsto \Lambda_qf,
\end{align*}given by
\begin{align*}
    \langle\Lambda_qf,h\rangle :=B_q(u_f, \tilde h) \  \mbox{ for } f, h\in H^s(\mathcal C (\Omega)\backslash\overline{\Omega}),
\end{align*}
where $u_f=P_q f\in  H^s(\mathcal C(\Omega))$ is the unique solution to \eqref{eq:Dirichlet2} and $\tilde h\in H^s(\mathcal C(\Omega))$ with $\tilde h|_{\mathcal C (\Omega)\backslash\overline{\Omega}}=h$.
\end{prop}

\begin{proof}
First, we prove that the definition of the Dirichlet-to-Neumann map does not depend on the choice of $\tilde h$.
Let $\tilde h, \tilde  h'\in H^s(\mathcal C(\Omega))$ with  $\tilde h|_{\mathcal C (\Omega)\backslash\overline{\Omega}}=\tilde h'|_{\mathcal C (\Omega)\backslash\overline{\Omega}}=h$. Hence, $\mathcal E _{\mathcal C(\Omega)}(\tilde h-\tilde h')\in \widetilde H^s(\Omega)$ by Lemma ~\ref{lem:zeroext}.
Since $B_q(u_f,v)=0$ for all  $v\in \widetilde H^s(\Omega)$ and recalling the definition in \eqref{eq:bil} which only ``sees" the contributions restricted to $\mathcal{C}(\Omega)$, we obtain
$B_q(u_f, \tilde h-\tilde h')=0$.

Next we prove that $\Lambda_q$ is bounded: By Lemma ~\ref{lem:bilcont} and \eqref{eq:exterior_data_est_2}  we have
\begin{align*}
    |\langle\Lambda_qf,h\rangle|=|B_q(u_f, \tilde h)|
    &\leq C\|u_f\|_{H^s(\mathcal C(\Omega))}\|\tilde h\|_{H^s(\mathcal C(\Omega))}
    \leq C\|f\|_{H^s(\mathcal C(\Omega)\backslash\overline{\Omega})}\|\tilde h\|_{H^s(\mathcal C(\Omega))}.
\end{align*}
Taking the infimum among all possible $\tilde h$, we finally infer
\begin{align*}
    |\langle\Lambda_qf,h\rangle|
    &\leq C\|f\|_{H^s(\mathcal C(\Omega)\backslash\overline{\Omega})}\|h\|_{H^s(\mathcal C(\Omega)\backslash\overline{\Omega})}.
\end{align*}
\end{proof}

\begin{rmk}\label{rmk:forAlessandrini}
Due to the symmetry of the bilinear form, we notice that for any $f, h\in H^s(\mathcal C(\Omega)\backslash\overline{\Omega})$
\begin{align*}
    \langle \Lambda_qf, h\rangle=\langle \Lambda_q h, f\rangle.
\end{align*}
Indeed, let $u_f=P_qf$ and $u_h=P_qh$. In particular, this implies $u_f|_{\mathcal C(\Omega)\backslash\overline{\Omega}}=f$ and  $u_h|_{\mathcal C(\Omega)\backslash\overline{\Omega}}=h$. Then, 
\begin{align*}
    \langle \Lambda_qf, h\rangle
    =&B_q(u_f,u_h)=
B_q(u_h,u_f)
    =\langle \Lambda_q h, f\rangle.
\end{align*}
\end{rmk}

Here we opted for a definition of the Dirichlet-to-Neumann map by means of the Sobolev spaces $H^s$, which assures as much symmetry as possible. Given the stronger well-posedness result from Proposition ~\ref{prop:exterior_data}, an alternative, less symmetric definition could have been imagined based on the functions spaces from there. This however would have required substantially more care in defining the corresponding duality pairing, which we thus do not discuss here.

Next, we turn to the notion of the Dirichlet-to-Neumann operator based on the bilinear form $\tilde{B}_q$. If the boundary data are compactly supported in $\mathcal C (\Omega)\backslash\overline{\Omega}$, by virtue of Lemma ~\ref{lem:idtwosols} and following \eqref{eq:DtNnaive}, we can  define the Dirichlet-to-Neumann operator as follows:

\begin{prop}[The Dirichlet-to-Neumann operator $\tilde{\Lambda}_q$]
\label{prop:DtNtilde}
Let $s\in (0,1)$,  $\Omega\subset\R^n$ be a bounded, Lipschitz open  set, $L(D)$  be as in \eqref{eq:L} with  $a$ satisfying \textnormal{(\hyperref[ass:A1]{A1})}-\textnormal{(\hyperref[ass:A3]{A3})} and  $q\in L^\infty(\Omega)$ be such that $0\notin\tilde\Sigma_q$. 
There exists a well-defined, bounded Dirichlet-to-Neumann operator 
\begin{align*}
    \tilde\Lambda_q: \widetilde H^s(\mathcal C (\Omega)\backslash\overline{\Omega}) \to \big( \widetilde H^s(\mathcal C (\Omega)\backslash\overline{\Omega})\big)^\ast:\ f\mapsto \tilde\Lambda_q f,
\end{align*}
given by
\begin{align*}
    \langle\tilde\Lambda_qf,h\rangle :=\tilde B_q(u_f, h) \  \mbox{ for } f, h\in \widetilde H^s(\mathcal C (\Omega)\backslash\overline{\Omega}),
\end{align*}
where $u_f$ denotes the extension by zero of the solution from Lemma ~\ref{lem:idtwosols}, 
i.e. $u_f=\mathcal E_{\mathcal C(\Omega)}(P_qf)$.
\end{prop}

\begin{proof}
We just need to notice that both $u_f, h\in \widetilde H^s(\mathcal C(\Omega))\subset H^s(\R^n)$ (see Remark \ref{rmk:zerodata1}) so the operator is well-defined.  Boundedness follows by Lemma ~\ref{lem:bilR_Fourier}.
\end{proof}

\begin{rmk}
\label{rmk:global_DtN}
We highlight that in contrast to the operator $\Lambda_q$ we have defined the operator $\tilde{\Lambda}_q$ only on $\widetilde{H}^s(\mathcal{C}(\Omega)\backslash \overline{\Omega})$ and not on ${H}^s(\mathcal{C}(\Omega)\backslash \overline{\Omega})$. This again is due to the fact that $\tilde{B}_q$ necessitates global information (see also the discussion below Definition ~\ref{eq:defi_weak_Bq}). Thus, for a more general set-up extensions would have had to be considered. With such a definition, in general the independence of $\tilde{\Lambda}_q$ of such an extension would not have been obvious.
\end{rmk}

We remark that for this operator we have the representation stated in \eqref{eq:DtNnaive}:

\begin{lem}[Distributional characterization]\label{lem:DtNdist}
Let $s\in (0,1)$, let $q\in L^{\infty}(\Omega)$ be such that $0\notin \Sigma_q$ and let $\tilde{\Lambda}_q$ be the operator from Proposition ~\ref{prop:DtNtilde}. Let $f\in\widetilde H^s(\mathcal C (\Omega)\backslash\overline{\Omega}) $ and let $u_f$ denote the extension by zero of the solution from Lemma ~\ref{lem:idtwosols} associated with the datum $f$. Then $\tilde{\Lambda}_q f = L(D)u_f |_{\mathcal C(\Omega)\backslash\overline{\Omega}}$ in the sense of distributions, i.e. for all $\varphi \in \widetilde H^s(\mathcal C(\Omega)\backslash\overline{\Omega})$  it holds that
\begin{align*}
   \langle\tilde \Lambda_qf, \varphi \rangle =    \langle  L(D) u_f, \varphi \rangle.
\end{align*} 
\end{lem}

\begin{proof}
By Proposition ~\ref{prop:DtNtilde} and Lemma ~\ref{lem:bilR_Fourier}
\begin{align*}
   \langle\tilde \Lambda_qf, \varphi \rangle 
   &=\tilde B_q(u_f, \varphi)
   =\big(L(D)^{\frac12}u_f,L(D)^{\frac 12}\varphi\big)_{L^2(\R^n)}+(qu_f,\varphi)_{L^2(\Omega)}
   \\&= \langle L(D) u_f , \varphi \rangle.
\end{align*}
\end{proof}

While in general the Dirichlet-to-Neumann maps from Propositions ~\ref{prop:DtN} and ~\ref{prop:DtNtilde} do not agree (see Remark ~\ref{rmk:inequalDtN}), we point out some cases in which these definitions coincide:

\begin{lem}\label{lem:idtwoDtN}
Let $s\in (0,1)$, let $q\in L^\infty(\Omega)$ be such that $0\notin\Sigma_q$ and let $\Lambda_q$, $\tilde\Lambda_{q}$ be the operators from above.
Let $W_1, W_2\subset \mathcal C (\Omega)\backslash\overline{\Omega}$. Then  for any $f\in \widetilde H^s(W_1)$
\begin{align}
\label{eq:equiv_DtN}
    \Lambda_q f|_{W_2}= \tilde \Lambda_q f|_{W_2}
\end{align}
provided $W_1\cap W_2=\emptyset$ or $\mathcal C(W_1\cap W_2)\subset \mathcal C(\Omega)$.
\end{lem}

Here the statement in \eqref{eq:equiv_DtN} is understood in the  sense that
\begin{align*}
    \langle\tilde \Lambda_qf, h \rangle= \langle \Lambda_qf, h \rangle \; \mbox{ for any } h\in \widetilde H^s(W_2).
\end{align*}

\begin{proof}
The claim follows by the identity
$\tilde B_q(u_f, h)= B_q(u_f, h)$ for $h\in \widetilde H^s(W_2)$ under the geometric conditions on $W_1, W_2$.
Indeed, let $f\in \widetilde H^s(W_1)$ and $h\in \widetilde H^s(W_2)$. Using the notation from \eqref{eq:kernel_a}, applying  Lemmas ~\ref{lem:idtwosols} and ~\ref{lem:comp} and taking into account the supports of all the functions involved (in particular the fact that $W_1, W_2\subset \mathcal{C}(\Omega)\backslash \overline{\Omega}$), we infer
\begin{align*}
    &\tilde B_q(u_f,h)-B_q(u_f,h)
    = \tilde{B}_q(u_f-f,h)-B_q(u_f-f,h) + \tilde B_q(f,h)-B_q(f,h) \\
    &\quad = \tilde B_q(f,h)-B_q(f,h)\\
    & \quad 
    = \frac{1}{2} \int\limits_{\R^n \backslash \mathcal{C}(\Omega)} \int\limits_{\R^n} \big(f(x)-f(y)\big)\big(h(x)-h(y)\big)k_a(x-y)dydx\\
   & \quad\quad
    + \frac{1}{2} \int\limits_{\mathcal{C}(\Omega)} \int\limits_{\R^n \backslash \mathcal{C}(\Omega)} \big(f(x)-f(y)\big)\big(h(x)-h(y)\big)k_a(x-y)dydx
    \\&\quad =\frac 12\int_{\mathcal C(\Omega)_e}\int_{\R^n} f(y)h(y)k_a(x-y)dydx
     +\frac12 \int_{\mathcal C(\Omega)}\int_{\mathcal C(\Omega)_e} f(x)h(x)k_a(x-y)dydx
     \\
     &\quad =\int_{\mathcal C(\Omega)}\int_{\mathcal C(\Omega)_e} f(x)h(x)k_a(x-y)dydx
     \\&\quad=\int_{W_1\cap W_2}\int_{\mathcal C(\Omega)_e} f(x)h(x)k_a(x-y)dydx,
\end{align*}
where $k_a$ is given in \eqref{eq:kernel_a}.
If the supports of $f$ and $h$ are disjoint, i.e. $W_1\cap W_2=\emptyset$, then the last integral vanishes.
Otherwise, if $\mathcal C(W_1\cap W_2)\subset \mathcal C(\Omega)$, 
then for  any  $x\in W_1\cap W_2$, 
$k_a(x-y)\neq0$ provided $y\in \mathcal C(x)\subset \mathcal C(\Omega)$.
Since the integral with respect to $y$ is restricted to $\mathcal C(\Omega)_e$, the difference vanishes. 
\end{proof}

\begin{figure}[t]
\begin{tikzpicture}

\pgfmathsetmacro{\angleone}{-12}
\pgfmathsetmacro{\angletwo}{12}

\pgfmathsetmacro{\angletwox}{cos(\angletwo)}
\pgfmathsetmacro{\angletwoy}{sin(\angletwo)}

\begin{scope}[scale=.7]

\pgfmathsetmacro{\r}{1}

\pgfmathsetmacro{\xx}{6}
\pgfmathsetmacro{\int}{(-\r*sin(\angletwo+90)+\r*sin(\angleone+90)+sin(\angleone)/cos(\angleone)*\r*(cos(\angletwo+90)-cos(\angleone+90)))/(sin(\angletwo)-sin(\angleone)*cos(\angletwo)/cos(\angleone))}

\draw[black!50,dashed,line width=.7pt, domain=-\xx:\xx] plot ({\r*cos(\angletwo+90)+\x*cos(\angletwo)},{\r*sin(\angletwo+90)+\x * sin(\angletwo)});
\draw[black!50,dashed,line width=.7pt, domain=-\xx:\xx] plot ({\r*cos(\angleone+90)+\x*cos(\angleone)},{\r*sin(\angleone+90)+\x * sin(\angleone)});
\draw[black!50,dashed,line width=.7pt, domain=-\xx:\xx] plot ({\r*cos(\angletwo-90)+\x*cos(\angletwo)},{\r*sin(\angletwo-90)+\x * sin(\angletwo)});
\draw[black!50,dashed,line width=.7pt, domain=-\xx:\xx] plot ({\r*cos(\angleone-90)+\x*cos(\angleone)},{\r*sin(\angleone-90)+\x * sin(\angleone)});

\pgfmathsetmacro{\Wx}{1.7}
\pgfmathsetmacro{\Wy}{0}
\pgfmathsetmacro{\rr}{.5}

\fill[black!10, opacity=.1]
({\r*cos(\angletwo+90)+\xx*cos(\angletwo)},{\r*sin(\angletwo+90)+\xx * sin(\angletwo)}) 
--({\r*cos(\angletwo+90)+\int*cos(\angletwo)},{\r*sin(\angletwo+90)+\int * sin(\angletwo)})
 --({\r*cos(\angleone+90)-\xx*cos(\angleone)},{\r*sin(\angleone+90)-\xx * sin(\angleone)})
 .. controls (-6.7,\Wy) ..({\r*cos(\angletwo-90)-\xx*cos(\angletwo)},{\r*sin(\angletwo-90)-\xx * sin(\angletwo)})
--({\r*cos(\angletwo-90)-\int*cos(\angletwo)},{\r*sin(\angletwo-90)-\int * sin(\angletwo)})
--({\r*cos(\angleone-90)+\xx*cos(\angleone)},{\r*sin(\angleone-90)+\xx * sin(\angleone)}) 
.. controls (6.7,\Wy) ..
({\r*cos(\angletwo+90)+\xx*cos(\angletwo)},{\r*sin(\angletwo+90)+\xx * sin(\angletwo)});

\fill[orange!20, opacity=.8] (0,0) circle [radius=\r,];
\node[orange] at (0,0) {$\Omega$};

\node[black!50] at ({-4*(\r)*\angletwox}, {4*(\r)*\angletwoy}) {$\mathcal C(\Omega)$};

\fill[red!50, opacity=.2] (\Wx,\Wy) circle [radius=\rr];
\node[red] at (\Wx,\Wy) {$W$};

\begin{scope}[xshift=\Wx cm]
\draw[red,densely dotted,line width=.7pt, domain=-0:(\xx-\Wx)] plot ({\rr*cos(\angletwo+90)+\x*cos(\angletwo)},{\rr*sin(\angletwo+90)+\x * sin(\angletwo)});
\draw[red,densely dotted,line width=.7pt, domain=-0:-(\xx+\Wx)] plot ({\rr*cos(\angleone+90)+\x*cos(\angleone)},{\rr*sin(\angleone+90)+\x * sin(\angleone)});
\draw[red,densely dotted,line width=.7pt, domain=-(\xx+\Wx):0] plot ({\rr*cos(\angletwo-90)+\x*cos(\angletwo)},{\rr*sin(\angletwo-90)+\x * sin(\angletwo)});
\draw[red,densely dotted,line width=.7pt, domain=-0:(\xx-\Wx)] plot ({\rr*cos(\angleone-90)+\x*cos(\angleone)},{\rr*sin(\angleone-90)+\x * sin(\angleone)});

\node[red!80] at ({5.5*(\rr)*\angletwox}, {5.5*(\rr)*\angletwoy}) {$\mathcal C(W)$};

\end{scope}

\end{scope}
\end{tikzpicture}
\caption{Example of a setting for which the second condition on Lemma ~\ref{lem:idtwoDtN} holds, i.e. $W=W_1\cap W_2 \subset\mathcal C(\Omega)\backslash\overline{\Omega}$ and $\mathcal C(W)\subset\mathcal C(\Omega)$. Notice that the last condition is satisfied only for a small region close to $\Omega$, which increases as the opening angle of the cone is decreased.}
\label{figure:settingidtwoDtN}
\end{figure}
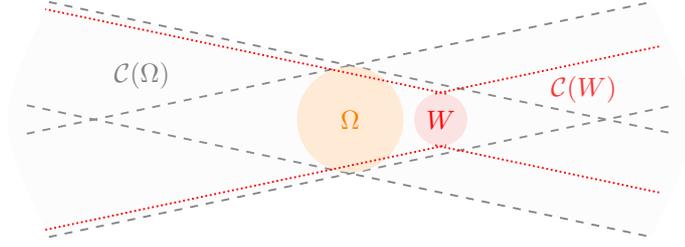

\begin{rmk}\label{rmk:inequalDtN}
In order to illustrate that it is necessary to impose geometric conditions for the equality of the two Dirichlet-to-Neumann maps, we provide a counterexample. We consider $W_1=W_2=W\subset \mathcal C(\Omega)\backslash\overline\Omega$ with $\mathcal C(W)\not\subset \mathcal C(\Omega)$, such that the  conditions on Lemma ~\ref{lem:idtwoDtN} are no longer satisfied.
Let $f\in \widetilde H^s(W)$. By Lemma ~\ref{lem:idtwosols}, the solutions based on each of the bilinear forms agree in $\mathcal C(\Omega)$. 
As a consequence, 
\begin{align*}
    \langle \tilde \Lambda_q f,f\rangle-\langle \Lambda_q f,f\rangle
    &=\tilde B_q(u_f,f)-B_q(u_f,f)
    \\&=\tilde B_q(u_f-f,f)+\tilde B_q(f,f)- B_q(u_f-f,f)- B_q(f,f)
    \\&=\tilde B_q(f,f)- B_q(f,f),
\end{align*}
where the last identity follows by Lemma ~\ref{lem:comp} and the fact that $u_f-f\in \widetilde H^s(\Omega)$.
However,  under the assumed geometric conditions, $\tilde B_q(f,f)-B_q(f,f)>0$ by   Remark ~\ref{rmk:inequal}.
\end{rmk}

\section{Antilocality, Unique Continuation and the Inverse Problem}
\label{sec:inv_first}
In this section we discuss the structural properties of antilocality, unique continuation and Runge approximation. We apply the corresponding results to deduce infinite (partial data) and single measurement uniqueness results for the associated inverse problems.

\subsection{Proofs of Proposition ~\ref{prop:exterior} and Theorem ~\ref{thm:dual}}

Equipped with the well-posedness results from the previous section, we next discuss the general results of Proposition ~\ref{prop:exterior} and of Theorem ~\ref{thm:dual}. In particular, this shows that the class of operators from \textnormal{(\hyperref[ass:A1]{A1})}-\textnormal{(\hyperref[ass:A3]{A3})}  satisfies a ``weak type" of directional antilocality property which is used in the analysis of the inverse problem below.

\subsubsection{Proof of Proposition ~\ref{prop:exterior}}
\label{sec:antiloc_ext}
We first turn to the proof of Proposition ~\ref{prop:exterior}. Here the main part of the proof consists in showing that condition $(ii)$ can be reduced to \cite[Corollary ~5.2]{RS17b}, and that the operator $L(D)$ behaves analogously to the fractional Laplacian in that its complex extension requires the presence of a branch-cut.

\begin{proof}[Proof of Proposition ~\ref{prop:exterior}]

\textit{Step 1: Condition $(i)$.} The claim that condition $(i)$ implies that $f \equiv 0$ in $\Omega$ if $f= 0= L(D)f$ in $\mathcal C(\Omega)\backslash \overline{\Omega}$ is a direct consequence of the antilocality in cones (Lemma ~\ref{lem:antilocHr}) and of the fact that any point in $\Omega$ (actually any point) can be reached through a translated two-sided cone of the form $\mathcal C(\{x_0\})$ with $x_0\in\mathcal C(\Omega)\backslash \overline{\Omega}$.

\textit{Step 2: Condition $(ii)$.} It thus remains to show that condition $(ii)$ also implies the desired result. To this end, we argue by contradiction as in the proof of \cite[Corollary 5.2]{RS17b} (with the fractional Laplacian replaced by the operator $L(D)$) assuming that $f \neq 0$. We note that by a mollification argument as in the proof of Lemma ~\ref{lem:antilocHr} we may without loss of generality assume that $f\in C^{\infty}_c(\R^n)$.

Due to our assumption we have $f = 0$ in $\R^n\backslash \overline{\Omega}$. 
By virtue of the domain of dependence structure, we then also obtain that $L(D) f = 0$ in $\R^n \backslash \overline{\Omega}$. Hence, we are in the setting of \cite[Corollary ~5.2]{RS17b}, but instead of the fractional Laplacian we consider the operator $L(D)$. More precisely, by the Paley-Wiener theorem,  both  $\widehat{L f}(\xi)$ and $\hat{f}(\xi)$  have analytic extensions into the complex plane $\C^n$. We infer that the quotient  $L(\xi) = \frac{\widehat{L f}(\xi)}{\hat{f}(\xi)}$  is thus a meromorphic function if $\hat{f}(\xi) \not \equiv 0$.

In order to conclude as in \cite[Corollary 5.2]{RS17b}, it remains to produce a contradiction to the assumption that $\hat{f} \not \equiv 0$. This will be achieved by proving that any continuation of the symbol $L(\xi)$ of the operator $L(D)$ into the complex plane $\C^n$ necessitates the presence of a branch-cut singularity, analogously as for the symbol $|\xi|^{2s}$. In this case $L(\xi)$ is not a meromorphic function, and the identity $L(\xi)  = \frac{\widehat{L f}(\xi)}{\hat{f}(\xi)}$ could thus only be valid if $\hat{f} \equiv 0$. This would imply the desired result. 

As in \cite[Corollary 5.2]{RS17b} it suffices to study restrictions of $L(\xi)$ to one-dimensional lines which do not fully lie in the nodal set of $\hat{f}(\xi)$. If $\hat{f} \not \equiv 0$, such lines exist. We assume without loss of generality that one of these lines is given by $\lambda e_1$, with $\lambda \in \R$ and $e_1$ denoting the first unit vector. On this line, the symbol turns into
\begin{align*}
L(\lambda e_1) = C |\lambda|^{2s},
\end{align*}
for some constant $C >0$. Now, as a function of $\lambda$, this function has the property that any analytic extension $L(\tilde{\lambda} e_1)$ with $\tilde{\lambda}\in \C$ necessarily has a branch cut singularity. This contradicts the fact that both $Lf(\tilde{\lambda} e_1)$ and $\hat{f}(\tilde{\lambda} e_1) \not \equiv 0$ are analytic functions in $\tilde{\lambda}\in \C$. Hence, $\hat{f}(\xi) \equiv 0$. 
\end{proof}

\subsubsection{Proof of Theorem ~\ref{thm:dual}}

The proof of Theorem ~\ref{thm:dual} follows along the same lines as the analogous result for the fractional Laplacian in \cite[Proposition 2.3]{GRSU18}. The implication $(b)\Rightarrow (a)$ is a consequence of the Hahn-Banach theorem, while the implication $(a)\Rightarrow (b)$ follows by density and the well-posedness of the equation.

\begin{proof}[Proof of Theorem ~\ref{thm:dual}]
\emph{Step 1: $(a) \Rightarrow (b)$.}
By the assumed Runge approximation property, for any $\psi \in \widetilde H^s(\Omega)$ and any $\epsilon>0$ there exists a function $f\in C_c^{\infty}(W)$ such that
\begin{align*}
\|P_qf|_\Omega-\psi\|_{H^s(\R^n)} \leq \epsilon.
\end{align*}
We notice that 
\begin{align*}
\langle v,\psi\rangle &= \langle v,\psi-P_qf|_\Omega \rangle + \langle v, P_q f|_\Omega\rangle.
\end{align*}
By  virtue of  Lemmas ~\ref{lem:bilR_Fourier} and ~\ref{lem:comp}, the symmetry of $\tilde B_q$
and the assumption that $L(D) w|_{W}=0$, we obtain for $v\in H^{-s}(\Omega)$
\begin{align*}
\langle v, P_q f|_\Omega\rangle 
 &=  B_q(w,P_q f|_{\Omega}) =  \tilde{B}_q(w,P_q f|_{\Omega})
=  \tilde{B}_q(w,P_q f - f) 
\\&= 
 - \tilde B_q(w, f)
=  - \langle L(D)w,  f\rangle=0.
\end{align*}
Now, by the assumed Runge approximation property, 
\begin{align*}
|\langle v,\psi\rangle| \leq \|v\|_{H^{-s}(\Omega)} \|\psi - P_qf|_\Omega\|_{H^s(\Omega)} \leq \epsilon \|v\|_{H^{-s}(\Omega)}.
\end{align*}
Considering $\epsilon \rightarrow 0$ then implies that $\langle v,\psi \rangle = 0$ for all $\psi \in \widetilde H^s(\Omega)$. This entails that $v\equiv 0$ in $\Omega$ and by the well-posedness of \eqref{eq:dual}, we obtain that $w\equiv 0$ in $\mathcal{C}(\Omega)$.

\medskip
\emph{Step 2: $(b) \Rightarrow (a)$.} 
By the Hahn-Banach theorem, it suffices to show the following implication:
\begin{equation}\label{hahn-banach}
   \mbox{If } F\in H^{-s}(\Omega) \mbox{ such that } \langle F, u \rangle=0\;\; \forall u\in\mathcal R \quad \Rightarrow \quad F\equiv 0\;. 
\end{equation} 
Therefore, let $F\in H^{-s}(\Omega)$ verify the assumption of \eqref{hahn-banach} and define $\phi\in \widetilde H^s(\Omega)$ as the unique solution  of
\begin{align*}
\begin{split}
\big(L(D) + q\big) \phi & = -F \;\mbox{ in } \Omega,\\
\phi & = 0 \;\quad \mbox{ in } \mathcal{C}(\Omega)\backslash \overline{\Omega}.
\end{split}
\end{align*}
Let $f\in C^\infty_c(W)$ with $W\subset \mathcal C(\Omega)\backslash \overline{\Omega}$. By the well-posedness assumption, $P_qf|_{\Omega} = P_qf-f \in \widetilde H^s(\Omega)$. 
Then, according to Definition ~\ref{def:weaksol1} and Lemma ~\ref{lem:comp}, it follows that
\begin{equation*}
    0= \langle F, P_qf|_{\Omega}\rangle = - B_q(\phi, P_qf|_{\Omega})= - \tilde B_q(\phi, P_qf|_{\Omega}) = -\tilde B_q(\phi,P_qf) + \tilde B_q(\phi,f)\;.
\end{equation*}
However, since $P_qf$ clearly verifies $\big(L(D)+q\big)P_qf=0$ in $\Omega$ and by definition $\phi\in \widetilde H^s(\Omega)$, the term $\tilde B_q(\phi,P_qf)=B_q(\phi,P_qf)$ vanishes. Therefore, by Lemma ~\ref{lem:bilR_Fourier}, we are left with
$$0=\tilde B_q(\phi,f)= \big( L^{\frac 12}(D)\phi, L^{\frac 12}(D)f \big)_{L^2(\R^n)} + ( q\phi, f )_{L^2(\Omega)}= \langle L(D)\phi,f \rangle.$$
By the arbitrary choice of $f\in C^\infty_c(W)$, we are lead to the conclusion that $L(D)\phi =0$ in $W$.  Assumption $(b)$ implies that $\phi\equiv 0$ in $\mathcal{C}(\Omega)$ and $F\equiv 0$, which thus concludes the proof.
\end{proof}

\subsection{Runge approximation and unique continuation}
\label{sec:Runge_approx_first}
In this section, we relate the  crucial properties of antilocality, Runge approximation and unique continuation. We emphasize that -- in spite of ellipticity -- the geometry of the operator now plays a more subtle role than in local inverse problems.

\subsubsection{Runge approximation (Proof of Theorem ~\ref{thm:Runge})}
We apply the assumed $-\mathcal C\cup\mathcal C$-antilocality to prove the Runge approximation property by  exploiting the  duality result of Theorem ~\ref{thm:dual}.

\begin{proof}[Proof of Theorem ~\ref{thm:Runge}]
By virtue of Theorem ~\ref{thm:dual}, it suffices to show that the assumption $(b)$ in Theorem ~\ref{thm:dual}  is satisfied. To this end, let $w\in \widetilde H^s(\Omega)$ solve
\begin{align*}
\begin{split}
\big(L(D)+q\big)w & = v \; \mbox{ in } \Omega,\\
w& = 0 \; \mbox{ in } \mathcal{C}(\Omega) \backslash \overline{\Omega},
\end{split}
\end{align*}
for some $v\in H^{-s}(\Omega)$. Assume also  $L(D)w = 0 $ in $W$. Since $W\subset \mathcal{C}(\Omega)\backslash \overline{\Omega}$, we have $w=0=L(D)w$ in $W$.
Since $L(D)$ is $-\mathcal C\cup\mathcal C$-antilocal, by Lemma ~\ref{lem:antilocHr} it follows  that  $w=0$ in $\mathcal C(W)$. 
Since $\Omega\subset \mathcal C(W)$, we have obtained that $w=0$ in $\Omega$. 
Recalling the boundary conditions, we hence conclude that $w \equiv 0$ in $\mathcal C(\Omega)$. Due to the domain of dependence of the operator $L(D)$, this implies $v\equiv 0$, so the proof is completed.
\end{proof}

\begin{rmk} \label{rmk:positionW}
We highlight that the condition $\Omega \subset\mathcal C(W)$ in Theorem ~\ref{thm:Runge} does not imply any restriction in the size of $W$.
Indeed, it can be taken  arbitrarily small provided its distance from $\Omega$ is large enough. 
In order to see this, for any $\theta \in (-\mathcal C\cup \mathcal{C})\cap\Ss^{n-1}$ let $r(\theta)>0$ be the radius of the largest ball contained in $-\mathcal C\cup\mathcal C$ and centered in $\theta $. In addition, for any  $x\in\Omega$,  we have $\Omega\subset B_{\diam(\Omega)}(x)$. Therefore, $\Omega\subset \mathcal{C}\big(\big\{x-\frac{\diam(\Omega)}{r(\theta)}\theta\big\}\big)$ (see Figure ~\ref{figure:rmkW}).
This implies that if there are $\textnormal{w}\in W$, $x\in \Omega$ and  $t\in \R$ with $$x=\textnormal{w}+t\theta\;,\quad\quad \dist(x,\textnormal{w})=|t|\geq \frac{\diam(\Omega)}{r(\theta)}\;,$$
then $\Omega\subset \mathcal{C}(W)$ regardless of the shape and size of $W$. The conditions above are granted, e.g. if $W$ is at least $\frac{\diam(\Omega)}{r(\theta)}$ away from $\Omega$ in the direction $-\theta$. 
\end{rmk}

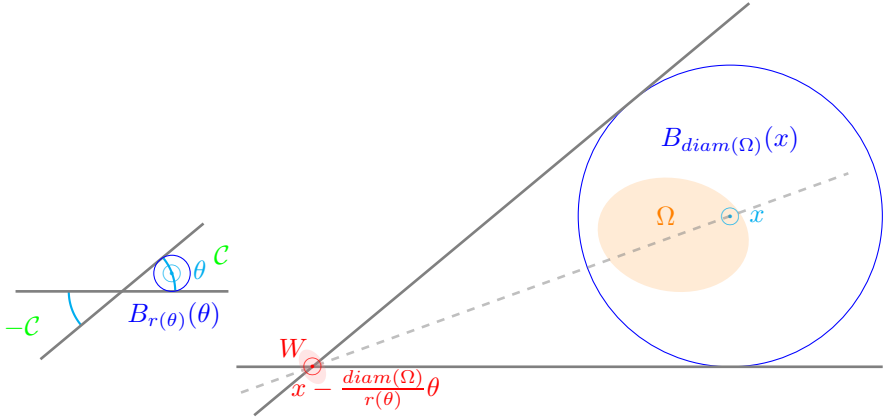
\begin{figure}[t]
\begin{tikzpicture}

\pgfmathsetmacro{\angleone}{0}
\pgfmathsetmacro{\angletwo}{40}

\pgfmathsetmacro{\anglemid}{(\angleone+\angletwo)/2}
\pgfmathsetmacro{\anglediff}{abs(\angleone-\angletwo)/2}

\pgfmathsetmacro{\midx}{2* cos(\anglemid)}
\pgfmathsetmacro{\midy}{2* sin(\anglemid)}

\pgfmathsetmacro{\thetax}{cos(\anglemid)}
\pgfmathsetmacro{\thetay}{sin(\anglemid)}
\pgfmathsetmacro{\rtheta}{sin(\anglediff)}

\pgfmathsetmacro{\angonex}{cos(\angleone-10)}
\pgfmathsetmacro{\angoney}{sin(\angleone-10)}

\begin{scope}[xshift=-8cm, yshift=-1cm, scale=.7]
\draw [cyan,thick,domain=\angleone:\angletwo] plot ({cos(\x)}, {sin(\x)});
\draw [cyan,thick,domain=180+\angleone:180+\angletwo] plot ({cos(\x)}, {sin(\x)});

\node[green] at (\midx,\midy) {$\mathcal C$};
\node[green] at (-\midx,-\midy) {$-\mathcal C$};

\draw[black!50,line width=1pt, domain=-2:2] plot ({\x*cos(\angleone)},{\x * sin(\angleone)});
\draw[black!50,line width=1pt, domain=-2:2] plot ({\x*cos(\angletwo)},{\x * sin(\angletwo)});

\draw[blue] (\thetax,\thetay) circle [radius=\rtheta];

\node[cyan] at (\thetax,\thetay) {$\odot$};
\node[cyan] at (1.2*\thetax,1.2* \thetay) {\quad\ $\theta$};

\node[blue] at (\angonex, \angoney-.3) {$B_{r(\theta)}(\theta)$};
\end{scope}

\begin{scope}[scale=.5]

\fill[orange!20, opacity=.8] (-1.5,-.5) ellipse [x radius=2cm,y radius=1.5cm, rotate=-10];
\node[orange] at (-1.7,0) {$\Omega$};

\draw[blue] (0,0) circle [radius=4cm];
\node[cyan] at (0,0) {$\odot$};
\node[cyan] at (0,0) {\qquad $x$};
\node[blue] at (0,2) {$B_{diam(\Omega)}(x)$};

\pgfmathsetmacro{\newx}{-4/ \rtheta * \thetax}
\pgfmathsetmacro{\newy}{-4/\rtheta * \thetay}

\draw[black!50,line width=1pt, domain=-2:15] plot ({\newx +\x*cos(\angleone)},{\newy +\x * sin(\angleone)});
\draw[black!50,line width=1pt, domain=-2:15] plot ({\newx+\x*cos(\angletwo)},{\newy+\x * sin(\angletwo)});

\draw[black!50,line width=1pt,dashed, opacity=.5, domain=-2:15] plot ({\newx+\x*cos(\anglemid)},{\newy+\x * sin(\anglemid)});

\fill[red!50, opacity=.2] (\newx,\newy) ellipse [x radius=.3cm,y radius=.5cm, rotate=30];
\node[red] at (\newx-.5,\newy+.5) {$W$};

\node[red] at (\newx,\newy) {$\odot$};
\node[red] at (\newx+1.4,\newy-.6) {$x-\frac{diam(\Omega)}{r(\theta)}\theta$};

\end{scope}
\end{tikzpicture}
\caption{On the possible location of $W$ with respect to $\Omega$ in Theorem ~\ref{thm:Runge}, see Remark ~\ref{rmk:positionW}. In particular, the assumption $\Omega \subset \mathcal{C}(W)$ from our Runge approximation result does not impose any size estimates for $W$.}
\label{figure:rmkW}
\end{figure}

We next show that the geometric conditions in Theorem ~\ref{thm:Runge} for the validity of the approximation are indeed necessary. 
We have already discussed that the domain of dependence of $L(D)$ requires $W\subset \mathcal C(\Omega)$.

\begin{lem}
\label{lem:necessity}
Assume that the conditions of Theorem ~\ref{thm:Runge} hold except for the assumption that $\Omega \subset \mathcal{C}(W)$. Then, in general, the Runge approximation property fails.
\end{lem}

\begin{proof}
By virtue of Theorem ~\ref{thm:dual}, it suffices to prove that the statement $(b)$ in Theorem ~\ref{thm:dual} fails. 

Let  $\Omega_1 := \Omega\cap\mathcal{C}(W)$ and $\Omega_2 := \Omega\backslash\overline{\Omega}_1$ (see Figure ~\ref{figure:counterexample}).
By assumption
$\Omega_2\neq \emptyset$.
Let $w\in C^\infty_c(\Omega_2)$ such that $v:=\big(L(D)+q\big)w|_\Omega\neq 0$.
Since $L(D)w\in C^\infty(\R^n)$, we have $v\in H^{-s}(\Omega)$. 
It is clear that then $w$ is a solution to \eqref{eq:dual}.
By construction of $w$, $\Omega_1$, $\Omega_2$ and $W$, it holds that $L(D)w|_W=0$.
Therefore, the condition $(b)$ in Theorem ~\ref{thm:dual} fails. 
\end{proof}

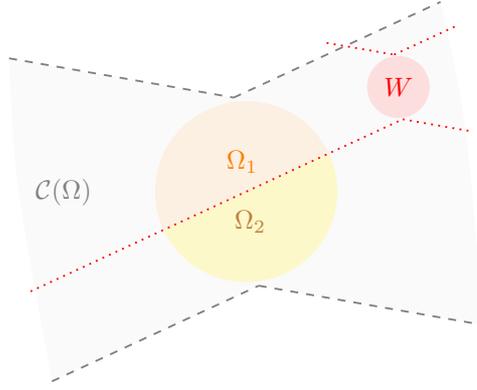
\begin{figure}[t]
\begin{tikzpicture}

\pgfmathsetmacro{\angleone}{-10}
\pgfmathsetmacro{\angletwo}{25}
\pgfmathsetmacro{\r}{2}

\pgfmathsetmacro{\anglemid}{(\angleone+\angletwo)/2}
\pgfmathsetmacro{\angleop}{(-\angleone+\angletwo)/2}

\pgfmathsetmacro{\midx}{1.4* cos(\anglemid)}
\pgfmathsetmacro{\midy}{1.4* sin(\anglemid)}

\begin{scope}[scale=.6]


\pgfmathsetmacro{\sep}{\r/cos(\angleop)}
\pgfmathsetmacro{\sepx}{-\sep*sin(\anglemid)}
\pgfmathsetmacro{\sepy}{\sep*cos(\anglemid)}

\pgfmathsetmacro{\xx}{5}
\pgfmathsetmacro{\thetax}{cos(\anglemid)}
\pgfmathsetmacro{\thetay}{sin(\anglemid)}
\pgfmathsetmacro{\omy}{\r/2* cos(\anglemid)}
\pgfmathsetmacro{\omx}{-\r/2* sin(\anglemid)}

\pgfmathsetmacro{\intx}{\r*cos(\angletwo)}
\pgfmathsetmacro{\inty}{\r*sin(\angletwo)}

\fill[black!20, opacity=.1]
({\sepx-\xx*cos(\angleone)},{\sepy-\xx * sin(\angleone)})--
({\sepx},{\sepy})--({\sepx+\xx*cos(\angletwo)},{\sepy+\xx * sin(\angletwo)})
.. controls (\omx+\xx*\thetax,\omy+\xx*\thetay)..
({-\sepx+\xx*cos(\angleone)},{-\sepy+\xx * sin(\angleone)})
--({-\sepx},{-\sepy})
--({-\sepx-\xx*cos(\angletwo)},{-\sepy-\xx * sin(\angletwo)})
.. controls (\omx-\xx*\thetax,\omy-\xx*\thetay)..
({\sepx-\xx*cos(\angleone)},{\sepy-\xx * sin(\angleone)});

\pgfmathsetmacro{\cpx}{4*cos(\angleone}
\pgfmathsetmacro{\cpy}{4*sin(\angleone}


\pgfmathsetmacro{\angleaux}{180-2*(-\angleone+\angletwo)+\angletwo}

\fill[orange!20, opacity=.5] (\angletwo:\r) arc (\angletwo:\angletwo+180:\r)--(\angletwo:\r);

\fill[yellow!50, opacity=.5] (\angletwo:\r) arc (\angletwo:\angletwo-180:\r)--(\angletwo:\r);

\draw[black!50,dashed,line width=.7pt, domain=-0:\xx] plot ({\sepx+\x*cos(\angletwo)},{\sepy+\x * sin(\angletwo)});

\draw[black!50,dashed,line width=.7pt, domain=-\xx:0] plot ({\sepx+\x*cos(\angleone)},{\sepy+\x * sin(\angleone)});

\draw[black!50,dashed,line width=.7pt, domain=0:\xx] plot ({-\sepx+\x*cos(\angleone)},{-\sepy+\x * sin(\angleone)});

\draw[black!50,dashed,line width=.7pt, domain=-\xx:0] plot ({-\sepx+\x*cos(\angletwo)},{-\sepy+\x * sin(\angletwo)});

\pgfmathsetmacro{\dimW}{1/3*sqrt(\sepx^2+\sepy^2)}
\pgfmathsetmacro{\distW}{2*\r}

\pgfmathsetmacro{\Wx}{(-\dimW* sin(\angletwo)+ \distW*cos(\angletwo))}
\pgfmathsetmacro{\Wy}{\dimW* cos(\angletwo)+ \distW*sin(\angletwo)}

\pgfmathsetmacro{\sepW}{\dimW/cos(\angleop)}
\pgfmathsetmacro{\sepxW}{-\sepW*sin(\anglemid)}
\pgfmathsetmacro{\sepyW}{\sepW*cos(\anglemid)}

\pgfmathsetmacro{\xxm}{1.5}
\pgfmathsetmacro{\xxp}{9}

\begin{scope}[xshift=\Wx cm, yshift=\Wy cm]

\fill[red!20, opacity=.6] (0,0) circle [radius=\dimW];

\node[red] at (0,0) {$W$};

\draw[red, dotted,line width=.7pt, domain=-0:\xxm] plot ({\sepxW+\x*cos(\angletwo)},{\sepyW+\x * sin(\angletwo)});

\draw[red, dotted,line width=.7pt, domain=-\xxm:0] plot ({\sepxW+\x*cos(\angleone)},{\sepyW+\x * sin(\angleone)});

\draw[red, dotted,line width=.7pt, domain=0:\xxm] plot ({-\sepxW+\x*cos(\angleone)},{-\sepyW+\x * sin(\angleone)});

\draw[red, dotted,line width=.7pt, domain=-\xxp:0] plot ({-\sepxW+\x*cos(\angletwo)},{-\sepyW+\x * sin(\angletwo)});

\end{scope}

\node[black!50] at (-2*\r, 0) {$\mathcal C(\Omega)$};

\pgfmathsetmacro{\omy}{\r/3* cos(\anglemid)}
\pgfmathsetmacro{\omx}{-\r/3* sin(\anglemid)}

\node[orange] at (\omx,\omy) {$\Omega_1$};
\node[brown] at (-\omx,-\omy) {$\Omega_2$};

\end{scope}

\end{tikzpicture}
\caption{Setting for the counterexample on the necessity of the geometric condition $\Omega\subset \mathcal C(W)$ in the Runge approximation, see Lemma ~\ref{lem:necessity}. }
\label{figure:counterexample}
\end{figure}

\subsubsection{Weak unique continuation for $-\mathcal C\cup\mathcal C$-antilocal operators}
As a further consequence of our structural assumptions, we argue that a two-sided antilocality property implies the weak unique continuation property. 

\begin{prop}[WUCP]\label{prop:WUCP}
Let $\mathcal C\subset \R^n\backslash\{0\}$ be an open, non-empty, convex cone. Let $L(D)$ be a $-\mathcal C\cup\mathcal C$-antilocal operator.
Then $L(D)$ satisfies the following weak unique continuation property:
Let $\Omega$ be a connected, bounded differentiable domain and  $q\in L^{\infty}(\Omega)$.
If $\big(L(D)+q\big)u=0$ in $\Omega$ and $u=0$ in $U\subset \Omega$, then $u=0$ in $\mathcal C (\Omega)$.
\end{prop}

\begin{proof}
Assume that  $u=0$ in $U$. By the equation and since $U\subset\Omega$, it also holds that $L(D)u=0$ in $U$.
By the $-\mathcal C\cup\mathcal C$-antilocality (Lemma ~\ref{lem:antilocHr}), we infer $u=0$ in $\mathcal C(U)$, which in particular implies that
$u=0=L(D)u$ in $U_1=\mathcal C(U)\cap\Omega$. 
Iterating this argument based on the antilocality in the two-sided cone, we further obtain $u=0=L(D)u$ in $U_{i+1}:=\mathcal C(U_i)\cap\Omega$.
Due to the differentiablity of $\Omega$ and the two-sidedness of $\mathcal C(U_i)$, there exists $N\in\N$ such that $U_N=\Omega$ (see Figure ~\ref{fig:WUCPdoublecone} and Lemma ~\ref{lem:coveringdouble}).
Therefore, applying the $-\mathcal C\cup\mathcal C$-antilocality argument iteratively, we conclude that $u=0$ in $\mathcal C(\Omega)$.
\end{proof}

\begin{figure}[t]
\begin{tikzpicture}

\pgfmathsetmacro{\anglesym}{20}
\pgfmathsetmacro{\angleone}{-\anglesym}
\pgfmathsetmacro{\angletwo}{\anglesym}
\pgfmathsetmacro{\angletwox}{cos(\angletwo)}
\pgfmathsetmacro{\angletwoy}{sin(\angletwo)}

\pgfmathsetmacro{\r}{.5}
\pgfmathsetmacro{\R}{1.5}
\pgfmathsetmacro{\RR}{2}

\begin{scope}[scale=.8]

\clip (-4, -3) rectangle (4,3);

\pgfmathsetmacro{\xx}{10}
\pgfmathsetmacro{\xxx}{10}
\pgfmathsetmacro{\int}{(-\r*sin(\angletwo+90)+\r*sin(\angleone+90)+sin(\angleone)/cos(\angleone)*\r*(cos(\angletwo+90)-cos(\angleone+90)))/(sin(\angletwo)-sin(\angleone)*cos(\angletwo)/cos(\angleone))}

\pgfmathsetmacro{\sep}{\r/cos(\anglesym)}


\fill[orange!20, opacity=.5] (-1,0) circle (\RR);

\draw[blue!50,dashed,line width=.7pt, domain=-0:\xx] plot ({\x*cos(\angletwo)},{\sep+\x * sin(\angletwo)});

\draw[blue!50,dashed,line width=.7pt, domain=-\xx:0] plot ({\x*cos(\angleone)},{\sep+\x * sin(\angleone)});

\draw[blue!50,dashed,line width=.7pt, domain=-\xx:0] plot ({\x*cos(\angletwo)},{-\sep+\x * sin(\angletwo)});

\draw[blue!50,dashed,line width=.7pt, domain=-0:\xx] plot ({\x*cos(\angleone)},{-\sep+\x * sin(\angleone)});

\pgfmathsetmacro{\Wx}{1.1*\xxx}
\pgfmathsetmacro{\Wy}{0}

\fill[blue!20, opacity=.1]
({\r*cos(\angletwo+90)+\xxx*cos(\angletwo)},{\r*sin(\angletwo+90)+\xxx * sin(\angletwo)}) 
--({\r*cos(\angletwo+90)+\int*cos(\angletwo)},{\r*sin(\angletwo+90)+\int * sin(\angletwo)})
 --({\r*cos(\angleone+90)-\xxx*cos(\angleone)},{\r*sin(\angleone+90)-\xxx * sin(\angleone)})
 .. controls (-\Wx,\Wy) ..({\r*cos(\angletwo-90)-\xxx*cos(\angletwo)},{\r*sin(\angletwo-90)-\xxx * sin(\angletwo)})
--({\r*cos(\angletwo-90)-\int*cos(\angletwo)},{\r*sin(\angletwo-90)-\int * sin(\angletwo)})
--({\r*cos(\angleone-90)+\xxx*cos(\angleone)},{\r*sin(\angleone-90)+\xxx * sin(\angleone)}) 
.. controls (\Wx,\Wy) ..
({\r*cos(\angletwo+90)+\xxx*cos(\angletwo)},{\r*sin(\angletwo+90)+\xxx * sin(\angletwo)});

\draw[blue!60, line width=.7 pt] (0,0) circle [radius=\r,];

\begin{scope}

\clip (-1,0) circle (\RR);

\fill[pattern color= red!50, pattern= crosshatch dots, opacity=.7]
({\r*cos(\angletwo+90)+\xxx*cos(\angletwo)},{\r*sin(\angletwo+90)+\xxx * sin(\angletwo)}) 
--({\r*cos(\angletwo+90)+\int*cos(\angletwo)},{\r*sin(\angletwo+90)+\int * sin(\angletwo)})
 --({\r*cos(\angleone+90)-\xxx*cos(\angleone)},{\r*sin(\angleone+90)-\xxx * sin(\angleone)})
 .. controls (-\Wx,\Wy) ..({\r*cos(\angletwo-90)-\xxx*cos(\angletwo)},{\r*sin(\angletwo-90)-\xxx * sin(\angletwo)})
--({\r*cos(\angletwo-90)-\int*cos(\angletwo)},{\r*sin(\angletwo-90)-\int * sin(\angletwo)})
--({\r*cos(\angleone-90)+\xxx*cos(\angleone)},{\r*sin(\angleone-90)+\xxx * sin(\angleone)}) 
.. controls (\Wx,\Wy) ..
({\r*cos(\angletwo+90)+\xxx*cos(\angletwo)},{\r*sin(\angletwo+90)+\xxx * sin(\angletwo)});

\end{scope}

\pgfmathsetmacro{\tt}{10}

\pgfmathsetmacro{\tL}{-(\sep*sin(\anglesym)-cos(\anglesym))+sqrt((\sep*sin(\anglesym)-cos(\anglesym))^2-(\sep^2+1-\RR^2))}

\pgfmathsetmacro{\xtL}{-\tL*cos(\anglesym}
\pgfmathsetmacro{\ytL}{\sep+\tL*sin(\anglesym}

\draw[red!50,dashed,line width=.7pt, domain=0:\tt] plot ({\xtL+\x*cos(\anglesym)},{\ytL+\x * sin(\anglesym)});

\draw[red!50,dashed,line width=.7pt, domain=0:-\tt] plot ({\xtL-\x*cos(\anglesym)},{-\ytL+\x * sin(\anglesym)});

\pgfmathsetmacro{\tR}{-(\sep*sin(\anglesym)+cos(\anglesym))+sqrt((\sep*sin(\anglesym)+cos(\anglesym))^2-(\sep^2+1-\RR^2))}

\pgfmathsetmacro{\xtR}{\tR*cos(\anglesym}
\pgfmathsetmacro{\ytR}{\sep+\tR*sin(\anglesym}

\draw[red!50,dashed,line width=.7pt, domain=0:\tt] plot ({\xtR-\x*cos(\anglesym)},{\ytR+\x * sin(\anglesym)});

\draw[red!50,dashed,line width=.7pt, domain=0:-\tt] plot ({\xtR+\x*cos(\anglesym)},{-\ytR+\x * sin(\anglesym)});

\pgfmathsetmacro{\tint}{(\xtR-\xtL-cos(\anglesym)/sin(\anglesym)*(\ytL-\ytR))/(2*cos(\anglesym)}

\pgfmathsetmacro{\xtint}{\xtL+\tint*cos(\anglesym}
\pgfmathsetmacro{\ytint}{\ytL+\tint*sin(\anglesym}

\fill[red!7, opacity=.2]
({\xtL+\tt*cos(\anglesym)},{\ytL+\tt * sin(\anglesym)}) 
--(\xtint, \ytint)
 --({\xtR-\tt*cos(\anglesym)},{\ytR+\tt * sin(\anglesym)})
  .. controls (-5,0) ..
  ({\xtR-\tt*cos(\anglesym)},{-\ytR-\tt * sin(\anglesym)})
--(\xtint, -\ytint)
--({\xtL+\tt*cos(\anglesym)},{-\ytL-\tt * sin(\anglesym)}) 
.. controls (5,0) ..
({\xtL+\tt*cos(\anglesym)},{\ytL+\tt * sin(\anglesym)});

\begin{scope}

\clip (-1,0) circle (\RR);

\fill[pattern color= green!90, pattern= dots, opacity=.5]
({\xtL+\tt*cos(\anglesym)},{\ytL+\tt * sin(\anglesym)}) 
--(\xtint, \ytint)
 --({\xtR-\tt*cos(\anglesym)},{\ytR+\tt * sin(\anglesym)})
  .. controls (-5,0) ..
  ({\xtR-\tt*cos(\anglesym)},{-\ytR-\tt * sin(\anglesym)})
--(\xtint, -\ytint)
--({\xtL+\tt*cos(\anglesym)},{-\ytL-\tt * sin(\anglesym)}) 
.. controls (5,0) ..
({\xtL+\tt*cos(\anglesym)},{\ytL+\tt * sin(\anglesym)});

\end{scope}

\pgfmathsetmacro{\tgL}{-(\ytR*sin(\anglesym)-(1+\xtR)*cos(\anglesym))+sqrt((\ytR*sin(\anglesym)-(1+\xtR)*cos(\anglesym))^2-(\ytR^2+(1+\xtR)^2-\RR^2))}

\pgfmathsetmacro{\xtgL}{\xtR-\tgL*cos(\anglesym}
\pgfmathsetmacro{\ytgL}{\ytR+\tgL*sin(\anglesym}

\draw[green!50,dashed,line width=.7pt, domain=0:\tt] plot ({\xtgL+\x*cos(\anglesym)},{\ytgL+\x * sin(\anglesym)});

\draw[green!50,dashed,line width=.7pt, domain=0:-\tt] plot ({\xtgL-\x*cos(\anglesym)},{-\ytgL+\x * sin(\anglesym)});

\pgfmathsetmacro{\tgR}{-(\ytL*sin(\anglesym)+(1+\xtL)*cos(\anglesym))+sqrt((\ytL*sin(\anglesym)+(1+\xtL)*cos(\anglesym))^2-(\ytL^2+(1+\xtL)^2-\RR^2))}

\pgfmathsetmacro{\xtgR}{\xtL+\tgR*cos(\anglesym}
\pgfmathsetmacro{\ytgR}{\ytL+\tgR*sin(\anglesym}

\draw[green!50,dashed,line width=.7pt, domain=0:\tt] plot ({\xtgR-\x*cos(\anglesym)},{\ytgR+\x * sin(\anglesym)});
\draw[green!50,dashed,line width=.7pt, domain=0:-\tt] plot ({\xtgR+\x*cos(\anglesym)},{-\ytgR+\x * sin(\anglesym)});


\node[orange] at (-1,0) {$\Omega$};
\node[blue] at (0,0) {$U$};
\node[blue!50] at ({2}, {0})  {$\mathcal C(U)$};
\node[red] at (-1.7,.6) {$U_1$};

\node[red!50] at ({2}, {.2+\ytint})  {$\mathcal C(U_1)$};
\node[green] at (-.5,-1.6) {$U_2$};

\end{scope}
\end{tikzpicture}
\caption{Subsets involved in the proof of Proposition ~\ref{prop:WUCP}. The equation is satisfied in $\Omega$ and the solution vanishes in $U$.
In this illustration we have that $\Omega\subset \mathcal C(U_2)$ and thus $U_3=\Omega$ which concludes the proof after three iterations.}
\label{fig:WUCPdoublecone}
\end{figure}
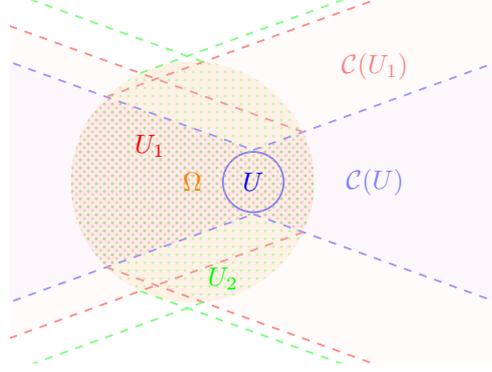

\begin{rmk}
We remark that the proof of Proposition ~\ref{prop:WUCP} strongly uses the two-sided antilocality to deduce the WUCP. In particular, in cases with only one-sided antilocality it is not immediate that antilocality implies the WUCP (see Lemma ~\ref{lem:lackWUCP} for a counterexample in the one-sided setting if $s\in \big(0,\frac{1}{2}\big)$).
\end{rmk}

\subsection{Uniqueness for the inverse problem}

With the Runge approximation and the WUCP results from Theorem ~\ref{thm:Runge} and Proposition ~\ref{prop:WUCP} in hand, we prove the uniqueness  results from Theorems  ~\ref{thm:inf_meas} and ~\ref{thm:single_meas}.

\begin{proof}[Proof of Theorem ~\ref{thm:inf_meas}]
Let $f_j\in C^\infty_c(W_j)$ and $u_j:=\mathcal E_{\mathcal C(\Omega)}P_{q_j}f_j$ for $j=1,2$. 
By Proposition ~\ref{prop:DtN} and Remark ~\ref{rmk:forAlessandrini}, we have the following \emph{Alessandrini identity}:
\begin{align*}
    \langle (\Lambda_{q_1}-\Lambda_{q_2})f_1,f_2\rangle & = \langle \Lambda_{q_1}f_1,f_2\rangle - \langle \Lambda_{q_2}f_2, f_1\rangle \\ & = B_{q_1}(u_1,u_2)-B_{q_2}(u_2,u_1) = \int_{\Omega} (q_1-q_2)u_1u_2 \,dx\;.
\end{align*}
Here we used of the fact that $u_j|_{\mathcal C(\Omega)\backslash\overline{\Omega}}=f_j$ for $j=1,2$. Now the assumption on the Dirichlet-to-Neumann maps implies that the above integral vanishes for all $f_1 \in C_c^{\infty}(W_1)$ and $f_2 \in C_c^{\infty}(W_2)$. Let  $\epsilon >0$ and $g\in C^\infty_0(\Omega)$. 
By the Runge approximation property we can find $f_1\in C^\infty_c(W_1)$ such that  
$u_1|_\Omega = g+r_1$ for $u_1=P_qf_1$, with $\|r_1\|_{H^s(\Omega)}<\epsilon$. 
Similarly, we can find $f_2\in C^\infty_c(W_2)$ such that $u_2|_\Omega=1+r_2$ for $u_2=P_qf_2$, with $\|r_2\|_{H^s(\Omega)}<\epsilon$. 
Thus, by the support assumptions we can write
\begin{align*}
    0 & = \int_{\Omega} (q_1-q_2)u_1u_2 \,dx= \int_{\Omega} (q_1-q_2)(g + gr_2+r_1u_2)\,dx.
\end{align*} 
For the term involving $r_2$ we can estimate
$$|\langle (q_1-q_2)g, r_2 \rangle| \leq \|(q_1-q_2)g\|_{H^{-s}(\Omega)}\|r_2\|_{H^s(\Omega)} \leq \epsilon \|q_1-q_2\|_{L^{\infty}(\Omega)}\|g\|_{H^s(\Omega)}\;, $$
and similarly for the term involving $r_1$. Taking the limit $\epsilon\rightarrow 0$, we eventually obtain $\int_{\Omega}(q_1-q_2)g\,dx=0$. By the arbitrary choice of $g\in C^\infty_c(\Omega)$, we are left with $q_1=q_2$ in $\Omega$.
\end{proof}

The claim of Remark ~\ref{rmk:inf_meas_tilde} follows in exactly the same way. 
Notice that an analogous Alessandrini identity holds for $\tilde \Lambda_{q_j}$ and $\tilde B_{q_j}$. 
Moreover, if $W_1\cap W_2=\emptyset$, by Lemma ~\ref{lem:idtwoDtN}, the two Dirichlet-to-Neumann operators even coincide.

\begin{proof}[Proof of Theorem ~\ref{thm:single_meas}]
We argue similarly as in the case of the fractional Laplacian \cite{GRSU18}.
Indeed, by the $-\mathcal C\cup\mathcal C$-antilocality of $L(D)$, the function $u$ is uniquely determined in $\mathcal{C}(W_{2})$ by the knowledge of  $u|_{W_2}=f|_{W_2}$ and $L(D)(u)|_{W_2}$, which, by Lemma ~\ref{lem:DtNdist}, is given by $\tilde \Lambda_q(f)|_{W_2}$. 
Since by assumption $ \Omega \subset \mathcal{C}(W_2)$, we in particular obtain $u$ in $\Omega$. Now, in $\Omega$ the function $u$ satisfies the equation
\begin{align*}
\big(L(D) + q\big)u & = 0 \mbox{ in } \Omega.
\end{align*}
Hence, solving for the potential, we obtain $q(x) = - \frac{L(D)u(x)}{u(x)}$, provided the quotient is well-defined on a sufficiently large set. Since $q\in C^{0}(\Omega)$, it suffices to ensure that the quotient $\frac{L(D)u(x)}{u(x)}$ is well-defined on open sets, i.e. that $u(x)$ cannot vanish on open sets. In this case, for any $\bar{x}\in \Omega$ and any $r>0$ there always exists a sequence of points $\{x_k\}_{k=1}^\infty \subset B_{r}(\bar{x})$ such that $u(x_k)\neq 0$ and $x_k \rightarrow \bar{x}$. By continuity of $q$ we then have
\begin{align*}
q(\bar{x}) = \lim\limits_{k\rightarrow \infty} q(x_k) = \lim\limits_{k\rightarrow \infty} \left(- \frac{L(D)u(x_k)}{u(x_k)} \right).
\end{align*}
Thus, it suffices to argue that $u(x)$ cannot vanish on open sets. 
If we assume that there is an open set $U\subset \Omega$ such that $u \equiv 0$ in $U$, then, by virtue of the equation, also $L(D)u = 0$ in $U$. Due to the assumed two-sided antilocality of the operator and the regularity of $\Omega$, by  Proposition ~\ref{prop:WUCP}  we infer that $u=0$ in  $\mathcal{C}(\Omega)$ and therefore $f\equiv 0$ in $W_1$. This however is not possible by our assumption $f\neq 0$.
\end{proof}

\begin{rmk}
We remark that as in \cite{GRSU18} it would also be possible to turn the above proof into an algorithmic reconstruction argument, for instance by using a Tychonov regularization argument.
\end{rmk}

Remark ~\ref{rmk:single_meas_tilde}  follows from Theorem ~\ref{thm:single_meas} and Lemma ~\ref{lem:idtwoDtN} under the stated geometric conditions on $W_1$, $W_2$. 
Notice that here, contrary to Remark ~\ref{rmk:inf_meas_tilde}, the proof uses that $\Lambda_q(f)|_{W_2}=L(D)u_f|_{W_2}$, which holds provided the conditions of Lemma ~\ref{lem:idtwoDtN} are satisfied.

\section{Examples: Direct and Inverse Problems for a Class of Non-Symmetric Directionally Antilocal Operators}
\label{sec:exs1}

In this section, we collect some prototypical examples of (not necessarily symmetric) operators with directional antilocality properties and discuss the corresponding direct and inverse problems. The antilocality of most of these operators has been studied in \cite{I86,I88,I89}. We postpone a further systematic study of more general nonlocal operators to future work. Due to the fact that not all of the operators under consideration are symmetric, further new phenomena arise in the formulation and derivation of the corresponding results for the inverse problems. In particular, this will necessitate the introduction of adapted function spaces. We will highlight these aspects in the corresponding sections below.

Let us begin by introducing the class of operators which we will focus on in this section.
To this end, let $p\in[0,1]$ and let  $\Gamma \subset \R^n\backslash\{0\}$ be an open, non-empty, convex cone.
We define
\begin{align}\label{eq:Aspgamma}
A_{p,\Gamma}^s(D)f(x) := 
\begin{cases} \displaystyle
\int_{\mathbb R^n} (f(x)-f(x+y))\nu^s_{p,\Gamma}(y)dy & \mbox{ if } s\in\big(0,\frac 12\big),
\\ \displaystyle
\int_{\mathbb R^n}\big(f(x)-f(x+y)+y\cdot\nabla f(x)e^{-\frac{|y|}{e}}\big)\nu^s_{p,\Gamma}(y)dy & \mbox{ if } s=\frac 12,
\\ \displaystyle
\int_{\mathbb R^n}\big(f(x)-f(x+y)+y\cdot\nabla f(x)\big)\nu^s_{p,\Gamma}(y)dy & \mbox{ if } s\in\big(\frac 12, 1\big),
\end{cases}    
\end{align}
with 
\begin{align}\label{eq:nuspgamma}
    \nu^s_{p,\Gamma}(y):=\frac{p\chi_{-\Gamma}(y) + (1-p)\chi_{\Gamma}(y)}{|y|^{n+2s}},
\end{align}
where $\chi_{\pm\Gamma}$ denotes the characteristic function of $\pm \Gamma$.
These operators are motivated by stochastic applications for stable processes.
For $n=1,2$ and $s\neq \frac 12$, subclasses of the operators $A_{p,\Gamma}^s(D)$ and their rigidity properties have been studied in a sequence of articles \cite{I86,I88,I89}. We refer to Section ~\ref{sec:directional_anti} for more on this.

The operators $A^s_{p,\Gamma}(D)$ in \eqref{eq:Aspgamma} correspond to the following Fourier symbols (see Lemma ~\ref{lem:AspgammaSymb})
\begin{align}\label{eq:Aspgammasymb}
    A_{p,\Gamma}^s(\xi) =
    \begin{cases}
   \displaystyle  c_s \int_{\Gamma\cap \Ss^{n-1}} \Big(1- i(1-2p)\tan(\pi s)\sign(\theta\cdot\xi)\Big)|\theta\cdot\xi|^{2s}d\theta & \mbox{ if } s\in\big(0,\frac 1 2\big)\cup\big(\frac 1 2, 1\big)\,,
    \\
    \displaystyle \int_{\Gamma\cap \Ss^{n-1}} \Big(\frac \pi 2- i(1-2p)\sign(\theta\cdot\xi)\log(|\theta\cdot\xi|)\Big)|\theta\cdot\xi|d\theta & \mbox{ if } s=\frac 1 2\,,
    \end{cases}
\end{align}
where $c_s=-\Gamma(-2s)\cos(\pi s)>0$.
From this we observe that $A_{p,\Gamma}^s(D)^* = A_{1-p,\Gamma}^s(D)=A_{p,-\Gamma}^s(D)$, since $\overline{A_{p,\Gamma}^s(\xi)}=A_{1-p,\Gamma}^s(\xi)$. In particular, in global $\widetilde{H}^s(-\Gamma\cup \Gamma)$ spaces the adjoint operator of $A_{0,\Gamma}^s(D)$ (which only sees the cone $\Gamma$) is $A_{1,\Gamma}^s(D)$ (which only sees the cone $-\Gamma$). With the exception of the case $p= \frac{1}{2}$, the operators $A_{p,\Gamma}^s(D)$ are not symmetric. 

We emphasize that due to the fact that we here deal with non-symmetric operators, new features arise. For instance, for $p\in\{0,1\}$, due to the one-sided antilocality, 
\begin{itemize}
\item new local and nonlocal data spaces will be considered (see Section ~\ref{sec:functionspacesA}), 
\item further geometric restrictions on the measurement location of the inverse problem have to be imposed (see Remark ~\ref{rmk:DtNspgammainfo}),
\item and, in general, the unique continuation properties of the one-sided operators are substantially weaker than for their symmetric counterparts (see Lemma ~\ref{lem:lackWUCP}). 
\end{itemize}
Moreover, contrary to the operators from Sections ~\ref{sec:direct}-\ref{sec:inv_first}, the operators $A^s_{1,\Gamma}(D)$ and $A^s_{0,\Gamma}(D)$ only ``see'' information in a one-sided cone and are thus of special interest to us.

In the following we recall the known antilocality results for the operators $A^s_{p,\Gamma}(D)$, formulate associated direct and inverse problems and study the features arising from the non-symmetric, possibly only one-sided directional antilocality properties of these operators.

\subsection{Directional antilocality}
\label{sec:directional_anti}

We start by recalling some directional antilocality results in the sense of Definition ~\ref{defi:anti} from \cite{I86,I88, I89} for the operators $A^s_{p,\Gamma}(D)$ in one and two dimensions and provide some variations of these results.  A more general, systematic discussion of the antilocality properties of these and related operators is postponed to future work.

\subsubsection{Some observations by Ishikawa in one and two dimensions and variations of these}
\label{sec:Ishi}

We begin by recalling the known one-dimensional results on the antilocality properties of the operators from \eqref{eq:Aspgamma}. To this end, we first observe that since we focus on $\Gamma \subset \R^n\backslash\{0\}$ convex, in one dimension, we either obtain $\Gamma=\R_+$ or $\Gamma=\R_-$.
Notice that $A^s_{p,\Gamma}(D)=A^s_{1-p, -\Gamma}(D)$, hence all the possible cases in one dimension are included in the simplified notation $$A^s_p(D):=A^s_{p,\R_+}(D),$$ 
omitting the dependence on $\Gamma$, which we will thus use in the sequel. Moreover, we also say that such an operator is \emph{antilocal to the right or left} if it is directionally antilocal in $\R_+$ or $\R_-$, respectively. In this setting, the following antilocality property is valid:

\begin{lem}[\cite{I86, I88}]
\label{lem:Ishik1D}
Let $s\in (0,1)\backslash\big\{\frac 12\big\}$. The operator $A^s_p(D)$ is $Y_p$-antilocal, where
\begin{align}\label{eq:Yp}
    Y_p=\begin{cases}
     \R_+ & \mbox{ if } p=0,\\
     \R & \mbox{ if } p\in (0,1),\\
     \R_- & \mbox{ if } p=1,
    \end{cases}
\end{align}
i.e. $A^s_p(D)$ is antilocal if $p\in (0,1)$ and only directionally antilocal to the right/left in the case that $p\in \{0,1\}$, respectively.
\end{lem}

We emphasize that in \cite[Theorem 2.3]{I88} Ishikawa provided a more general result by connecting the global property of \emph{antilocality} (to the left/right) for one-dimensional operators with the \emph{anti-transmission} property (to the right/left). These properties are satisfied by larger classes of operators, including rather general operators of fractional Laplacian type (c.f. the $\mu$-transmission condition from \cite{G15}).

For completeness and in order to illustrate that the results from Lemma ~\ref{lem:Ishik1D} are sharp, we recall an example from \cite{I86} (which is formulated for $s\in (0,\frac{1}{2})$ there, but which remains valid for $s\in\big(\frac 1 2, 1\big)$), showing that $A^s_{0}$ indeed does \emph{not} possess the antilocality condition to the left (see Figure ~\ref{fig:really_directional}):

\begin{lem}[\cite{I86}, p.8]
Let $s\in(0,1)$. Then there exist an open set $U \subset \R$ and a function $f\in C_c^{\infty}(\R)$ such that $f=0=A_{0}^s(D)f$ in $U$ but $f \neq 0$. 
\end{lem}

\begin{proof}
Following \cite{I86} it suffices to consider $f\in C_c^{\infty}(\R)$ with the property that
\begin{align*}
f(x) = \left\{
\begin{array}{ll}
1 &\mbox{ for } x \in (-3,-2),\\
0 & \mbox{ for } x\in (-\infty,-4)\cup (-1,\infty),
\end{array}
\right.
\end{align*}
and to smoothly connect these two intervals (see Figure ~\ref{fig:really_directional}). Then, due to the domain of dependence structure, the claim follows with, for instance, $U =\big(-\frac 12,\frac 1 2\big)$.
\end{proof}

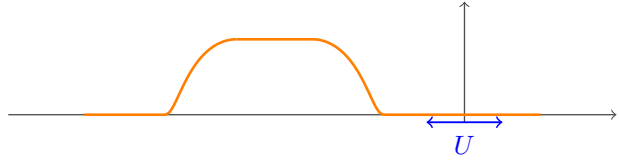
\begin{figure}[t]
\begin{tikzpicture}

\draw[black!70, ->, line width=.5pt] (-6,0)--(2,0);
\draw[black!70, ->, line width=.5pt] (0,-.1)--(0,1.5);

\draw[orange,line width=1pt, domain=0:1] (-3,1)--(-2,1) plot ({\x-2},{(exp(1-1/(1-(\x)^2))})--(-1,0)--(1,0);

\draw[orange,line width=1pt, domain=0:1]  plot ({-\x-3},{(exp(1-1/(1-(\x)^2))})--(-4,0)--(-5,0);

\draw[blue, line width=.7pt, <->] (-1/2,-.1)--(1/2,-.1);
\node[blue] at (0, -.4) {$U$};

\end{tikzpicture}
\caption{
Function for the counterexample of Lemma ~\ref{lem:As0AntMoments}
}
\label{fig:really_directional}
\end{figure}

Ishikawa's results in one dimension do not consider the case $s=\frac 12$ except for $p=\frac 12$. 
Next we provide an alternative proof of the one-sided antilocality which also holds for the case $s=\frac 12$ (and general $p\in [0,1]$) and which builds on the ideas in \cite{GFR20}. We present the result for $A^s_0(D)$, but an analogous result is valid for $A^s_1(D)$.
This, together with the antilocality  for $p=\frac 12$ and Lemma ~\ref{lem:reduceantiloc} below, gives antilocality for any $p\in(0,1)$ in one dimension.

\begin{lem}\label{lem:As0AntMoments}
Let $s\in(0,1)$. Then the operator $A^s_0(D)$ is $\R_+$-antilocal.
\end{lem}

\begin{proof}
Let $f\in C^\infty_c(\R)$ and let $U\subset \R$ be an open bounded interval where  $f=0=A^s_0(D)f$. We seek to deduce that $f=0$ in $U_+=U+\R_+$.

Without loss of generality (by scaling and translation), we assume $U=(0, \beta)$, with $\beta<1$, and $\supp(f)\subset \R_-\cup [2,+\infty)$. Notice that then $f(x+y)=0$ for any $x\in U$ and $y\in (0,1)$. 

By assumption, for any $x\in U$ it holds
\begin{align*}
    A^s_0(D) f(x)=\int_1^\infty \frac{f(x+y)}{y^{1+2s}}dy=0.
\end{align*}
In addition, all the derivatives of $A^s_0(D) f$ vanish in $U$, which after integrating by parts implies 
\begin{align*}
  \int_1^\infty \frac{f(x+y)}{y^{1+2s+k}}dy=0, \quad k\in\N_0.
\end{align*}
Fixing any $x\in U$ and applying the change of variables $y=z^{-1}$, we infer
\begin{align*}
   \int_0^1 F_x(z) z^kdz=0, \quad k\in\N_0,
\end{align*}
where $F_x(z)=z^{2s-1}f\big(x+\frac 1z\big)\in L^2((0,1))$.
By the uniqueness of the Hausdorff-moment problem (see for instance \cite{Talenti87}), $F_x(z)=0$ for all $z\in (0,1)$, and thus,  $f(x+y)=0$ for $y>1$. As a consequence, $f=0$ in $\R_+=U_+$.
\end{proof}

Using the observations from \cite{GFR20}, it is further possible to turn the uniqueness results provided by antilocality into a conditional stability estimate:

\begin{lem}
Let $s\in (0,1)$ and let $\Omega,\ U\subset\R$ be  open, bounded intervals with $\Omega$ to the right of $U$. Then there exists a constant $C>1$ (depending on $\Omega$, $U$ and $s$) such that  for any  $f\in C^\infty_c(\Omega)$
it holds
\begin{align*}
    \|f\|_{L^2(\Omega)}\leq C e^{C\frac{\|f\|_{H^1(\Omega)}}{\|f\|_{L^2(\Omega)}}}\|A^s_0(D) f\|_{L^2(U)}. 
\end{align*}
\end{lem}

\begin{proof}
Without loss of generality, by scaling and translation, we may assume $U= (0, \beta)$, with $\beta<1$, and $\Omega \subset (2,+\infty)$.
Let $x\in U$ and consider the function $F_x(z)=z^{2s-1}f\big(x+\frac 1z\big)$, with $\supp (F_x)\subset I=\big(0,\frac{1}{2-x}\big)$.

Arguing as in the previous lemma, we observe that the  moments of $F_x$ are given by 
\begin{align*}
    F^k_x & = c_{s,k}^{-1}\frac{d^k}{dx^k} \big(A^s_0(D) f\big)(x),
\end{align*}
where $c_{s,k}=\frac{\Gamma(1+2s+k)}{\Gamma(1+2s)}$.
In addition, we have
\begin{align*}
    \|F_x\|_{L^2(I)}^2
    &=\int_{2-x}^\infty \frac{f(x+y)^2}{y^{4s}}dy,\\
    \|F_x'\|_{L^2(I)}^2
    &=\int_{2-x}^\infty \big( (2s-1){f(x+y)}{y^{1-2s}}-{f'(x+y)}{y^{2-2s}}\big)^2dy.
\end{align*}

Since $F_x\in H^1(I)$, by \cite[Lemma 2.4]{GFR20}, there is $C>1$ such that
\begin{align*}
    \|F_x\|_{L^2(I)}^2 & \leq C\frac{ e^{C(N_{F_x}+1)}}{\min_{k\leq N_{F_x}} \frac{c_{s,k}}{k!}}\Big\|\sum_{k=0}^\infty \frac{c_{s,k}}{k!} F_x^k z^k\Big\|_{L^2(I)}^2,
\end{align*}
where $N_{F_x}+1\simeq \frac{\|F_x'\|_{L^2(I)}}{\|F_x\|_{L^2(I)}}$.
Using  the previous identities, the compact support of $f$, the fact that $\frac{c_{s,k}}{k!}$ is an increasing sequence {and considering only $x\in U$}, we infer 
\begin{align}\label{eq:moments1}
    \|f\|_{L^2(\Omega)}^2 & \leq C e^{C(N_{f}+1)}\Big\|\sum_{k=0}^\infty \frac{c_{s,k}}{k!} F_x^k z^k\Big\|_{L^2(I)}^2,
\end{align}
where $N_{f}+1 \leq C\frac{\|f\|_{H^1(\Omega)}}{\|f\|_{L^2(\Omega)}}$.

In order to relate the right hand side of \eqref{eq:moments1} to $A^s_0(D)f$ in $U$,  we notice that $A^s_0(D)f$ is real analytic in $\R\backslash\supp (f)$  and   for any $y\in (0,1)$
\begin{align*}
    A^s_0(D) f(y)=\sum_{k=0}^\infty \frac{c_{s,k}}{k!}F^k_x(y-x)^k.
\end{align*}
Hence, 
\begin{align*}
    \|A^s_0(D) f\|_{L^2(U)}^2
    &=\int_0^\beta \Big(\sum_{k=0}^\infty \frac{c_{s,k}}{k!}{F^k_x}(y-x)^k\Big)^2 dy
    \\&=\int_{-x}^{\beta-x} \Big(\sum_{k=0}^\infty \frac{c_{s,k}}{k!}{F^k_x}z^k\Big)^2 dz.
\end{align*}

If $I\subset U-x$ (i.e. $\frac{1}{2-x}<\beta-x$), which holds provided $\beta>\frac 12$ and $x$ is taken small enough,  we conclude
\begin{align*}
    \Big\|\sum_{k=0}^\infty \frac{c_{s,k}}{k!} F_x^k z^k\Big\|_{L^2(I)}\leq \|A^s_0(D) f\|_{L^2(U)}.
\end{align*}
Otherwise, we invoke quantitative analytic continuation. 
Indeed, we seek to apply \cite[Remark 2.6]{GFR20}, based on  quantitative analytic continuation results from  \cite{Vessella99, AE13}.
For any  $y\in (0,1)$
\begin{align*}
    \Big|\frac {d^k}{d y ^k} \big(A^s_0 (D)f\big)(y)\Big| \leq
    c_{s,k}\int_1^\infty |f(y+y')|dy'
    \leq (k+3)!\|f\|_{L^2(\Omega)}
    \leq k! e^{3k}\|f\|_{L^2(\Omega)}. 
\end{align*}
Then, there is $C>1$ and $\theta\in(0,1)$ such that
\begin{align*}
    \|A^s_0(D) f\|_{L^2((0,1))}\leq C \|f\|_{L^2(\Omega)}^{1-\theta}\|A^s_0(D) f\|_{L^2(U)}^\theta.
\end{align*}
Taking $x<\frac 12$, it holds that $I\subset (-x,1-x)$, and therefore we can conclude 
\begin{align*}
    \Big\|\sum_{k=0}^\infty \frac{c_{s,k}}{k!}  F_x^k z^k\Big\|_{L^2(I)}\leq \|f\|_{L^2(\Omega)}^{1-\theta}\|A^s_0(D) f\|_{L^2(U)}^\theta,
\end{align*}
which together with  \eqref{eq:moments1} implies the desired result.
\end{proof}

We conclude our discussion of the observations by Ishikawa and related ideas by quoting the result from \cite{I89} which provides directional antilocality of our operators also in two dimensions.

\begin{lem}[\cite{I89}, Theorem 2.4]\label{lem:Ishik2D}
Let $s\in (0,1)\backslash\big\{\frac 12\big\}$ and $\Gamma\subset \R^2\backslash\{0\}$ be an open, non-empty, convex cone. Then the operators $A_{0,\Gamma}^s(D)$, $A_{1,\Gamma}^s(D)$ given in \eqref{eq:Aspgamma} for $n=2$ are $\Gamma$-antilocal and $-\Gamma$-antilocal, respectively.
\end{lem}

\subsubsection{Partial antilocality in higher dimensions}

While we postpone a systematic $n$-dimensional discussion of the ideas from \cite{I89} (which focuses on two dimensions and particular choices of $p \in [0,1]$) to a future project, in this section we illustrate that certain weaker antilocality conditions can be derived under appropriate assumptions on the set where $f$ and $A^s_{p,\Gamma}(D)f $ vanish and on the set where $f$ may be supported. This is achieved by a suitable reduction argument.
In Lemma ~\ref{lem:excombpartantil} we present similar directional antilocality results under suitable geometric assumptions. The operators considered there however are not antilocal in general domains. 

Similarly to the one-dimensional setting from above, we introduce the following notation:
\begin{align}
\label{eq:Gammap}
   \Gamma_p=\begin{cases}
    \Gamma & \mbox{ if } p=0,\\
    -\Gamma\cup\Gamma & \mbox{ if } p\in(0,1),\\
    -\Gamma & \mbox{ if } p=1.
   \end{cases} 
\end{align}

Seeking to provide (directional) antilocality results in arbitrary dimensions for the operators from \eqref{eq:Aspgamma}, we first reduce the general property of $\Gamma_p$-antilocality to the knowledge of that for the two special cases $p=0$ (or $p=1$) and $p=\frac 12$.

\begin{lem}\label{lem:reduceantiloc}
Let $s\in(0,1)$ and  $\Gamma\subset \R^n\backslash\{0\}$ be an open, non-empty, convex cone. Assume that 
$A^s_{p,\Gamma}(D)$ is $\Gamma_p$-antilocal for the cases  $p\in\big\{0,\frac 12\big\}$.
Then $A^s_{p,\Gamma}(D)$ is $\Gamma_p$-antilocal for any $p\in[0,1]$.
\end{lem}

\begin{proof}
We first observe that using the identity $A_{1,-\Gamma}^s = -A_{0,\Gamma}^s$, the result for $p=1$ is immediate.
Let thus $p\in(0,1)\backslash\big\{\frac 12\big\}$. 
Let  $f\in C^\infty_c(\R^n)$ be such that  $f=0=A^s_{p,\Gamma}(D)f$ in  $U$ and assume that $f\neq 0$ in $U+\Gamma_p$.  
Notice that $$A^s_{p,\Gamma} (D)f= p A^s_{1,\Gamma}(D)f+(1-p) A^s_{0,\Gamma} (D)f.$$
By $\Gamma_p$-antilocality for $p\in\{0,1\}$, it is not possible that $f\neq 0$ only on $\Gamma+ U$ or on $-\Gamma+ U$. 

For $n=1$, we have $f=f_++f_-$ with $f_\pm= f|_{U+\R_\pm}\in C^\infty_c(\R)$.  
By the assumptions for $f$, it holds that $A^s_{\sfrac 12}(D)g=A^s_{p}(D) f=0$ in $U$ for $g:={2}{p}f_- +{2}{(1-p)}f_+\in C^\infty_c(\R)$.
The antilocality of $A^s_{\sfrac 12}(D)$ now implies $g=0$ in $\R$, and therefore $f=0$ in $\R$.

The higher dimensional case holds similarly but requires a bit more care, since we cannot simply split $f$ into two contributions while preserving its smoothness. 
Instead, we start by considering an open subset $U_0\subset U$ such that $(U_0+\Gamma)\cap (U_0-\Gamma)\subset U$ (see Figure ~\ref{figure:settingreduceantiloc}).
Further, let $\eta$ be a smooth function such that $\eta=1$ in $(U_0+\Gamma)\backslash\overline{U}$
and $\eta=0$ in $(U_0-\Gamma)\backslash\overline{U}$.
Let $f_+=\eta f$ and $f_-=(1-\eta)f$.
Then, $f_\pm\in C^\infty_c(\R^n)$ and $f|_{U_0\pm\Gamma}=f_\pm$, so we can argue as above.
We  iterate this argument in subsets $U_j\subset U$ provided that   $(U_j+\Gamma)\cap (U_j-\Gamma)\subset U\cup \big( U_{j-1}+(-\Gamma\cup\Gamma)\big)$, where we already know that $f=0$, until $U$ is completely covered.
\end{proof}

\begin{figure}[t]
\begin{tikzpicture}

\pgfmathsetmacro{\anglesym}{36}
\pgfmathsetmacro{\angleone}{-\anglesym}
\pgfmathsetmacro{\angletwo}{\anglesym}
\pgfmathsetmacro{\angletwox}{cos(\angletwo)}
\pgfmathsetmacro{\angletwoy}{sin(\angletwo)}

\pgfmathsetmacro{\r}{1.1}
\pgfmathsetmacro{\R}{1.5}

\begin{scope}[scale=.8]

\clip (-4, -2.2) rectangle (4,3);

\pgfmathsetmacro{\xx}{3}
\pgfmathsetmacro{\xxx}{4}
\pgfmathsetmacro{\int}{(-\r*sin(\angletwo+90)+\r*sin(\angleone+90)+sin(\angleone)/cos(\angleone)*\r*(cos(\angletwo+90)-cos(\angleone+90)))/(sin(\angletwo)-sin(\angleone)*cos(\angletwo)/cos(\angleone))}

\pgfmathsetmacro{\sep}{\R/cos(\anglesym)}


\draw[orange!50,dashed,line width=.7pt, domain=-0:\xx] plot ({\x*cos(\angletwo)},{\sep+\x * sin(\angletwo)});

\draw[orange!50,dashed,line width=.7pt, domain=-\xx:0] plot ({\x*cos(\angleone)},{\sep+\x * sin(\angleone)});

\draw[orange!50,dashed,line width=.7pt, domain=-\xx:0] plot ({\x*cos(\angletwo)},{-\sep+\x * sin(\angletwo)});

\draw[orange!50,dashed,line width=.7pt, domain=-0:\xx] plot ({\x*cos(\angleone)},{-\sep+\x * sin(\angleone)});

\fill[orange!20, opacity=.8] (0,0) circle (\R);

\node[orange] at (0,0) {$U$};


\draw[blue!50,dashed,line width=.7pt, domain=-0:\xxx] plot ({\r*cos(\angletwo+90)+\x*cos(\angletwo)},{\r*sin(\angletwo+90)+\x * sin(\angletwo)});
\draw[blue!50,dashed,line width=.7pt, domain=-\xxx:0] plot ({\r*cos(\angleone+90)+\x*cos(\angleone)},{\r*sin(\angleone+90)+\x * sin(\angleone)});
\draw[blue!50,dashed,line width=.7pt, domain=-\xxx:0] plot ({\r*cos(\angletwo-90)+\x*cos(\angletwo)},{\r*sin(\angletwo-90)+\x * sin(\angletwo)});
\draw[blue!50,dashed,line width=.7pt, domain=-0:\xxx] plot ({\r*cos(\angleone-90)+\x*cos(\angleone)},{\r*sin(\angleone-90)+\x * sin(\angleone)});

\pgfmathsetmacro{\Wx}{4}
\pgfmathsetmacro{\Wy}{0}

\fill[blue!20, opacity=.1]
({\r*cos(\angletwo+90)+\xxx*cos(\angletwo)},{\r*sin(\angletwo+90)+\xxx * sin(\angletwo)}) 
--({\r*cos(\angletwo+90)+\int*cos(\angletwo)},{\r*sin(\angletwo+90)+\int * sin(\angletwo)})
 --({\r*cos(\angleone+90)-\xxx*cos(\angleone)},{\r*sin(\angleone+90)-\xxx * sin(\angleone)})
 .. controls (-\Wx,\Wy) ..({\r*cos(\angletwo-90)-\xxx*cos(\angletwo)},{\r*sin(\angletwo-90)-\xxx * sin(\angletwo)})
--({\r*cos(\angletwo-90)-\int*cos(\angletwo)},{\r*sin(\angletwo-90)-\int * sin(\angletwo)})
--({\r*cos(\angleone-90)+\xxx*cos(\angleone)},{\r*sin(\angleone-90)+\xxx * sin(\angleone)}) 
.. controls (\Wx,\Wy) ..
({\r*cos(\angletwo+90)+\xxx*cos(\angletwo)},{\r*sin(\angletwo+90)+\xxx * sin(\angletwo)});

\draw[blue!60, line width=.7 pt] (0,0) circle [radius=\r,];

\node[blue!60] at ({1.3*(\r)*\angletwox}, {1.3*(\r)*\angletwoy}) {$U_0$};

\node[blue!50] at ({3*(\r)*\angletwox}, {0}) {$U_0+\Gamma$};
\node[blue!50] at ({-3*(\r)*\angletwox}, {0}) {$U_0-\Gamma$};

\begin{scope}[xshift=-.4cm, yshift=1.1cm]
\pgfmathsetmacro{\rr}{.3}
\pgfmathsetmacro{\xxb}{3.7}

\draw[red!80,opacity=.6, densely dotted,line width=.7pt, domain=-0:\xxb] plot ({\rr*cos(\angletwo+90)+\x*cos(\angletwo)},{\rr*sin(\angletwo+90)+\x * sin(\angletwo)});
\draw[red!80,opacity=.6,densely dotted,line width=.7pt, domain=-\xx:0] plot ({\rr*cos(\angleone+90)+\x*cos(\angleone)},{\rr*sin(\angleone+90)+\x * sin(\angleone)});
\draw[red!80,opacity=.6,densely dotted,line width=.7pt, domain=-\xx:0] plot ({\rr*cos(\angletwo-90)+\x*cos(\angletwo)},{\rr*sin(\angletwo-90)+\x * sin(\angletwo)});
\draw[red!80,opacity=.6,densely dotted,line width=.7pt, domain=-0:\xxb] plot ({\rr*cos(\angleone-90)+\x*cos(\angleone)},{\rr*sin(\angleone-90)+\x * sin(\angleone)});

\draw[red!80,opacity=.6, line width=.7 pt] (0,0) circle [radius=\rr,];

\node[red!80,opacity=.6] at (-{2*(\rr)*\angletwox}, {2*(\rr)*\angletwoy}) {\footnotesize $U_1$};

\end{scope}

\end{scope}

\end{tikzpicture}
\caption{Setting for deducing the $\Gamma_p$-antilocality for $p\in(0,1)$ in the proof of Lemma ~\ref{lem:reduceantiloc}. The subset $U_j$ is chosen such that we  already know that $f$ vanishes in the intersection of the cones. }
\label{figure:settingreduceantiloc}
\end{figure}
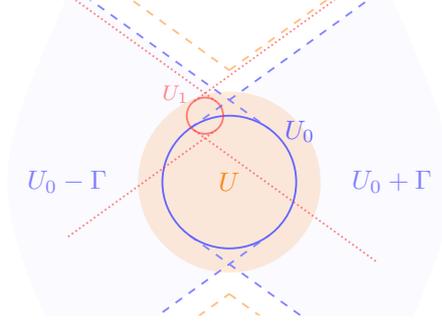

\begin{lem}\label{lem:partantiloc}
Let $s\in (0,1)$, $p\in[0,1]$ and  $\Gamma\subset \R^n\backslash\{0\}$ be an open, non-empty, convex cone.
Let $\Omega\subset\R^n$ and $U\subset \R^n\backslash \overline \Omega$ be connected and assume that
\begin{align*}
    \Omega\cap (U+\Gamma_p) \Subset \bigcap_{x\in U} (x+\Gamma_p).
\end{align*} 
Let $f\in C^\infty_c(\Omega)$ be such that  $A^s_{p,\Gamma}(D)f=0$ in $U$. Then $f=0$ in $\Omega\cap(U+\Gamma_p)$.
\end{lem}

In the following, we will refer to antilocality results of the type as in Lemma \ref{lem:partantiloc} by saying that the operator $A_{p,\Gamma}^s(D)$ is \emph{$\Gamma_p$-antilocal from $U$ to $\Omega$}.

We refer to Figure ~\ref{figure:setting}  for an illustration of the assumptions on the subsets $\Omega$ and $U$.
Notice that the geometric assumption implies that
\begin{align*}
    U\subset \Omega+\Gamma_{1-p}.
\end{align*}

\begin{proof}
    By Lemma ~\ref{lem:reduceantiloc}, it is enough to prove the result for $p=0$ and  $p=\frac 12$.
    Consider $f \in C_c^{\infty}(\Omega)$ and let $\tilde f_p=\mathcal E_{U+\Gamma_p}f$ be the extension of $f$ by zero outside of $U+\Gamma_p$. By assumption, it satisfies $ \tilde f_p\in C^\infty_c(U+\Gamma_p)$. 
    
    We notice that by the assumptions for $p=\frac 12$ it holds that  
    \begin{align*}
    A_{\sfrac 12, \Gamma}^s(D) f(x) =c_{n,s} (-\Delta)^s \tilde f_{\sfrac 12}(x) \mbox{ for } x \in U,
    \end{align*} 
    and for $p=0$ we have that
    \begin{align*}
    A_{0, \Gamma}^s(D) f(x) = 2c_{n,s} (-\Delta)^s\tilde f_{0}(x) \mbox{ for } x \in U.
    \end{align*}
     This follows from the fact that $ \tilde f_p=0$ outside $x+\Gamma_p$ for all $x\in U$. As a consequence, in this geometric setting, for $p\in \big\{0,\frac{1}{2}\big\}$ it holds that $(-\D)^{s}\tilde f_p(x)=0$ and $\tilde f_p(x)$ for $x\in U$. Then, applying the antilocality of the fractional Laplacian, we infer $\tilde f_p=0$ in $\R^n$ and therefore $f=0$ in $U+\Gamma_p$ for $p\in \big\{0,\frac{1}{2}\big\}$. Hence, $A_{\sfrac 12, \Gamma}^s(D) $ and $A_{0, \Gamma}^s(D)$ are directionally antilocal from $U$ to $\Omega$.
\end{proof}

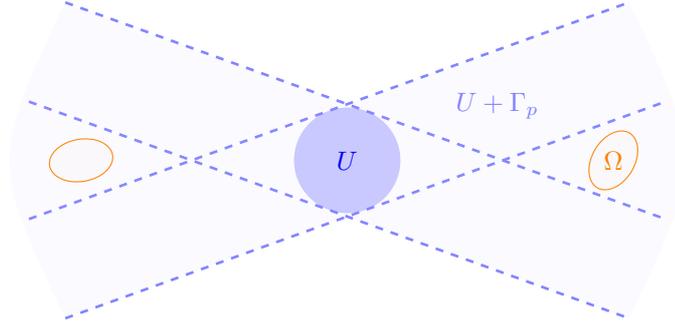
\begin{figure}[t]
\begin{tikzpicture}

\pgfmathsetmacro{\angleone}{-20}
\pgfmathsetmacro{\angletwo}{20}

\pgfmathsetmacro{\angletwox}{cos(\angletwo)}
\pgfmathsetmacro{\angletwoy}{sin(\angletwo)}

\begin{scope}[scale=.7]

\pgfmathsetmacro{\r}{1}

\pgfmathsetmacro{\xx}{6}
\pgfmathsetmacro{\int}{(-\r*sin(\angletwo+90)+\r*sin(\angleone+90)+sin(\angleone)/cos(\angleone)*\r*(cos(\angletwo+90)-cos(\angleone+90)))/(sin(\angletwo)-sin(\angleone)*cos(\angletwo)/cos(\angleone))}

\draw[blue!50,dashed,line width=1pt, domain=-\xx:\xx] plot ({\r*cos(\angletwo+90)+\x*cos(\angletwo)},{\r*sin(\angletwo+90)+\x * sin(\angletwo)});
\draw[blue!50,dashed,line width=1pt, domain=-\xx:\xx] plot ({\r*cos(\angleone+90)+\x*cos(\angleone)},{\r*sin(\angleone+90)+\x * sin(\angleone)});
\draw[blue!50,dashed,line width=1pt, domain=-\xx:\xx] plot ({\r*cos(\angletwo-90)+\x*cos(\angletwo)},{\r*sin(\angletwo-90)+\x * sin(\angletwo)});
\draw[blue!50,dashed,line width=1pt, domain=-\xx:\xx] plot ({\r*cos(\angleone-90)+\x*cos(\angleone)},{\r*sin(\angleone-90)+\x * sin(\angleone)});

\pgfmathsetmacro{\Wx}{5}
\pgfmathsetmacro{\Wy}{0}

\fill[blue!20, opacity=.1]
({\r*cos(\angletwo+90)+\xx*cos(\angletwo)},{\r*sin(\angletwo+90)+\xx * sin(\angletwo)}) 
--({\r*cos(\angletwo+90)+\int*cos(\angletwo)},{\r*sin(\angletwo+90)+\int * sin(\angletwo)})
 --({\r*cos(\angleone+90)-\xx*cos(\angleone)},{\r*sin(\angleone+90)-\xx * sin(\angleone)})
 .. controls (-6.7,\Wy) ..({\r*cos(\angletwo-90)-\xx*cos(\angletwo)},{\r*sin(\angletwo-90)-\xx * sin(\angletwo)})
--({\r*cos(\angletwo-90)-\int*cos(\angletwo)},{\r*sin(\angletwo-90)-\int * sin(\angletwo)})
--({\r*cos(\angleone-90)+\xx*cos(\angleone)},{\r*sin(\angleone-90)+\xx * sin(\angleone)}) 
.. controls (6.7,\Wy) ..
({\r*cos(\angletwo+90)+\xx*cos(\angletwo)},{\r*sin(\angletwo+90)+\xx * sin(\angletwo)});

\draw[orange] (\Wx,\Wy) ellipse [x radius=.4cm,y radius=.6cm, rotate=-30];
\draw[orange] (-\Wx,\Wy) ellipse [x radius=.6cm,y radius=.4cm, rotate=10];
\fill[orange!20, opacity=.1] (\Wx,\Wy) ellipse [x radius=.4cm,y radius=.6cm, rotate=-30];
\fill[orange!20, opacity=0.1] (-\Wx,\Wy) ellipse [x radius=.6cm,y radius=.4cm, rotate=10];
\node[orange] at (\Wx,\Wy) {$\Omega$};

\fill[blue!50, opacity=.4] (0,0) circle [radius=\r,];
\node[blue] at (0,0) {$U$};

\node[blue!50] at ({3*(\r)*\angletwox}, {3*(\r)*\angletwoy}) {$U+\Gamma_p$};

\end{scope}
\end{tikzpicture}
\caption{An example of the setting for the partial antilocality results for $A^s_{p,\Gamma}$ with $p\in(0,1)$ (see Lemma ~\ref{lem:partantiloc}). We assume that   $f\in C^\infty_c(\R^n)$ is supported  in $\Omega$, which, in turn, is contained  in the   intersection of the cones $x+\Gamma_p$ for $x\in U$. Then we deduce $f=0$ in $U+\Gamma_p$, which consists of the blue areas.}
\label{figure:setting}
\end{figure}

Last but not least, we observe that along the same lines as in the proof of Proposition ~\ref{prop:exterior} we obtain the following antilocality from the exterior result:

\begin{prop}[Exterior antilocality]
\label{prop:exterior2}
Let $s\in (0,1)$, $p\in[0,1]$, $\Gamma\subset \R^2\backslash\{0\}$ be an open, non-empty, convex cone,   $A^s_{p,\Gamma}(D)$ be as in \eqref{eq:Aspgamma} and let $f\in \widetilde H^{r}(\Omega)$, $r\in \R$.
If $ A^s_{p,\Gamma}(D)f=0$ in $(\Omega+\Gamma_{1-p})\backslash \overline{\Omega}$, then $f \equiv 0$ in $\Omega$. 
\end{prop}

\begin{proof}
This follows along the same lines as the Step 2 of the  proof of  Proposition ~\ref{prop:exterior}.
We just need to observe that because of the domain of dependence of the operator, it holds that $A^s_{p,\Gamma}(D)f=0$ also in $\R^n\backslash(\overline{\Omega}+\Gamma_{1-p})$.
\end{proof}

\subsection{The direct problem: Well-posedness}
\label{sec:well_posedAsp}

In this section we study the direct problem associated with $A^s_{p,\Gamma}(D)$ for any $p\in[0,1]$, any $s\in (0,1)\backslash \big\{\frac{1}{2}\big\}$ and any dimension $n\in \N$. Since the operators from \eqref{eq:Aspgamma} are in general not symmetric, additional features arise compared to our discussion in Section ~\ref{sec:direct}. This in particular involves our choice of the bilinear forms and function spaces in the cases $p\in\{0,1\}$, in which different parts of $\partial \Omega$ play different roles (depending on whether they are contained in $\Omega + \Gamma_p$ or the remainder of $\partial \Omega$). For $s> \frac{1}{2}$ and $p\in \{0,1\}$ these are of particular interest, since they allow for both local and nonlocal data.

As in Section ~\ref{sec:direct}, well-posedness could in principle be discussed in various slightly different settings. One natural setting corresponds to a  Fourier space definition of the operators under consideration. This is associated with the bilinear form $\tilde{B}^s_{p,\Gamma;q}$ defined below. Alternatively, one could consider a ``more local" interpretation in the spirit of the bilinear form $B^s_{q}$ from Section ~\ref{sec:direct}. Since the more local definition requires dealing with additional technical difficulties, we only discuss the Fourier version here but complement this discussion by the investigation of the local version in Appendix ~\ref{sec:loc_B}.

Throughout this section, we assume that $\Omega\subset \R^n$ is an open, bounded Lipschitz domain.

\subsubsection{The bilinear form}
We begin by providing a suitable weak formulation of our problem and by introducing the associated bilinear form. Similarly as in the discussion of symmetric operators, we prescribe data only in the region which the operators ``sees", and thus study the weak form of the following equation
\begin{align}\label{eq:AspgammaExt}
\begin{split}
    \big(A^s_{p,\Gamma}(D)+q\big) u&=0 \;\mbox{ in } \Omega,\\
    u&=f \; \mbox{ in } (\Omega+\Gamma_p)\backslash\overline{\Omega}.
\end{split}
\end{align}
Here $q$ is a real-valued, bounded, measurable function on $\Omega$ and $\Gamma_p$ is given in \eqref{eq:Gammap}.

In parallel to our discussion from Section ~\ref{sec:direct}, we introduce the following bilinear form for $s\in(0,1)\backslash\big\{\frac 12\big\}$, which is motivated by a Fourier version of the equation and which is analogous to \eqref{eq:bilR}: For $u,v \in H^s(\R^n)$ we set
\begin{align}\label{eq:tBspgammaq}
\begin{split}
    \tilde B^s_{p,\Gamma;q}(u,v)&:=\beta_s
    \int_{\R^n}\int_{\Gamma_p}
    \big(u(x)-u(x+y)\big)\big(v(x)-v(x-y)\big)\nu^s_{p,\Gamma}(y) dydx
     +(qu,v)_{L^2(\Omega)},
\end{split}
\end{align}
where $\beta_s=(2-2^{2s})^{-1}$ and $\nu^s_{p,\Gamma}$ is given  in \eqref{eq:nuspgamma}.
Compared to the bilinear form $B^s_{p,\Gamma;q}$ from \eqref{eq:Bspgammaq} in Appendix ~\ref{sec:loc_B}, it corresponds to a ``more global'' version of the operator from \eqref{eq:AspgammaExt} (in a sense the analogue to the bilinear form $\tilde{B}_q$ from Section \ref{sec:direct}). 
More precisely, we note that for $u,v \in H^s(\R^n)$
\begin{align}
\label{eq:Fourier_nonlocal}
    \tilde B^s_{p,\Gamma;q}(u,v)=\left(A^s_{p,\Gamma}\left(\frac{\xi}{|\xi|}\right)|\xi|^{s}\hat u,\ |\xi|^s\hat v\right)_{L^2(\R^n)}+(qu, v)_{L^2(\Omega)},
\end{align}
with the symbol $A^s_{p,\Gamma}(\xi)$ as in \eqref{eq:Aspgammasymb}. Indeed, this follows from Fourier-transforming the bilinear form $\tilde B_q$: Arguing similarly as in Remark ~\ref{rmk:consistent1} (which also gives rise to the normalizing prefactor $\beta_s$), for $u,v \in C^\infty_c(\R^n)$
\begin{align*}
    \tilde  B^s_{p,\Gamma;q}(u,v)
    &=\big(A^s_{p,\Gamma;q}(D) u, v\big)_{L^2(\R^n)}+(qu, v)_{L^2(\Omega)}\\
    &=\big(A^s_{p,\Gamma}(\xi) \hat u, \hat v\big)_{L^2(\R^n)}+(qu, v)_{L^2(\Omega)}.
\end{align*}
The claimed identity \eqref{eq:Fourier_nonlocal} is a consequence of the fact that $A^s_{p,\Gamma}(\xi)=A^s_{p,\Gamma}\big(\frac{\xi}{|\xi|}\big)|\xi|^{2s}$ for the symbol from \eqref{eq:Aspgammasymb}.

We note that in one dimension and for $p\in \{0,1\}$, the Fourier based representation from \eqref{eq:Fourier_nonlocal} can be formulated more symmetrically using the adjointness properties of the operators $A^s_{p}(D)$ and $A^{s}_{1-p}(D)$. 

\begin{rmk}[Alternative factorization of $\tilde{B}^s_{p, \Gamma;q}$, one dimension, $p\in \{0,1\}$]
We note that in one dimension and for $p\in\{0,1\}$ and for $s \in (0,1)\backslash \{\frac{1}{2}\}$º, we could also write
\begin{align*}
    \tilde B^s_{p;q}(u,v):=\frac{c_s}{c_{\sfrac s2}^2}\left(A^{\sfrac s2}_{p}\hat u, \ A^{\sfrac s2}_{1-p}\hat v\right)_{L^2(\R^n)}+(qu, v)_{L^2(\Omega)}.
\end{align*}
This formulation is based on the identity
\begin{align*}
    \Big(\cos\big(\frac{\pi s}{2}\big)\pm i\sin \big(\frac{\pi s}{2}\big)\sign(\xi)\Big)^2
    &=\cos^2\big(\frac{\pi s}{2}\big)-\sin^2\big(\frac{\pi s}{2}\big)\pm i2\cos\big(\frac{\pi s}{2}\big)\sin\big(\frac{\pi s}{2}\big)\sign(\xi)
    \\&=\cos(\pi s)\pm i\sin(\pi s)\sign(\xi).
\end{align*}

The case $p=\frac 12$ (for any dimension) is included in Section ~\ref{sec:direct} (with  $\mathcal C=\Gamma$ and $a=\frac 12\chi_{-\Gamma\cup\Gamma}$), for which, since $A^s_{\sfrac 12,\Gamma}(\xi)$ is real-valued, we can take
\begin{align*}
    \tilde B^s_{\sfrac 12,\Gamma;q}(u,v):=\left( (A^s_{\sfrac 12,\Gamma}(\xi))^{\frac 12}\hat u, (A^s_{\sfrac 12,\Gamma}(\xi))^{\frac 12} \hat v\right)_{L^2(\R^n)}+(qu, v)_{L^2(\Omega)}.
\end{align*}
\end{rmk}

\begin{rmk}[The case $s=\frac{1}{2}$, $p\neq \frac{1}{2}$]
If $s=\frac 12$ and $p\neq \frac 12$, there is an extra term in the definition of the bilinear form originating from the logarithm in the symbol of $A^{s}_{p,\Gamma}$:
\begin{align*}
    \tilde B^{\sfrac 12}_{p,\Gamma;q}(u,v)
    &:=\left(A^{\sfrac 12}_p\left(\frac{\xi}{|\xi|}\right)|\xi|^{\frac 12}\hat u, |\xi|^{\frac 12}\hat v\right)_{L^2(\R^n)}+(qu, v)_{L^2(\Omega)}
    \\&\quad +\left(\alpha_{p,\Gamma}\left(\frac{\xi}{|\xi|}\right)\log(|\xi|)|\xi|^{\frac 12}\hat u, |\xi|^{\frac 12}\hat v\right)_{L^2(\R^n)},
\end{align*}
where $\alpha_{p,\Gamma}(\xi)=-i(1-2p)\int_{\Gamma\cap\Ss^{n-1}}(\theta\cdot\xi)d\theta$.
This term does not verify the continuity properties (see the proof of Proposition ~\ref{prop:AspgammaInh}). We do not discuss this case in the sequel. Since the case $s=\frac 12$ and $p=\frac 12$ is already included in Sections ~\ref{sec:direct} and ~\ref{sec:inv_first}, in the remaining of this section we only consider $s\in(0,1)\backslash\big\{\frac 12\big\}$. 
\end{rmk}

\subsubsection{The interior problem}
Arguing in parallel to the discussion of the symmetric setting from Section ~\ref{sec:direct}, we first study the well-posedness of the problem in the presence of an interior source term: 

\begin{align}\label{eq:AspgammaInh}
\begin{split}
    \big(A^s_{p,\Gamma}(D)+q\big) u&=g \;\mbox{ in } \Omega,\\
    u&=0 \; \mbox{ in } (\Omega+\Gamma_p)\backslash\overline{\Omega}.
\end{split}
\end{align}

In this context, we will work with the following set-up:

\begin{defi}
Let $s\in (0,1)\backslash\big\{\frac 12\big\}$, $p\in[0,1]$ and let $\tilde B^s_{p,\Gamma;q}$ be the bilinear form from \eqref{eq:tBspgammaq}.  For $g\in H^{-s}(\Omega)$, a function $u\in \widetilde H^s(\Omega)$  is a \emph{(weak) solution to \eqref{eq:AspgammaInh}}
if 
\begin{align*}
    \tilde B^s_{p,\Gamma;q}(u,v)=\langle g,v \rangle \ \mbox{ for all } v\in \widetilde H^s(\Omega).
\end{align*}
\end{defi}

Given this notion, the well-posedness of the problem \eqref{eq:AspgammaInh} follows along similar lines as in the symmetric (and in the whole space) settings.

\begin{prop}\label{prop:AspgammaInh}
Let $s\in(0,1)\backslash\big\{\frac 12\big\}$, $p\in[0,1]$, $\Omega\subset\R^n$ be a bounded, Lipschitz open set,  $\Gamma\subset \R^n\backslash\{0\}$ be an open, non-empty, convex cone and  $q\in L^\infty(\Omega)$. Then there is a countable set $\tilde\Sigma^s_{p,\Gamma;q}\subset \C$ such that
if $\lambda\notin\tilde \Sigma^s_{p,\Gamma;q}$, for any $g\in H^{-s}(\Omega)$, there is a unique (weak) solution $u\in\widetilde H^s(\Omega)$ of 
\begin{align}
\label{eq:AspgammaInhlambda}
    \begin{split}
        \big(A^s_{p,\Gamma}(D)+q-\lambda\big)u &=g \;\mbox{ in }\Omega,\\
        u&=0 \; \mbox{ in } (\Omega+\Gamma_p)\backslash\overline{\Omega}.
    \end{split}
\end{align}
In addition, the solution satisfies
\begin{align*}
    \|u\|_{H^s(\R^n)}\leq C\|g\|_{H^{-s}(\Omega)}.
\end{align*}
\end{prop}

\begin{proof}
The proof follows along the same lines of Proposition ~\ref{prop:inhom_int}. 
Let $\gamma=\|q_-\|_{L^\infty(\Omega)}$, where $q_-(x):=\min\{0, q(x)\}$. 
We first prove that the bilinear form $\tilde B^s_{p,\Gamma;q+\gamma}$ is continuous and coercive in $\widetilde H^s(\Omega)$. 
Indeed, we notice that for $s\neq \frac 12$, $\dot\xi\in\Ss^{n-1}$
\begin{align*}
    |A^s_{p,\Gamma}(\dot\xi)|
    \leq |\Gamma(-2s)| \int_{\Gamma\cap\Ss^{n-1}}|\theta\cdot\dot\xi|^{2s}d\theta
    \leq |\Gamma(-2s)| |\Gamma\cap\Ss^{n-1}|
    \leq C_s.
\end{align*}
Therefore, for any $u, v\in\widetilde H^s(\Omega)$  it holds
\begin{align}\label{eq:conttBspgamma}
\begin{split}
    |\tilde B^s_{p,\Gamma;q}(u,v)|
    &\leq C_s\||\cdot|^s\hat u\|_{L^2(\R^n)}\||\cdot|^s\hat v\|_{L^2(\R^n)}+\|q\|_{L^\infty(\Omega)}\|u\|_{L^2(\R^n)}\|v\|_{L^2(\R^n)}\\
    &\leq \tilde C_s\|u\|_{H^s(\R^n)}\|v\|_{H^s(\R^n)}.
\end{split}
\end{align}

In order to prove coercivity, it suffices to consider $\tilde B^s_{p,\Gamma;0}$. We first observe  that   for $\dot\xi\in\Ss^{n-1}$
\begin{align*}
    \text{Re } A^s_{p,\Gamma}(\dot\xi)=-\Gamma(-2s) \cos(\pi s) \int_{\Gamma\cap\Ss^{n-1}}|\theta\cdot\dot\xi|^{2s}d\theta.
\end{align*}
Hence, for any $s\in(0,1)$, since $-\Gamma(-2s) \cos(\pi s)>0$, there is a positive constant $C_s$ such that for any $v\in\widetilde H^s(\Omega)$ 
\begin{align*}
   \text{Re } \tilde B^s_{p,\Gamma;0}(v,v)
    =\int_{\R^n} \text{Re } A^s_{p,\Gamma}(\dot\xi) \big(|\xi|^s\hat v(\xi)\big)^2d\xi
    \geq C_s\|v\|_{\dot H^s(\R^n)}^2.
\end{align*}
Finally, by the Poincar\'e inequality, for $v \in \widetilde{H}^s(\Omega)$,
\begin{align*}
   \text{Re } \tilde B^s_{p,\Gamma;0}(v,v)
    \geq C_s\|v\|_{H^s(\R^n)}^2.
\end{align*}

Therefore there is a unique $u=K_{p,\Gamma}g\in\widetilde H^s(\Omega)$ satisfying 
\begin{align*}
    \tilde B^s_{p,\Gamma;q}(u,v)+\gamma(u,v)_{L^2(\R^n)}=\langle g,v \rangle \mbox{ for all } v\in \widetilde H^s(\Omega)
\end{align*}
and
$\|u\|_{H^s(\R^n)}\leq C\|g\|_{H^{-s}(\Omega)}$.
By the compactness of the inclusion $\widetilde H^s(\Omega)\hookrightarrow L^s(\Omega)$, we infer that $K_{p,\Gamma}:L^2(\Omega)\to L^2(\Omega)$ is compact. By the Fredholm alternative, the claim follows.
\end{proof}

Next we seek to study the well-posedness of the exterior problem \eqref{eq:AspgammaExt}.
While it is possible to study its well-posedness in spaces in the spirit of  $V^s(\Omega,a)$, we here focus on the more symmetric setting of Sobolev spaces. A discussion in spaces of the type $V^s(\Omega,a)$ for the bilinear form $B^s_{p,\Gamma;q}$ is carried out in Section ~\ref{sec:app_functions} in  Appendix ~\ref{sec:loc_B}.

\subsubsection{Function spaces for the exterior problem}
\label{sec:functionspacesA}

In our discussion of the exterior problem, features of the (possibly) only one-sided domain of dependence structure of the operator now play a significant role and give rise to differences with respect to the setting of symmetric operators (with two-sided domain of dependence structures). Indeed, in order to take the different parts of $\partial \Omega$ into account in the case $p\in \{0,1\}$ (one part contained in $\Omega+\Gamma_p$ while the other is not, see Figure ~\ref{figure:onecone} for an example of such sets in the case $p=0$), we will choose boundary data and corresponding solutions which respect these differences. In particular, for $s>\frac{1}{2}$ and $p \in \{0,1\}$ this will include both local and nonlocal data.

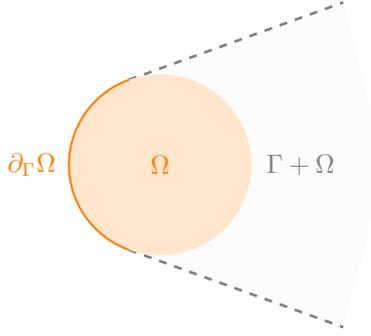
\begin{figure}[t]
\begin{tikzpicture}

\pgfmathsetmacro{\angleone}{-20}
\pgfmathsetmacro{\angletwo}{20}

\pgfmathsetmacro{\anglemid}{(\angleone+\angletwo)/2}
\pgfmathsetmacro{\anglediff}{abs(\angleone-\angletwo)/2}

\pgfmathsetmacro{\midx}{1.4* cos(\anglemid)}
\pgfmathsetmacro{\midy}{1.4* sin(\anglemid)}

\pgfmathsetmacro{\thetax}{cos(\anglemid)}
\pgfmathsetmacro{\thetay}{sin(\anglemid)}
\pgfmathsetmacro{\rtheta}{sin(\anglediff)}

\pgfmathsetmacro{\angonex}{cos(\angleone-10)}
\pgfmathsetmacro{\angoney}{sin(\angleone-10)}

\begin{scope}[scale=.6]

\pgfmathsetmacro{\r}{2}

\pgfmathsetmacro{\xx}{5}
\pgfmathsetmacro{\int}{(-\r*sin(\angletwo+90)+\r*sin(\angleone+90)+sin(\angleone)/cos(\angleone)*\r*(cos(\angletwo+90)-cos(\angleone+90)))/(sin(\angletwo)-sin(\angleone)*cos(\angletwo)/cos(\angleone))}

\fill[black!10, opacity=.1]
({\r*cos(\angletwo+90)+\xx*cos(\angletwo)},{\r*sin(\angletwo+90)+\xx * sin(\angletwo)}) 
--({\r*cos(\angletwo+90)+\int*cos(\angletwo)},{\r*sin(\angletwo+90)+\int * sin(\angletwo)})
--({\r*cos(\angletwo-90)-\int*cos(\angletwo)},{\r*sin(\angletwo-90)-\int * sin(\angletwo)})
--({\r*cos(\angleone-90)+\xx*cos(\angleone)},{\r*sin(\angleone-90)+\xx * sin(\angleone)}) 
.. controls ({\xx*\thetax},{\xx*\thetay}) ..
({\r*cos(\angletwo+90)+\xx*cos(\angletwo)},{\r*sin(\angletwo+90)+\xx * sin(\angletwo)});

\draw[black!50,dashed,line width=1pt, domain=0:\xx] plot ({\r*cos(\angletwo+90)+\x*cos(\angletwo)},{\r*sin(\angletwo+90)+\x * sin(\angletwo)});

\draw[black!50,dashed,line width=1pt, domain=0:\xx] plot ({\r*cos(\angleone-90)+\x*cos(\angleone)},{\r*sin(\angleone-90)+\x * sin(\angleone)});

\fill[orange!20, opacity=1] (0,0) circle [radius=\r,];
\node[orange] at (0,0) {$\Omega$};

\node[black!50] at ({(\r+.2)*\midx}, {(\r+.2)*\midy}) {$\Gamma+\Omega$};

\pgfmathsetmacro{\angleonen}{\angleone+270}
\pgfmathsetmacro{\angletwon}{\angletwo+90}

\draw[orange, line width=.8pt] (\angletwon:\r) arc (\angletwon:\angleonen:\r);

\node[orange] at ({(-\r-.01)*\midx}, {(-\r-.01)*\midy}) {$\p_{\Gamma}\Omega$};

\end{scope}
\end{tikzpicture}
\caption{Sets in Lemma ~\ref{lem:Aspgammadata} for $p=0$, so $\Gamma_p=\Gamma$.}
\label{figure:onecone}
\end{figure}

Relying on the bilinear form introduced in \eqref{eq:tBspgammaq} above, for technical reasons, we again work with corresponding quotient spaces, which in a sense again interpret the problem as a ``global" problem in $\R^n$. To this end, for $s\in (0,1)$ we introduce the following quotient space
\begin{align*}
    H^s_{p,\Gamma}(\Omega)&:=\raisebox{.3em}{$H^s(\R^n)$}
    \Big/\raisebox{-.3em}{$\widetilde H^s(\Omega) \oplus \widetilde H^s(\R^n\backslash(\overline{\Omega}+\Gamma_p))$}.
\end{align*}
$H^s_{p,\Gamma}(\Omega)$ is the space of equivalence classes $\{[f]:f\in H^s(\R^n)\}$, where
\begin{align*}
    [f]=\big\{\tilde f\in H^s(\R^n): \tilde f-f\in \widetilde H^s(\Omega) \oplus \widetilde H^s(\R^n\backslash(\overline{\Omega}+\Gamma_p)) \big\}.
\end{align*}
Here $\tilde f-f\in \widetilde H^s(\Omega) \oplus \widetilde H^s(\R^n\backslash(\overline{\Omega}+\Gamma_p))$ means that there exist $f_1 \in \widetilde{H}^s(\Omega)$ and $f_2 \in \widetilde{H}^s(\R^n \backslash (\overline{\Omega} + \Gamma_p))$ such that $\tilde{f}-f = f_1 + f_2$.
We equip the space $H^{s}_{p,\Gamma}(\Omega)$ with the quotient topology, i.e. for any $f\in H^s_{p,\Gamma}(\Omega)$ 
\begin{align*}
    \|f\|_{H^s_{p,\Gamma}(\Omega)}:=\min\big\{\|\tilde f\|_{H^s(\R^n)}: \ \tilde f\in H^s(\R^n), \ \tilde f\in [f] \big\}.
\end{align*}
 We remark that these spaces inherit the Banach structure from the $H^s(\R^n)$ spaces.
With slight abuse of notation, in the following, and in parallel to the usual notation for $f\in H^s(\Omega)$, we will simply write $f\in H^s_{p,\Gamma}(\Omega)$ and we will only refer to the equivalence class  $[f]$ when necessary.

In order to illustrate the information contained in the space $H^{s}_{p,\Gamma}(\Omega)$, we relate it to more ``standard'' function spaces under suitable assumptions. While strictly speaking it is not necessary in the following sections, this discussion provides important intuition on these rather abstractly defined function spaces. For $s>\frac{1}{2}$ and $p\in \{0,1\}$ we define the following ``trace space'' which combines a ``local'' boundary contribution and ``nonlocal" boundary data:
\begin{align*}
W^{s}_{p,\Gamma}(\Omega)&:=\big\{
  (g, h) \in H^{s-\frac{1}{2}}(\p_{\Gamma_p} \Omega) \times H^{s}((\Omega + \Gamma_p) \backslash \overline{\Omega}):\\
  & \qquad \mbox{ there exists }  \tilde f\in  H^s(\R^n) \mbox{ with } \tilde{f}|_{\p_{\Gamma_p} \Omega} = g, \ \tilde{f}|_{(\Omega + \Gamma_p) \backslash \overline{\Omega}} = h\\
  & \qquad \mbox{ and } \|\tilde{f}\|_{H^{s}(\R^n)} \leq C \|(g,h)\|_{W^{s}_{p,\Gamma}(\Omega)}  \big\},
\end{align*}
where 
$\p_{\Gamma_p} \Omega=\p\Omega \backslash (\Omega+\Gamma_p)$,  $$\|(g,h)\|_{W^{s}_{p,\Gamma}(\Omega)} := \|g\|_{H^{s-\frac{1}{2}}(\partial_{\Gamma_p} \Omega)} + \|h\|_{H^{s}((\Omega + \Gamma_p) \backslash \overline{\Omega})},$$
and where $C>1$ denotes a fixed constant. We remark that the completeness of the space $W^{s}_{p,\Gamma}(\Omega)$ is inherited from the corresponding $H^s$ spaces.

\begin{lem}\label{lem:Aspgammadata}
Let $\Omega\subset\R^n$ be a bounded, Lipschitz open set and  $\Gamma\subset \R^n\backslash\{0\}$ be an open, non-empty, convex cone. The following identifications hold:
\begin{itemize}
    \item[$(i)$] 
If $p\in(0,1)$ or $s< \frac 12$, then
\begin{align*}
    H^s_{p,\Gamma}(\Omega)&=  H^s((\Omega+\Gamma_p)\backslash\overline{\Omega}),
\end{align*}
and $\|f\|_{H^{s}((\Omega+\Gamma_p)\backslash\overline{\Omega})} = \|f\|_{H^{s}_{p,\Gamma}(\Omega)}$.
\item[$(ii)$]
Let $p\in\{0,1\}$  and $s>\frac 12$.  
For any $f\in H^s_{p,\Gamma}(\Omega)$ there exists $(g,h)\in W^s_{p,\Gamma}(\Omega)$ such that for any $\tilde{f}\in H^s(\R^n)$ with $\tilde f\in [f]$ it holds
\begin{align*}
   \tilde{f}|_{\p_{\Gamma_p} \Omega}=g, \quad \tilde{f}|_{(\Omega + \Gamma_p) \backslash \overline{\Omega}}=h,
\end{align*}
and
\begin{align*}
   \|(g,h)\|_{W^s_{p,\Gamma}(\Omega)}\leq C \|f\|_{H^s_{p,\Gamma}(\Omega)}.
\end{align*}
Conversely, let $(g,h)\in W^{s}_{p,\Gamma}(\Omega)$. Then all $\tilde f\in H^{s}(\R^n)$  such that $({\tilde f}|_{\p_{\Gamma_p} \Omega}, {\tilde f}|_{(\Omega + \Gamma_p) \backslash \overline{\Omega}}  )= (g,h)\in W^{s}_{p,\Gamma}(\Omega)$ belong to the same equivalence class $[f]$ in $ H^{s}_{p,\Gamma}(\Omega)$ and
\begin{align*}
   \|f\|_{H^s_{p,\Gamma}(\Omega)} \leq C \|(g, h  )\|_{W^s_{p,\Gamma}(\Omega)}.
\end{align*}
Thus, as Banach spaces, we may identify
\begin{align*}
    H^s_{p,\Gamma}(\Omega) = W^{s}_{p,\Gamma}(\Omega).
\end{align*}
\item[$(iii)$]
If $n=1$, the previous identification  can be simplified to read
\begin{align*}
    H^s_{p, \R_+}(\Omega)&=
    \big\{(r,h)\in \R \times H^s((\Omega+ Y_p)\backslash\overline{\Omega})\big\}.
\end{align*}
\end{itemize}
\end{lem}

We emphasize that in particular (ii) is of interest in the context of inverse problems since it allows one to consider problems with local \emph{and} nonlocal boundary contributions simultaneously. Further, due to the generally non-Lipschitz regularity of $(\Omega+\Gamma_p)\backslash \overline{\Omega}$ (which may contain cusps), the existence of the extension as in the definition for $W^{s}_{p,\Gamma}(\Omega)$ is not immediate in general geometries. This explains the requirement on the existence of an extension in the definition of the function spaces from above. Avoiding the cusp region (e.g. by imposing suitable support assumptions), it is possible to prove the existence of an extension as in the definition of the space $W^{s}_{p,\Gamma}(\Omega)$.

\begin{proof}
\emph{Step 1: $(i)$.} The first part follows from the observation that
\begin{align*}
    \widetilde H^s(\Omega) \oplus \widetilde H^s\big(\R^n\backslash(\Omega+\Gamma_p)\big)= \widetilde H^s\big((\R^n\backslash(\overline{\Omega}+\Gamma_p))\cup\Omega\big).
\end{align*}
If $p\in(0,1)$, this holds by Lemma ~\ref{lem:zeroext}.
If $p\in\{0,1\}$ and $s<\frac 12$, the identity follows by the identification $H^s(U)=\widetilde H^s(U)$ in Lemma ~\ref{lem:HsLip}.
Then, if $h \in \widetilde H^s((\R^n\backslash(\overline{\Omega}+\Gamma_p))\cup\Omega)$  we can write $h=h|_\Omega+h|_{\R^n\backslash(\Omega+\Gamma_p)}$. Since both $\Omega$ and $\R^n\backslash(\Omega+\Gamma_p)$ satisfy the assumptions of Lemma ~\ref{lem:HsLip}, the restriction of $h$ to these sets belongs to the corresponding $\widetilde H^s$ space. The equality of the norms follows by definition.

\emph{Step 2: $(ii)$.}  
For the second part of the lemma, we first seek to see that any equivalence class in $H^{s}_{p,\Gamma}(\Omega)$ can be identified with a pair $(g, h)\in H^{s-\frac12}(\p_{\Gamma_p}\Omega)\times H^s((\Omega+\Gamma_p)\backslash\overline{\Omega})$ as in the definition of $W^s_{p,\Gamma}(\Omega)$. 
Indeed, let $f\in H^s_{p,\Gamma}(\Omega)$ and let
$\tilde f_j\in H^s(\R^n)$, $j\in\{1,2\}$, be such that $\tilde f_j\in[f]$.
Then $\tilde f_1-\tilde f_2=\varphi_1+\varphi_2$, with $\varphi_1\in \widetilde H^s(\Omega)$ and $\varphi_2\in \widetilde H^s(\R^n\backslash(\overline{\Omega}+\Gamma_p))$.
If $s\in \big(\frac 12,1\big)$, by \eqref{eq:Hsid} the functions  in $\widetilde H^s(\Omega)$ have zero trace on $\p \Omega$. Therefore,
\begin{align*}
    (\tilde f_1-\tilde f_2)\big|_{(\Omega+\Gamma_p)\backslash\overline{\Omega}}=0, 
    \qquad
    (\tilde f_1-\tilde f_2)\big|_{\p_{\Gamma_p}\Omega} =0.
\end{align*}
Hence, it suffices to set $g=\tilde f_j|_{\p_{\Gamma_p}\Omega}$ and $h=\tilde f_j|_{(\Omega+\Gamma_p)\backslash\overline{\Omega}}$. Both belong to the claimed spaces by definition. 

In order to prove the equivalence of norms, we observe that
by trace estimates and the definition of $H^s$ spaces, for $\tilde f\in H^s(\R^n)$
\begin{align*}
    \big\|\tilde f|_{\p_{\Gamma_p}\Omega}\big\|_{H^{s-\frac12}({\p_{\Gamma_p}\Omega})}
    &\leq C\|\tilde f\|_{H^s(\Omega)}\leq C \|\tilde f\|_{H^s(\R^n)},\\
    \big\|\tilde f|_{(\Omega+\Gamma_p)\backslash\overline{\Omega}}\big\|_{H^s({(\Omega+\Gamma_p)\backslash\overline{\Omega}})}
    &\leq \|\tilde f\|_{H^s(\R^n)}.
\end{align*}
Moreover, by the above consideration $\tilde{f}|_{\p_{\Gamma_p}\Omega} =g $ and $\tilde f|_{(\Omega+\Gamma_p)\backslash\overline{\Omega}} = h$ for all $\tilde{f}\in [f]$.
Thus, 
\begin{align*}
    \|(g,h)\|_{W^s_{p,\Gamma}(\Omega)}\leq C \|\tilde f\|_{H^s(\R^n)}.
\end{align*}
Taking the infimum among all possible extensions $\tilde f\in H^s(\R^n)$ such that  $\tilde f\in[f]$ implies the first estimate.

The converse statement can be proved as follows: Let $(g,h)\in W^s_{p,\Gamma}(\Omega)$ and  $\tilde f_j\in H^s(\R^n)$, $j\in\{1,2\}$, be such that $\tilde{f}_j|_{\p_{\Gamma_p}\Omega} =g $ and $\tilde f_j|_{(\Omega+\Gamma_p)\backslash\overline{\Omega}} = h$. Then, arguing as above, we infer $\tilde f_1-\tilde f_2\in\widetilde H^s(\Omega) \oplus \widetilde H^s(\R^n\backslash(\overline{\Omega}+\Gamma_p))$.

In order to prove the corresponding estimate, we argue by contradiction.
Let us assume that for any $k\in\N$ there exists  $(g_k,h_k)\in W^s_{p,\Gamma}(\Omega)$ 
with
\begin{align}
\label{eq:compact_bound}
    \|(g_k,h_k)\|_{W^s_{p,\Gamma}(\Omega)}=1,
\end{align}
but such that
\begin{align*}
    \| f_k\|_{H^s_{p,\Gamma}(\Omega)}>k,
\end{align*}
where $f_k\in H^s_{p,\Gamma}(\Omega)$ denotes the equivalence class of 
functions  $\tilde f_k\in H^s(\R^n)$ with  $\tilde{f}_k|_{\p_{\Gamma_p}\Omega} =g_k $, $\tilde f_k|_{(\Omega+\Gamma_p)\backslash\overline{\Omega}} = h_k$.
This means that for any $\tilde f_k\in[f_k]$
\begin{align}
\label{eq:contra}
    \|\tilde f_k\|_{H^s(\R^n)}>k.
\end{align}
This however contradicts the norm bound in the definition of the space $W^{s}_{p,\Gamma}(\Omega)$.

\emph{Step 3: $(iii)$}. Finally, if $n=1$, we notice that
for any $(r,h)\in \R \times H^s((\Omega+ Y_p)\backslash\overline{\Omega})$ there is always a function $\tilde f\in H^s(\R)$ such that $\tilde f|_{\p_{\Gamma_p}\Omega}=r$ and $\tilde f|_{(\Omega+ Y_p)\backslash\overline{\Omega}}=h$.
We just need to observe that the point $x_0=\p_{\Gamma_p}\Omega$ and $(\R_\pm+\Omega)\backslash\overline{\Omega}$ are always separated.
Then we can take $\tilde f=\tilde h_1+\tilde h_2$, where  $\tilde h_1\in H^s(\R)$  is an extension of $h$ supported in $\Omega+\R_\pm$ and $\tilde h_2\in H^s(\R)$  is   supported in  $\Omega+\R_\mp$ and verifies $\tilde h_2(x_0)=r$.
\end{proof}

\subsubsection{The exterior problem}

With the properties of the function spaces associated with the bilinear form $\tilde B^s_{p,\Gamma;q}$ in hand, we define our notion of a weak solution based on this bilinear form:

\begin{defi}
\label{defi:nonloc}
Let $s\in (0,1)\backslash\big\{\frac 12\big\}$, $p\in[0,1]$ and let $\tilde B^s_{p,\Gamma;q}$ be the bilinear form from \eqref{eq:tBspgammaq}. 
Given $f\in  H^s_{p,\Gamma}(\Omega)$, a function $u\in  H^s(\Omega+\Gamma_p)$ is a \emph{(weak) solution of \eqref{eq:AspgammaExt} based on $\tilde B^s_{p,\Gamma;q}$} 
if
\begin{align*}
    &\tilde B^s_{p,\Gamma;q}(\tilde u,v)=0 \ \mbox{ for all } v\in\widetilde H^s(\Omega) \mbox{ and any } \tilde u\in H^s(\R^n) \mbox{ such that } \tilde u|_{\Omega+\Gamma_p}=u 
    \\& \mbox{and }\ \mathcal E_{\Omega+\Gamma_p}(u-\tilde f) \in\widetilde H^s(\Omega) \ \mbox{ for any } \tilde f\in H^s(\R^n) \mbox{ with } \tilde f\in [f].
\end{align*}
\end{defi}

\begin{rmk}
We emphasize that $\tilde B^s_{p,\Gamma;q}(\tilde u,v)$ does not depend on the extension $\tilde u$.
Indeed, let $\tilde u_j\in H^s(\R^n)$ be such that $\tilde u_j|_{\Omega+\Gamma_p}=u$ for $j\in\{1,2\}$.
Let $w:=\tilde u_1-\tilde u_2\in \widetilde H^s(\R^n\backslash(\overline{\Omega}+\Gamma_p))$.
Then, for  any $v\in\widetilde H^s(\Omega)$, \eqref{eq:tBspgammaq} can be reduced to
\begin{align*}
    \tilde B^s_{p,\Gamma;0}(w,v)
    &=-\beta_s \int_{\R^n\backslash(\overline{\Omega}+\Gamma_p)} \int\limits_{\Gamma_p}
    \big(w(x)-w(x+y) \big) v(x-y)\nu^s_{p, \Gamma}(y)dy dx.
\end{align*}
Since $x-y\notin \Omega$ for any $y\in \Gamma_p$ and  $x\notin \Omega+\Gamma_p$, for these values of $x,y \in \R^n$ we obtain $v(x-y)= 0$ and the integral vanishes. 
\end{rmk}

\begin{rmk}
\label{zero_data2}
As an alternative, we could also consider data in more restrictive data space $\widetilde{H}^s((\Omega+\Gamma_p)\backslash\overline{\Omega})$. In this case, we infer that the solution $u$ from Definition \ref{defi:nonloc} satisfies $u\in \widetilde{H}^s(\Omega+\Gamma_p)$ and  it is not necessary to consider arbitrary extensions but only the (in this case canonical) extension by zero outside the region of definition.
\end{rmk}

\begin{prop}
\label{prop:tildeB_exist}
Let $s\in (0,1)\backslash\big\{\frac 12\big\}$, $p\in[0,1]$,  $\Omega \subset \R^n$ be a bounded, Lipschitz open set,  $\Gamma\subset\R^n\backslash\{0\}$ be an open, non-empty, convex cone and $q\in L^\infty(\Omega)$.
Then there is a countable set $\tilde{\Sigma}^s_{p,\Gamma;q}\subset\C$ such that if $\lambda\notin \tilde{\Sigma}^s_{p,\Gamma;q}$,
for any $f\in  {H}^{s}_{p,\Gamma}(\Omega)$ and $g\in H^{-s}(\Omega)$, there is a unique  weak solution  $u\in  H^s(\Omega+\Gamma_p)$ based on $\tilde B^s_{p,\Gamma;q}$ of
\begin{align*}
    \big(A^s_{p,\Gamma}(D)+q-\lambda\big)u&=g\; \mbox{ in } \Omega,\\
    u&=f \; \mbox{ in } (\Omega+\Gamma_p)\backslash\overline{\Omega}.
\end{align*}
Moreover,
\begin{align*}
    \|u\|_{H^s(\Omega+\Gamma_p)} \leq C\big( \|f\|_{H^s_{p,\Gamma}(\Omega)} + \|g\|_{H^{-s}(\Omega)}\big).
\end{align*}
\end{prop}

\begin{proof}
We reduce this problem  to the interior one in \eqref{eq:AspgammaInh}, solved in Proposition ~\ref{prop:AspgammaInh}.
Let $f\in H^s_{p,\Gamma}(\Omega)$ and let $\tilde f\in H^s(\R^n)$ such that $\tilde f\in[f]$.
We  construct $\tilde u=w+\tilde f\in H^s(\R^n)$, where 
$w\in\widetilde H^s(\Omega)$ is the solution of \eqref{eq:AspgammaInhlambda} with inhomogeneous term $\text{g}=g-(A^s_{p,\Gamma}(D)+q-\lambda)\tilde f|_\Omega$. This requires proving that $A^s_{p,\Gamma}(D) \tilde f|_\Omega\in H^{-s}(\Omega)$ (interpreted weakly) for $\tilde f\in H^s(\R^n)$.
Arguing as in \eqref{eq:conttBspgamma}, we obtain for any $v\in\widetilde H^s(\Omega)$
\begin{align*}
    |\tilde B^s_{p,\Gamma;q}(f,v)|\leq C\|\tilde f\|_{H^s(\R^n)}\|v\|_{H^s(\R^n)}.
\end{align*}
This implies that $\text{g}\in H^{-s}(\Omega)$ and 
$\|\text{g}\|_{H^{-s}(\Omega)}\leq C \|\tilde f\|_{H^s(\R^n)}+\|g\|_{H^{-s}(\Omega)}$.

By Proposition ~\ref{prop:AspgammaInh}, if $\lambda\notin \tilde \Sigma^s_{p,\Gamma;q}$, a weak solution $w\in \widetilde H^s(\Omega)$ exists, is unique and satisfies
\begin{align*}
    \|w\|_{H^s(\R^n)}
    \leq C\|\text{g}\|_{H^{-s}(\Omega)}\leq C \big(\|\tilde f\|_{H^s(\R^n)}+\|g\|_{H^{-s}(\Omega)}\big).
\end{align*}
Therefore,
\begin{align*}
    \|\tilde u\|_{H^s(\R^n)}
    \leq C \big(\|\tilde f\|_{H^s(\R^n)}+\|g\|_{H^{-s}(\Omega)}\big).
\end{align*}

Taking the infimum among all possible $\tilde f$ and defining $u:=\tilde u|_{\Omega+\Gamma_p}$ yields the result.

\end{proof}

With this well-posedness result in hand, we define a corresponding Poisson operator:

\begin{defi}[Poisson operator]
\label{def:Poissonspgamma}
Let $s\in (0,1)\backslash \big\{\frac{1}{2}\big\}$ and $p\in[0,1]$. Let $\Omega$ and $\Gamma $ be as in Proposition ~\ref{prop:tildeB_exist}. Let $q\in L^\infty(\Omega)$ be such that $0\notin \tilde \Sigma^s_{p,\Gamma;q}$. We define the \emph{Poisson operator} as the mapping
\begin{align}
    \tilde P_{p,\Gamma;q}: H^s_{p,\Gamma}(\Omega) \to H^s(\Omega+\Gamma_p), \ f\mapsto \tilde P_{p,\Gamma;q}f=u,
\end{align}
where $u$ is the  (weak) solution to \eqref{eq:AspgammaExt} from Proposition ~\ref{prop:tildeB_exist}.
\end{defi}

\subsubsection{The Dirichlet-to-Neumann map}\label{sec:DtNspgamma}

In this section, we seek to define the Dirichlet-to-Neumann operator, in order to formulate the  associated  inverse  problem. In the sequel, we will always assume that $0\notin \tilde\Sigma^s_{p, \Gamma;q}\cap \tilde \Sigma^s_{1-p, \Gamma;q} $.

We define the Dirichlet-to-Neumann map for boundary data compactly supported in $(\Omega+\Gamma_p)\backslash\overline{\Omega}$, which avoids possible non-uniqueness issues originating from the choice of an extension (c.f. Remark ~\ref{rmk:global_DtN}), as follows:

\begin{defi}[Dirichlet-to-Neumann operator $\tilde \Lambda_{p,\Gamma;q}$]
\label{defi:tDtNspgamma}
Let $s\in (0,1)\backslash \{\frac{1}{2}\}$,  $p\in [0,1]$. Let $\Omega$ and $\Gamma $ be as in Proposition ~\ref{prop:tildeB_exist}. Let $q\in L^\infty(\Omega)$ be such that $0\notin \tilde \Sigma^s_{p,\Gamma;q}$ and let $\tilde{B}_{p,\Gamma;q}^s$ be the bilinear form \eqref{eq:tBspgammaq}. We set 
\begin{align*}
    \tilde \Lambda_{p,\Gamma;q}:  \widetilde H^s((\Omega+\Gamma_p)\backslash\overline{\Omega})\to H^{-s}(\R^n): f \mapsto \tilde \Lambda_{p,\Gamma;q} f,
\end{align*}
given by
\begin{align*}
    \langle \tilde \Lambda_{p,\Gamma;q} f, h\rangle=\tilde B^s_{p,\Gamma;q}(u_f, h) \; \mbox{ for any } h\in H^s(\R^n)
\end{align*}
where $u_f:=\mathcal E_{\Omega+\Gamma_p} (\tilde P_{p,\Gamma;q}f)$.
\end{defi}

\begin{rmk}\label{rmk:DtNspgammainfo}
We highlight that the Dirichlet-to-Neumann operator contains non-trivial information about $u_f$  only in $(\Omega+\Gamma_{1-p})$.
Indeed, notice that we need to ``see'' $u_f$ in $\Omega$, since in the remaining region $u_f$ corresponds to the known exterior data $f$.
By the definition of $\tilde B^s_{p,\Gamma;q}$, this requires considering test functions $h$ with $supp(h)\cap (\Omega+\Gamma_{1-p})\neq\emptyset$. This is particularly of relevance for the cases $p\in \{0,1\}$. 

Furthermore, thinking of applications of inverse problems, usually $\Omega$ is not accesible for measurements, so we will be particularly interested in the Dirichlet-to-Neumann maps  in subsets of 
$(\Omega+\Gamma_{1-p})\backslash\overline{\Omega}$.
\end{rmk}

We next  relate the two Dirichlet-to-Neumann operators $\tilde \Lambda_{p,\Gamma;q}$ and $\tilde \Lambda_{1-p,\Gamma;q}$  for data supported in appropriate sets (they differ only if $p\in\{0,1\}$). This will be crucial for the Alessandrini identity and thus the uniqueness results for the inverse problem.

\begin{lem}
\label{rmk:DtNforAlessAspgamma}
Let $s\in (0,1)\backslash \{\frac{1}{2}\}$,  $p\in [0,1]$.  Let $\Omega$ and $\Gamma $ be as in Proposition ~\ref{prop:tildeB_exist}. Let $q\in L^\infty(\Omega)$ be such that  $0\notin \tilde \Sigma^s_{p,\pm \Gamma;q}$ and let  $\tilde{\Lambda}_{p,\Gamma;q}^s$ be the operator from above. 
Let $f\in  \widetilde H^s((\Omega+\Gamma_p)\backslash\overline{\Omega})$ and $h\in \widetilde H^s((\Gamma_{1-p}+\Omega)\backslash\overline{\Omega})$. 
Then,
\begin{align*}
 \langle\tilde\Lambda_{p,\Gamma;q}f, h\rangle = \langle \tilde \Lambda_{1-p,\Gamma;q} h,  f\rangle.
\end{align*}
\end{lem}

\begin{proof} 
The claim follows  from the definition of the Dirichlet-to-Neumann maps $\tilde \Lambda^s_{p,\Gamma;q}$. To this end, let $u_f\in \widetilde H^s(\Omega+\Gamma_p)$ and $w_h \in \widetilde H^s(\Gamma_{1-p}+\Omega)$ be the extensions by zero of $\tilde P_{p,\Gamma;q}f$ and $\tilde P_{1-p,\Gamma;q}h$, respectively.
Recall that $u_f-f, w_h-h\in\widetilde H^s(\Omega)$.
Therefore,
\begin{align*}
    \langle\tilde \Lambda_{p,\Gamma;q}f, h\rangle
    &=\tilde B^s_{p,\Gamma;q}(u_f, h)
    =\tilde B^s_{p,\Gamma;q}(u_f, w_h)-\tilde B^s_{p,\Gamma;q}(u_f,  w_h- h)
    \\&=\tilde B^s_{p,\Gamma;q}(u_f- f, w_h)+\tilde B^s_{p,\Gamma;q}(f, w_h),
    \\&=\tilde B^s_{1-p,\Gamma;q}( w_h,  f)
    =\langle\tilde\Lambda_{1-p,\Gamma;q}h, f\rangle.
\end{align*}
\end{proof}

Finally, as in Section ~\ref{sec:direct}, we discuss how to exploit the Dirichlet-to-Neumann operator as data for the inverse problem associated with the operators $A^s_{p,\Gamma}(D)$ by deducing its distributional form:

\begin{lem}[Distributional characterization]
Let $s\in (0,1)\backslash \{\frac{1}{2}\}$,  $p\in [0,1]$.  Let $\Omega$ and $\Gamma $ be as in Proposition ~\ref{prop:tildeB_exist}. Let $q\in L^\infty(\Omega)$ be such that $0\notin \tilde \Sigma^s_{p,\Gamma;q}$ and let $\tilde{\Lambda}_{p,\Gamma;q}$ be the operator from Proposition ~\ref{prop:DtNtilde}. Let $f\in  \widetilde H^s((\Omega+\Gamma_p)\backslash\overline{\Omega})$ and let $u_f= \tilde{P}_{p,\Gamma;q} f$. Then $\tilde{\Lambda}_{p,\Gamma;q} f = A_{p,\Gamma}^s(D) u_f |_{\mathcal C(\Omega)\backslash\overline{\Omega}}$ in the sense that for all $\varphi \in C_c^{\infty}((\Omega+\Gamma_p)\backslash\overline{\Omega})$  it holds that
\begin{align*}
   \langle \tilde \Lambda_{p,\Gamma; q} f, \varphi \rangle =    \langle  A_{p,\Gamma}^s(D) u_f, \varphi \rangle.
\end{align*} 
\end{lem}

\begin{proof}
The proof follows directly from the Fourier characterization of the bilinear form \eqref{eq:Fourier_nonlocal}.
\end{proof}

\subsection{Runge approximation and unique continuation}
\label{sec:Runge_approx_one_sided}

In the next results, both directional antilocality and geometric assumptions play a crucial role. 
In order to work in the settings for which directional antilocality is known, we will consider one of the following assumptions for the sets $\Omega, W\subset \R^n$ specified below:
\begin{itemize}
    \item[$(a)$]\label{ass:a} $n=1$ and $p\in [0,1]$ arbitrary,
    \item[$(b)$]\label{ass:b} $n=2$,  $p\in\{0,1\}$ and $\Omega\subset W+\Gamma_{1-p}$, 
    \item[$(c)$]\label{ass:c} $\Omega\subset \bigcap_{x\in W} (x+\Gamma_{1-p})$, for any $n\geq 2$ and $p\in[0,1]$. 
\end{itemize}

Notice that in the first two cases the $\Gamma_p$-antilocality of $A^s_{\Gamma,p}$ is ensured by Ishikawa's results in Lemmas ~\ref{lem:Ishik1D} and ~\ref{lem:Ishik2D} (also Lemma ~\ref{lem:As0AntMoments}). The geometric assumption in $($\hyperref[ass:b]{$b$}$)$ is irrelevant for \emph{deducing} the directional antilocality but crucial for \emph{exploiting} the directional antilocality and thus for deriving the result.  
Under the hypotheses in $($\hyperref[ass:c]{$c$}$)$, by virtue of Lemma ~\ref{lem:partantiloc}, we have that $A^s_{1-p,\Gamma}$ is $\Gamma_{1-p}$-antilocal  from $U$ for functions supported in $\Omega$. Notice that in this case the geometric assumption is already necessary for deducing the antilocality property. We emphasize that the conditions $($\hyperref[ass:a]{$a$}$)$-$($\hyperref[ass:c]{$c$}$)$ are not expected to be exhaustive for the results presented below; any condition under which sufficiently strong antilocality properties are guaranteed would be admissible.

\subsubsection{Runge approximation}

We remark that even in the case of only one-sided antilocality of the operators, analogous density results hold as in the fully antilocal setting:

\begin{prop}[Runge approximation]
\label{prop:Runge_one_side}
Let $s\in (0,1)\backslash \{\frac{1}{2}\}$, $\Gamma\subset\R^n\backslash\{0\}$ be an open, non-empty, convex cone, $\Omega \subset \R^n$ be open, bounded, Lipschitz and let $W \subset (\Omega+\Gamma_p)\backslash\overline{\Omega}$ be open. 
Assume that one of the conditions $($\hyperref[ass:a]{$a$}$)$, $($\hyperref[ass:b]{$b$}$)$ or $($\hyperref[ass:c]{$c$}$)$ holds.
Let $q\in L^{\infty}(\Omega)$ and suppose that $0 \notin \tilde{\Sigma}^s_{p;q}\cup \tilde{\Sigma}^s_{1-p;q}$. Then, the set
\begin{align*}
\mathcal{R}:=\{u:= \tilde P_{p,\Gamma;q} f|_{\Omega}: \ f \in C_c^{\infty}(W)\}
\end{align*}
is dense in $L^2(\Omega)$.
\end{prop}

\begin{proof}
The proof is analogous to Step 2 of the proof of Theorem ~\ref{thm:dual} and the proof of Theorem ~\ref{thm:Runge},  but requires taking into account the non-self-adjointness of the operator when $p\neq\frac 12$. 

Let $F\in L^2(\Omega)$ be such that  $\langle F, \tilde P_{p,\Gamma;q}f|_{\Omega}\rangle=0$ for any $f\in \widetilde H^s(W)$ with  $W\subset(\Omega+\Gamma_p)\backslash\overline{\Omega}$.
 Let $\phi\in \widetilde H^s(\Omega)$ be the solution to  
\begin{align*}
    \big(A^s_{1-p,\Gamma}(D)+q\big)\phi&=F \;\mbox{ in } \Omega,\\
    \phi&=0 \; \ \mbox{ in } (\Gamma_{1-p}+\Omega)\backslash\overline{\Omega}.
\end{align*}
This implies that for all $v\in\widetilde H^s(\Omega)$,
\begin{align*}
   \langle F, v\rangle=\tilde B^s_{1-p,\Gamma;q}(\phi, v).
\end{align*}
Thus, for any $f\in \widetilde H^s(W)$,
\begin{align*}
   0&= \langle F,\ \tilde P_{p,\Gamma;q} f|_{\Omega} \rangle
    =\langle F,\ \tilde P_{p,\Gamma;q}f-f \rangle
    = \tilde B^s_{1-p,\Gamma;q}(\phi,\  \tilde P_{p,\Gamma;q}f-f)
    \\&=\tilde  B^s_{1-p,\Gamma;q}(\phi, \tilde P_{p,\Gamma;q}f)-\tilde  B^s_{1-p,\Gamma;q}(\phi,f)
    =-\langle A^s_{1-p,\Gamma}(D)\phi,\ f \rangle.
\end{align*}
Here have used that by \eqref{eq:tBspgammaq}
\begin{align*}
   \tilde B^s_{1-p,\Gamma;q}(\phi, \tilde P_{p,\Gamma;q}f)
   =\tilde  B^s_{p,\Gamma;q}( \tilde P_{p,\Gamma;q}f,\phi),
\end{align*}
and, since  $\phi\in\widetilde H^s(\Omega)$, $\tilde  B^s_{p,\Gamma;q}( \tilde P_{p,\Gamma;q}f,\phi)=0$.

This implies $A^s_{1-p,\Gamma}(D)\phi=0=\phi$ in $W$. By virtue of the assumptions $($\hyperref[ass:a]{$a$}$)$, $($\hyperref[ass:b]{$b$}$)$ or $($\hyperref[ass:c]{$c$}$)$, $A^s_{1-p,\Gamma}$ is $\Gamma_{1-p}$-antilocal (at least from $W$ to $\Omega$). We hence infer that $\phi=0$ in $\Omega$.
Therefore, $F=0$ and the density result follows by the Hahn-Banach theorem. 
\end{proof}

\subsubsection{Weak unique continuation}
We have already seen in Proposition ~\ref{prop:WUCP} that if a (not necessarily symmetric) operator is antilocal in the \emph{two-sided} cone $-\mathcal{C}\cup \mathcal{C}$, then the weak unique continuation property holds in differentiable domains. 
This in particular includes $A^s_{p,\Gamma}(D)$ for $p\in (0,1)$.
For $p\in\{0,1\}$, by relying on well-posedness, we show that  weak unique continuation holds for $n=1$ and that it may not hold in general for $n\geq 2$.

\begin{lem}[WUCP for $A^s_0(D)$]
\label{lem:1Dwucp}
Let $s\in (0,1)\backslash\big\{\frac 12\big\}$ and let $\Omega\subset \R$ be bounded, open, Lipschitz and $q\in L^\infty(\Omega)$ such that $0\notin \tilde \Sigma^s_{0;q}$. Assume that
\begin{align*}
\big(A^s_0(D) + q\big)u =0 \mbox{ in } \Omega,
\end{align*}
and that $u =0$ in some open set $U \subset \Omega$. Then,  $u\equiv 0$ in $\R_++\Omega$.
\end{lem}

\begin{proof}
By virtue of the equation, if $u=0$ in $U$, then $A^s_0(D) u=0$ in $U$. 
By the antilocality to the right, it follows that $u\equiv 0$ in $\R_++U$. Hence, the Dirichlet data for the problem vanish. Thus, by well-posedness $u\equiv 0$ also in $\Omega$.
\end{proof}

\begin{lem}[Lack of WUCP for $A^s_{0,\Gamma}$, $s<\frac 12$]\label{lem:lackWUCP}
Let  $0<s<\frac 12$ and assume that $A^s_{0,\Gamma}(D)$ is $\Gamma$-antilocal. Let $\Omega\subset \R^n$ be bounded, open, Lipschitz and let $q\in L^\infty(\Omega)$ be such that $0\notin \tilde \Sigma^s_{0;q}$. Assume that there is $W\subset \Omega+\Gamma$  such that, if $\Omega_1:=\Omega\cap (W-\Gamma)$ and $\Omega_2:=\Omega\backslash\overline{\Omega}_1$, then zero is not a Dirichlet eigenvalue for  $A^s_{0,\Gamma}(D)+q$ in $\Omega_2$. 
Then there exist  an open set $U\subset \Omega$ 
and a non-trivial $u\in H^s(\Gamma+\Omega)$ such that 
\begin{align*}
\big(A^s_0(D) + q\big)u =0 \mbox{ in } \Omega,
\end{align*}
and  $u=0$ in  $U$.
\end{lem}

\begin{proof}
Let $f\in \widetilde H^s(W)$ be non-trivial and let $u\in \widetilde{H}^s(\Gamma+\Omega)$ (see Remark \ref{zero_data2}) be the unique solution to 
\begin{align}\label{eq:probu}
\begin{split}
    \big(A^s_{0,\Gamma}(D)+q\big)u &=0 \;\mbox{ in }\Omega,\\
    u&=f\; \mbox{ in }(\Omega+\Gamma)\backslash\overline{\Omega}.
\end{split}
\end{align}
Notice that $u\neq 0$ in $\Omega$, otherwise by the $\Gamma$-antilocality, it would follow that $f=0$. 

Let $u_j=u|_{\Omega_j}\in H^s(\Omega_j)$, $j\in\{1,2\}$. Since $s<\frac 12$ and $\Omega_j$ is Lipschitz, $u_j\in\widetilde H^s(\Omega_j)$.
By the construction of the subsets,  we observe that $u_2$ solves
\begin{align*}
    \big(A^s_{0,\Gamma}(D)+q\big)u_2 &=0 \;\mbox{ in }\Omega_2,\\
    u_2&=0\; \mbox{ in }(\Omega_2+\Gamma)\backslash\overline{\Omega}_2.
\end{align*}
Indeed, $u_2\in\widetilde H^s(\Omega_2)$
and for any $\varphi\in\widetilde H^s(\Omega_2)\subset\widetilde H^s(\Omega)$
\begin{align*}
    \tilde B^s_{0,\Gamma;q}(u_2,\varphi)
    &=\tilde B^s_{0,\Gamma;q}(u,\varphi)-\tilde B^s_{0,\Gamma;q}(u_1,\varphi))-\tilde B^s_{0,\Gamma;q}(f,\varphi).
\end{align*}
Since $u$ is a solution to \eqref{eq:probu}, it follows that $\tilde B^s_{0, \Gamma;q}(u,\varphi)=0$.
In addition, since $\supp(u_1)\subset \overline{\Omega}_1$ and $\supp(\varphi)\subset\overline{\Omega}_2$, it holds by \eqref{eq:tBspgammaq}
\begin{align*}
    \tilde B^s_{0, \Gamma;q}(u_1,\varphi)
    &=-\beta_s\int_{\Omega_1}\int_{\Gamma}\big(u_1(x)-u_1(x+y)\big)\varphi(x-y)\nu^s_{p,\Gamma}(y)dy dx
    \\& \quad-\beta_s\int_{(\Omega_1-\Gamma)\backslash\overline{\Omega}_1}\int_{\Gamma}u_1(x+y)\big(\varphi(x)-\varphi(x-y)\big)\nu^s_{p,\Gamma}(y)dy dx.
\end{align*}
Since $\Omega_2\cap (\Omega_1-\Gamma)=\emptyset$, both integrals vanish. 
Similarly, taking into account the supports of $\varphi$ and $f$ and that $\Omega_2\cap(W-\Gamma)=\emptyset$, we infer $\tilde B^s_{0,\Gamma;q}(f,\varphi)=0$.

We remark that in the problem for $u_2$ no data on $\p_{\Gamma} \Omega_2=(\p_\Gamma\Omega\cap\overline{\Omega}_2)\cup (\overline{\Omega}_1\cap\overline{\Omega}_2)$ needs to be chosen,   according to Lemma ~\ref{lem:Aspgammadata} for $s<\frac 12$. 
By the well-posedness of the problem on $\Omega_2$, we finally conclude that
$u|_{\Omega_2}=u_2=0$.
Therefore, by the equation  $u=0=A^s_{0,\Gamma}(D)u$ in $\Omega_2$.
\end{proof}

\begin{rmk} We highlight that
the previous argument does not work for $s\in \big(\frac 12,1\big)$ due to the possible presence of traces on $\partial \Omega_1 \cap \partial \Omega_2$.
\end{rmk}

\subsection{The inverse problem} 
\label{sec:inv_ex}
Using the results from the previous subsection, we  study the inverse problem associated with the operators $A^s_{p,\Gamma}(D)$ under conditions ensuring (partial) directional antilocality. 
As discussed  in Remark ~\ref{rmk:DtNspgammainfo}, in the case of operators seeing only one cone, we are forced to take measurements in the opposite cone of $\Omega$ in order to capture information on the solution. 
Hence we need to modify the conditions $($\hyperref[ass:b]{$b$}$)$ and $($\hyperref[ass:c]{$c$}$)$ to include both cones as follows (for sets $W_1, W_2, \Omega \subset \R^n$ specified below):
\begin{itemize}
    \item[$(b')$]\label{ass:b'} $n=2$,  $p\in\{0,1\}$ and $\Omega\subset (W_1+\Gamma_{1-p})\cap (W_2+\Gamma_{p})$, 
    \item[$(c')$]\label{ass:c'} $\Omega\subset \big(\bigcap_{x\in W_1} (x+\Gamma_{1-p})\big) \cap \big(\bigcap_{x\in W_2} (x+\Gamma_{p})\big)$, for any $n\geq 2$ and $p\in[0,1]$.
\end{itemize}
As above, we do not expect that these conditions are exhaustive for the following results to hold.

\begin{prop}[Uniqueness]
Let $s\in (0,1)\backslash \{\frac{1}{2}\}$, $\Gamma\subset\R^n\backslash\{0\}$ be an open, non-empty, convex cone  and  $\Omega\subset \R^n$ be open, bounded, Lipschitz. Let $q_1,q_2\in L^\infty(\Omega)$ be such that $0\notin \tilde \Sigma^s_{p,\Gamma;q_i}$.
Let $W_1\subset \Omega+\Gamma_p$,  $W_2\subset \Omega+\Gamma_{1-p}$  be open, non-empty sets and assume that one of the conditions $($\hyperref[ass:a]{$a$}$)$, $($\hyperref[ass:b']{$b'$}$)$ or $($\hyperref[ass:c']{$c'$}$)$ holds.
If for any $f\in\widetilde H^s(W_1)$ 
\begin{align*}
    \tilde\Lambda_{p,\Gamma;q_1}f|_{W_2}=\tilde\Lambda_{p,\Gamma; q_2}f|_{W_2},
\end{align*}
then $q_1=q_2$.
\end{prop}

\begin{proof}
Let $f_j\in\widetilde H^s(W_j)$ for $j\in\{1,2\}$.
Let $u_j\in H^s(\R^n)$ be the solutions to
\begin{align*}
    \begin{aligned}
    \big(A^s_{p,\Gamma}(D)+q_1\big)u_1&= 0\ \;\mbox{ in } \Omega,\\
    u_1&=f_1 \;\mbox{ in } (\Omega+\Gamma_p)\backslash\overline{\Omega},
    \end{aligned}
    \qquad
    \begin{aligned}
    \big(A^s_{1-p,\Gamma}(D)+q_2\big)u_2&= 0\ \;\mbox{ in } \Omega,\\
    u_2&=f_2 \;\mbox{ in } (\Gamma_{1-p}+\Omega)\backslash\overline{\Omega}.
    \end{aligned}
\end{align*}
By  Lemma ~\ref{rmk:DtNforAlessAspgamma} the following Alessandrini identity holds:
\begin{align*}
    \langle(\tilde\Lambda_{p,\Gamma;q_1}-\tilde\Lambda_{p,\Gamma;q_2})f_1, f_2\rangle
    &=\langle\tilde\Lambda_{p,\Gamma;q_1}f_1, f_2\rangle-\langle\tilde\Lambda_{1-p,\Gamma;q_2}f_2, f_1\rangle
    \\&=\tilde B^s_{p,\Gamma;q_1}(u_1, f_2)-\tilde B^s_{1-p,\Gamma;q_2}( u_2, f_1)
    \\&=\tilde B^s_{p,\Gamma;q_1}(u_1, u_2)-\tilde B^s_{p,\Gamma;q_2}(u_1, u_2)
    =((q_1-q_2)u_1, u_2)_{L^2(\Omega)}.
\end{align*}

Now, by the Runge approximation result from Proposition \ref{prop:Runge_one_side} applied to both $W_1$ and $W_2$, we obtain a dense set of functions $\{u_1u_2\}$ and conclude the desired uniqueness result analogously as in the proof of Theorem ~\ref{thm:inf_meas}.
\end{proof}

In the investigation of uniqueness by only one single measurement, in order to avoid conditions on the (exact) support of $f$, we impose stricter geometric assumptions on $W_1$ but essentially do not modify the ones on $W_2$. In order to avoid confusion, we state the particular alternative conditions on $W_2$ here:
\begin{itemize}
    \item[$(b'')$]\label{ass:b''} $n=2$,  $p\in\{0,1\}$, $\Omega\subset  (\Gamma_{p}+W_2)$ and $W_1\subset \bigcap_{x\in\Omega}(x+\Gamma_p)$, 
    \item[$(c'')$]\label{ass:c''} $\Omega\subset  \bigcap_{x\in W_2} (\Gamma_{p}+x)$ and $W_1\subset \bigcap_{x\in\Omega}(x+\Gamma_p)$, for any $n\geq 2$ and $p\in[0,1]$.
\end{itemize}

\begin{prop}[Single measurement uniqueness]
Let $s\in (0,1)\backslash \{\frac{1}{2}\}$, $\Gamma\subset\R^n\backslash\{0\}$ be an open, non-empty, convex cone  and  $\Omega\subset \R^n$ be open, bounded, Lipschitz. Let $q\in C^0(\Omega)$ be such that $0\notin \tilde \Sigma^s_{p,\Gamma; q}$.
Let $W_1\subset \Omega+\Gamma_p$,  $W_2\subset \Omega+ \Gamma_{1-p}$  be open, non-empty Lipschitz sets and assume that one of the conditions $($\hyperref[ass:a]{$a$}$)$, $($\hyperref[ass:b'']{$b''$}$)$ or $($\hyperref[ass:c'']{$c''$}$)$ holds.
Let $f\in \widetilde H^s(W_1)\backslash \{0\}$. Then the knowledge of $f$ and  $\tilde \Lambda_{p,\Gamma;q} f|_{W_2}$ determines $q$ uniquely.
\end{prop}

\begin{proof}
Let $u \in \widetilde{H}^s(\Omega \cup W_1)$ be the solution of \eqref{eq:AspgammaExt} with $f\in \widetilde H^s(W_1)\backslash \{0\}$. By  $($\hyperref[ass:a]{$a$}$)$ or  $($\hyperref[ass:b'']{$b''$}$)$, $A^s_{p,\Gamma}(D)$ is $\Gamma_p$-antilocal. If  $($\hyperref[ass:c'']{$c''$}$)$ holds, then we have $\Gamma_p$-antilocality from $W_2$ to $\Omega\cup W_1$ (notice that $W_1\subset \Omega+\Gamma_p\subset \bigcap_{x\in W_2}(x+\Gamma_p)$). Therefore the function $u$ is uniquely determined in $\Omega+\Gamma_p$ by $u|_{W_2}=0$ and $A^s_{p,\Gamma}(D)u|_{W_2}=\tilde \Lambda_{p,\Gamma;q} f|_{W_2}$.  
Since $\Omega\subset W_2+\Gamma_p$, we recover the function $u$ in $\Omega$. Together with the fact that $u=f$
in $(\Omega+\Gamma_p)\backslash\overline{\Omega}$, we can calculate
$A^s_{p,\Gamma}(D)u$ in $\Omega$. 

Now, since $(A^s_{p,\Gamma}(D)+q)u=0$ in $\Omega$, we can recover $q$ provided $u$ does not vanish on open sets.
But this cannot be the case, because if $u=0$ in $U\subset \Omega$, it would also hold that $A^s_{p,\Gamma}(D)u=0$. By the  $\Gamma_p$-antilocality and the assumption on $W_1$, we would conclude that $u=f=0$ in $W_1$ contradicting our assumption that $f\in \widetilde H^s(W_1)\backslash \{0\}$.
In the case of the validity of the hypothesis $($\hyperref[ass:c'']{$c''$}$)$, the geometric assumption on $W_1$ is also used  for the $\Gamma_p$-antilocality from any $U\subset\Omega$ to $W_1$. 
\end{proof}

\section{More Examples of Partially Antilocal Operators}
\label{sec:exs2}

In this section, we discuss further examples of operators with certain (weaker) directional antilocality properties and provide counterexamples to (strong forms of) directional antilocality of these and related nonlocal elliptic operators.

\subsection{Non-antilocal operators} 

We now consider combinations of one-dimensional operators $A^s_p(D)$ acting on independent directions.
These operators are \emph{not} directionally antilocal in general, yet we prove certain forms of \emph{partial} directional antilocality under suitable assumptions on the sets where the function may be supported and where the observation is made.

Let $\boldsymbol{p} = (p_{1}, \dots, p_{n})\in [0,1]^n$.
We define the $n$-dimensional operator $A_{\boldsymbol{p}}^s(D)$
by
\begin{align*}
    A^s_{\boldsymbol{p}}(D):=\sum_{k=1}^n A^s_{p_{k}}(D_{x_k}).
\end{align*}
This operator only  ``sees'' the region given by $\cup_{k=1}^n Z_k$, where
\begin{align*}
    Z_k= \{(0, \overset{(k-1)}{\cdots}, 0, t, 0, \overset{(n-k)}{\cdots}, 0): t\in Y_{p_{k}} \},
\end{align*}
and $Y_{p}$ as in \eqref{eq:Yp}.
Notice that  for $ \boldsymbol{p}=\big(\frac12, \dots,\frac 12\big)$
\begin{align}
\label{eq:ops_sum}
    A^s_{\boldsymbol{p}}(D)=\frac 12\sum_{k=1}^n (-\p_{x_k})^{2s}.
\end{align}
We remark that operators of the type \eqref{eq:ops_sum}  can still be viewed in a generalized framework as the one outlined in \eqref{eq:L}. In addition to dropping the symmetry condition, the operators from now no longer have convex, open cones as the support of their kernels, but only consist of the union of convex, non-open cones.

\begin{rmk}
The operator $\frac 12\sum_{k=1}^n (-\p_{x_k})^{2s}$ and its quantitative nonlocality properties were studied in \cite{GFR19}. 
A qualitative version of the statement studied there (see \cite[(85)]{GFR19}) would be the following: 
\begin{align}\label{eq:nonlocalityex}
    \mbox{If } u \mbox{ is supported in } \Omega=(-1,1)^n \mbox{ and if } \sum_{k=1}^n (-\p_{x_k})^{2s} u=0 \mbox{ in } U\subset \Omega_e, \mbox{ then } u=0.
\end{align} 

We emphasize that this result requires geometric conditions on $U$ (which were stated and assumed in the proof of \cite[Proposition 8.1]{GFR19}, but not formulated explicitely in the result itself). Step 1 of the proof of \cite[Proposition 8.1]{GFR19} uses the assumption that
\begin{align*}
    \Omega\cap (U \pm \tilde Z_k)= \Omega \ \mbox{ for all } k\in \{1,\dots, n\},
\end{align*}
where 
\begin{align}\label{eq:tZk}
    \tilde Z_k= \{(0, \overset{(k-1)}{\cdots}, 0, t, 0, \overset{(n-k)}{\cdots}, 0): t\in \R_+\}.
\end{align}
This condition is in the spirit of the antilocality from the exterior in Proposition ~\ref{prop:exterior}. 

For the qualitative  statement \eqref{eq:nonlocalityex}, this is not optimal; a weaker sufficient condition for \eqref{eq:nonlocalityex} (see Lemma ~\ref{lem:excombpartantil} and Remark ~\ref{rmk:excombpartantil} from below)  would be  that
\begin{align} \label{eq:conditionOmsquare}
    \Omega\subset  U+\Big(\bigcup_{k=1}^n Z_k\Big),
\end{align} 
where in this case 
\begin{align*}
     Z_k= \{(0, \overset{(k-1)}{\cdots}, 0, t, 0, \overset{(n-k)}{\cdots}, 0): t\in \R\}.
\end{align*}
\end{rmk}

Motivated by these considerations for the special operator $\sum_{k=1}^n A^s_{\sfrac 12}(D_{x_k})$, we prove that also for a general choice of $\boldsymbol{p}\in [0,1]^n$ geometric conditions have to be imposed in order to obtain antilocality and that, without these, antilocality fails in general.

\begin{lem}\label{lem:excombnoantil}
Let $s\in (0,1)$ and $\boldsymbol{p}\in [0,1]^n$. Then the operator  $A^s_{\boldsymbol{p}}(D)$  is not antilocal in general.
\end{lem}

\begin{proof}
For simplicity, we present the proof for the case $n=2$ only. The generalization to higher dimensions is immediate. 

We construct a counterexample consisting of a function $u\in C^\infty_c(\R^2)$ and an open set $U\subset \R^2$, such that $u=0=A^s_{\boldsymbol{p}}(D)u$ in $U$
and $u\neq 0$ in $U+(Z_1\cup Z_2)$. 
We only consider the case $p_{1}, p_{2}\in[0,1)$. The remaining cases can be covered by suitable symmetrization.

Let $U=(-1,1)^2$ and let $W\subset (1, +\infty)$ be a bounded open interval. Let $f\in  C^\infty_c(W)$ and $g=A^s_{0}(D)f\in C^\infty(\R)$. 
Let $\eta\in C^\infty_c((-2,2))$ with $\eta=1$ in $(-1,1)$.
We define 
\begin{align*}
    u(x_1,x_2):=\frac{1}{1-p_1} f(x_1) g(x_2)\eta(x_2)-\frac{1}{1-p_2} g(x_1)\eta(x_1) f(x_2) \in C^\infty_c(\R^2).
\end{align*}
We observe that $\supp(u)\subset\Omega_1\cup\Omega_2$ with
$\Omega_1=W\times(-2,2)$ and $\Omega_2=(-2,2)\times W$, see Figure ~\ref{fig:excomb}.
Moreover,  for $(x_1,x_2)\in (-1,1)^2$ due to the one-sided support of the data
\begin{align*}
    A^s_{\boldsymbol{p}}(D)u(x_1,x_2)
    &=\big((1-p_1) A^s_0(D_{x_1})+(1-p_2) A^s_0(D_{x_2})\big) u(x_1,x_2)
    \\&=g(x_1)g(x_2)-g(x_1)g(x_2)=0.
\end{align*}
We summarize that this construction yields an example of a function with data supported to the right and to the top of $U$ such that 
$u=0=A^s_{\boldsymbol{p}}(D)u$ in $U$.

We next show that this can be strengthened to an example of a function with data supported on all sides of $U$:
For $p_{1}, p_{2}\neq 0$, we could also consider 
\begin{align*}
    u(x_1,x_2)&:=\left(\frac{1}{1-p_1} f(x_1) g(x_2)+\frac{1}{p_1} f'(x_1) g'(x_2)\right)\eta(x_2)
    \\&\qquad-\left(\frac{1}{1-p_2} g(x_1) f(x_2)+\frac{1}{p_2} g'(x_1) f'(x_2)\right)\eta(x_1) \in C^\infty_c(\R^2).
\end{align*}
with $f'\in C^\infty_c(W')$, 
$W'\subset(-\infty, -1)$ bounded,  and $g'=A^s_1(D)f'$. Then $\Omega\supset\supp(u)$ and $U$ satisfy 
\begin{align*}
    U\cap(\Omega\pm\tilde Z_k)=U,
\end{align*}
with $\tilde Z_k$ as in \eqref{eq:tZk}.
\end{proof}

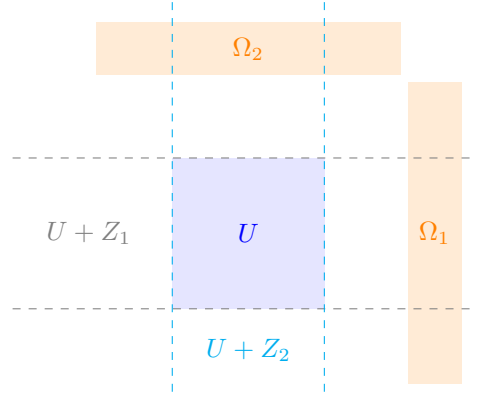
\begin{figure}[t]
\begin{tikzpicture}

\pgfmathsetmacro{\Wa}{2.1}
\pgfmathsetmacro{\Wb}{2.8}
\pgfmathsetmacro{\Wc}{(\Wa+\Wb)/2}

\pgfmathsetmacro{\Bs}{2.1}
\pgfmathsetmacro{\bs}{(\Bs+1)/2}
\pgfmathsetmacro{\Bl}{3.1}

\fill[blue!20, opacity=.5] (-1,1)--(-1,-1)--(1,-1)--(1,1)--cycle;
\node[blue] at (0,0) {$U$};

\fill[orange!20, opacity=.8] (\Wa,2)--(\Wa,-2)--(\Wb,-2)--(\Wb,2)--cycle;
\node[orange] at (\Wc,0) {$\Omega_1$};

\fill[orange!20, opacity=.8] (-2,\Wb)--(-2,\Wa)--(2,\Wa)--(2,\Wb)--cycle;
\node[orange] at (0,\Wc) {$\Omega_2$};

\draw[black!50, dashed] (-\Bl,1)--(\Bl, 1);
\draw[black!50, dashed] (-\Bl,-1)--(\Bl, -1);
\node[black!50] at (-\Bs, 0) {$U+Z_1$};

\draw[cyan!, dashed] (1,-\Bs)--(1,\Bl);
\draw[cyan!, dashed] (-1,-\Bs)--(-1,\Bl);
\node[cyan] at (0,-\bs) {$U+Z_2$};

\end{tikzpicture}
\caption{Counterexample for the proof of Lemma ~\ref{lem:excombnoantil} with $\boldsymbol p\in(0,1)^2$.
}
\label{fig:excomb}
\end{figure}

The following result recovers a certain weak form of directional antilocality in two dimensions under some assumptions on the support of the function and the subset $U$, in the spirit of Lemma ~\ref{lem:partantiloc}.

\begin{lem}\label{lem:excombpartantil}
Let $s\in (0,1)$ and $\boldsymbol{p}\in[0,1]^2$.
Let $\Omega\subset \R^2$ be open and let $U\subset \R^2\backslash \overline{\Omega}$  be connected and open. 
Assume that there exists $k\in\{1,2\}$ such that
$$ U \cap (\Omega - Z_k)\neq U.$$
Let  $u\in C^\infty_c(\Omega)$. If $A^s_{\boldsymbol{p}}(D)u=0$ in $U$, then 
$u=0$ in $\Omega\cap \big(U + (\cup_{k=1}^2 Z_k)\big)$.
\end{lem}

\begin{proof}
Without loss of generality, we assume that $U\cap (\Omega-Z_1)\subsetneq U$ (see the left panel in Figure ~\ref{fig:excombpartial}).
    Then $A^s_{p_{1}}(D_{x_1})u=0$ in $U_1:=U\backslash (\Omega-Z_1)$. However, since $A^s_{\boldsymbol{p}}(D)u=0$ in $U$, this implies $A^s_{p_{2}}(D_{x_2}) u=0$ in $U_1$.
    By the $Y_{p_{2}}$-antilocality  of $A^s_{p_{2}}(D_{x_2})$, we infer that $u=0$ in $U_1+Z_2$.
    This results in $A^s_{p_{2}}(D_{x_2}) u=0$ in $U_2:=U\cap(U_1+Z_2)$. Again by the fact that $A^s_{\boldsymbol{p}}(D)u=0$ in $U$ this also implies that $A^s_{p_{1}}(D_{x_1})u=0$ in $U_2$. Repeating the same argument, $u=0$ in $U_2+ Z_1$. 
    Iterating this argument and since $U$ is connected, we finally obtain $u=0$ in $U+Z_1$.
    Therefore, $A^s_{p_{2}}(D_{x_1})u=0$ in $U$ and hence $u=0$ in $U+Z_2$.
\end{proof}

\begin{rmk}\label{rmk:excombpartantil}
If $U$ were not connected, the assumption on $\Omega$ would have to be required for every compact connected component of $U$. Indeed, the right panel of Figure   ~\ref{fig:excombpartial} provides a counterexample for that. Let us consider $p_1=0=p_2$. Arguing as in the proof of Lemma ~\ref{lem:excombpartantil}, we can conclude $u=0$ in $\Omega\cap (U_1+Z_2)$.
But nothing can be inferred about $u$ from the other connected component $\tilde U$; on the contrary, it is even possible to choose $u$ such that it is nontrivial in $\Omega\cap(\tilde U+(Z_1\cup Z_2))$. 
Indeed, we can construct $u$ as in the proof of Lemma ~\ref{lem:excombnoantil} after a suitable change of variables. 
\end{rmk}

\begin{figure}[t]

\begin{tikzpicture}

\pgfmathsetmacro{\Bs}{2}
\pgfmathsetmacro{\bs}{(\Bs+1)/2}
\pgfmathsetmacro{\Bl}{2.8}

\pgfmathsetmacro{\dUone}{sqrt(3)/2}

\draw[blue] (0,0) circle [radius=1cm];

\draw[orange] (-2,2.3)--(-2,1.7)--(1,1.7)--(1,2.3)--cycle;

\draw[orange] (2, -.5)--(3, -.5)--(3,1.5)--(2,1.5)--cycle;

\node[orange] at (1.3,2) {\footnotesize$\Omega$};

\draw[blue!50, dashed] (-\dUone, -.5)--(-\dUone,\Bl);
\draw[blue!50, dashed] (\dUone, -.5)--(\dUone,\Bl);
\node[blue!50] at (0, 2.6) {\footnotesize$U_1+Z_2$};

\draw[red!70, densely dotted] (0,1)--(3.8,1);
\draw[red!70, densely dotted] (0,-1)--(3.8,-1);
\node[red!70] at (3.7,0) {\footnotesize$U_2+Z_1$};

 \node[blue] at (1.2,.8) {\footnotesize$U$};

\draw[black!50, dashed] (-\Bs,-.5)--(3,-.5);

\fill[blue!50, opacity=.2] (0,-.5) -- (210:1cm) arc (210:330:1cm) -- cycle;

\pattern[pattern=crosshatch dots, pattern color=red!20]  (210:1cm) arc (210:330:1cm) -- (30:1cm) arc (30:150:1cm) --cycle;

\node[blue] at (0,-.75) {\footnotesize$U_1$};
\node[red] at (0,.0) {\footnotesize$U_2$};

\begin{scope}[xshift=8cm]

\draw[orange] (-2,2.3)--(-2,1.7)--(1,1.7)--(1,2.3)--cycle;

\draw[orange] (2, -.5)--(3, -.5)--(3,1.5)--(2,1.5)--cycle;

\node[orange] at (1.25,2) {\footnotesize$\Omega$};

\draw[black!50, dashed] (-\Bs,-.5)--(3,-.5);

\draw[blue!50, dashed] (-1, -1)--(-1,\Bl);
\draw[blue!50, dashed] (0, -1)--(0,\Bl);
\node[blue!50] at (-.5,2.7) {\footnotesize $U_1+Z_2$};

\draw[red!70, densely dotted] (-1,-.5)--(3.5,-.5);
\draw[red!70, densely dotted] (-1,-1)--(3.5,-1);
\node[red!70] at (2.7,-.75) {\footnotesize $U_2+Z_1$};

\draw[blue] (0, -.5)--(1, -.5)--(1,1)--(0,1)--cycle;
\draw[blue] (0, -.5)--(-1, -.5)--(-1,-1)--(0,-1)--cycle;

\node[cyan] at (.5,.25) {\footnotesize $\tilde U$};

\node[blue] at (1.2,1.2) {\footnotesize $U$};

\fill[blue!50, opacity=.2] (0, -.5)--(-1, -.5)--(-1,-1)--(0,-1)--cycle;
\node[blue] at (-.95,-.75) {\footnotesize $\textcolor{red}{U_2}=U_1$};

\end{scope}

\end{tikzpicture}
\caption{ Illustration of the setting from Lemma ~\ref{lem:excombpartantil}. Left: The domains in the proof for   $p_{1}=p_{2}=0$. This implies $Z_1=\R_+\times\{0\}$ and $Z_2=\{0\}\times\R_+$.
Right: A geometry illustrating that the conclusion fails for non-connected domain $U$, as explained in Remark ~\ref{rmk:excombpartantil}.
}
\label{fig:excombpartial}
\end{figure}
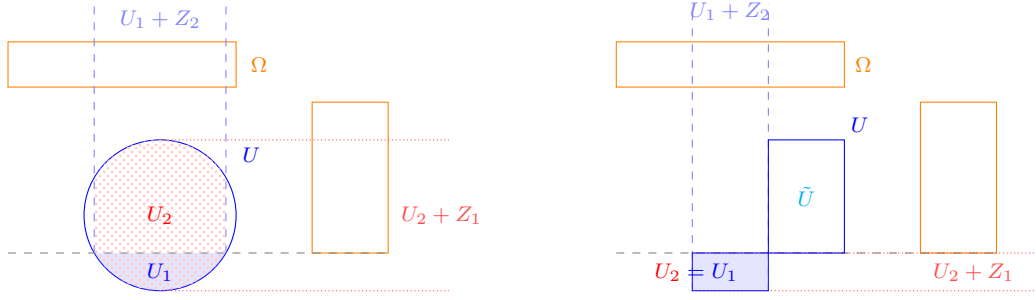

\begin{rmk}
\label{rmk:weakantilochigher}
A similar weak form of directional antilocality  can be deduced in general dimensions. However, the geometric conditions become less transparent:  
Let us assume there is a finite sequence $n_1,\dots, n_N\in\{1,\dots, n\}$ such that 
\begin{align*}
    U_N=U,
\end{align*}
where:
\begin{align*}
    U_j&= \bigcap_{\substack{k=1\\k\neq n_j}}^{n} U'_{j-1,k}, \quad j\in\{1,\dots, N\}, \\
    U'_{j,k}&=U\backslash (\Omega_j-Z_k),\quad j\in\{0,\dots, N-1\},\\
    \Omega_0&=\Omega, \qquad \Omega_j=\Omega\backslash(U_j+Z_{n_j}), \quad  j\in\{1,\dots, N-1\}.
\end{align*}
Then, arguing as above, we can conclude that 
$u=0$ in $\Omega\cap\big(U+\cup_{k=1}^n Z_k\big)$.
\\
Indeed, by construction, we know $A^s_{p_k}(D_{x_k})u=0$ in $U_{0,k}$, for $k\in\{1,\dots,n\}$. Let us assume that there is $n_1\in\{1,\dots, n\}$ such that  $U_1\neq \emptyset$. We have $A^s_{\boldsymbol p}(D)u=A^s_{p_{n_1}}(D_{x_{n_1}})u=0$ in $U_1$, and applying the $Y_{p_{n_1}}$-antilocality, we deduce $u=0$ in $\Omega\cap(U_1+Z_{n_1})$.
We can now repeat the argument ignoring that part of the assumed support, i.e. with $\Omega_1$ in the role of $\Omega$. See Figure ~\ref{fig:weakantiloc3d} for a setting in which the previous conditions hold.

\end{rmk}

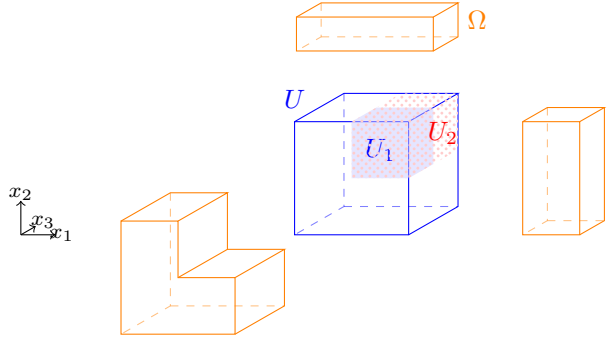
\begin{figure}[t]
\begin{tikzpicture}

\pgfmathsetmacro{\sa}{.5*sin(30)}
\pgfmathsetmacro{\ca}{.5*cos(30)}

\begin{scope}[scale=1.5]

\draw[blue] (0,0)--(0,1)--(1,1)--(1,0)--cycle;
\draw[blue] (0,1)--(\ca,1+\sa)--(1+\ca,1+\sa)--(1,1);
\draw[blue] (1+\ca,1+\sa)--(1+\ca,\sa)--(1,0);
\draw[blue, opacity=.5, dashed] (\ca,1+\sa)--(\ca,\sa)--(0,0);
\draw[blue, opacity=.5, dashed] (\ca,\sa)--(1+\ca,\sa);
\node[blue] at (0,1.2) {$U$};

\pgfmathsetmacro{\sah}{.25*sin(30)}
\pgfmathsetmacro{\cah}{.25*cos(30)}

\fill[blue!20, opacity=.5]
(.5,.5)--(.5,1)--(.5+\cah, 1+\sah)--(1+\cah, 1+\sah)--(1+\cah, .5+\sah)--(1,.5)--cycle;
\node[blue] at (.75,.75) {$U_1$};

\pattern[pattern=crosshatch dots, pattern color=red!20] (.5,.5)--(.5,1)--(.5+\ca, 1+\sa)--(1+\ca, 1+\sa)--(1+\ca, .5+\sa)--(1,.5)--cycle;
\node[red] at (1.3,.9) {$U_2$};  

\draw[orange] (-0.2+\cah,1.5+\sah)--(-.2+\cah,1.8+\sah)--(1+\cah,1.8+\sah)--(1+\cah,1.5+\sah)--cycle;
\draw[orange] (-0.2+\cah,1.8+\sah)--(-.2+\ca,1.8+\sa)--(1+\ca,1.8+\sa)--(1+\cah,1.8+\sah);
\draw[orange] (1+\ca,1.8+\sa)--(1+\ca,1.5+\sa)--(1+\cah,1.5+\sah);
\draw[orange, opacity=.5, dashed] (-.2+\ca,1.8+\sa)--(-.2+\ca,1.5+\sa)--(-.2+\cah,1.5+\sah);
\draw[orange, opacity=.5, dashed] (-.2+\ca,1.5+\sa)--(1+\ca,1.5+\sa);

\node[orange] at (1.6,1.9 ) {$\Omega$};

\draw[orange] (2,0)--(2,1)--(2.5,1)--(2.5,0)--cycle;
\draw[orange] (2,1)--(2+\cah,1+\sah)--(2.5+\cah,1+\sah)--(2.5,1);
\draw[orange] (2.5+\cah,1+\sah)--(2.5+\cah,\sah)--(2.5,0);
\draw[orange, opacity=.5, dashed] (2+\cah,1+\sah)--(2+\cah,\sah)--(2,0);
\draw[orange, opacity=.5, dashed] (2+\cah,\sah)--(2.5+\cah,\sah);

\begin{scope}[scale=.6,
xshift=-4cm]
    \draw[->] (0,0)--(.5,0);
   \node at (.6,0) {\footnotesize $x_1$};
    \draw[->] (0,0)--(0,.5);
   \node at (0,.6) {\footnotesize $x_2$};
   \draw[->] (0,0)--(\cah,\sah);
   \node at ({1.5*\cah},{1.5*\sah}) {\footnotesize$x_3$};
\end{scope}

\pgfmathsetmacro{\sashift}{-50*sin(30)}
\pgfmathsetmacro{\cashift}{-50*cos(30)}

\begin{scope}[xshift=\cashift , yshift=\sashift ]
\draw[orange] (0,0)--(0,1)--(.5,1)--(.5,.5)--(1,.5)--(1,0)--cycle;
\draw[orange] (0,1)--(\ca,1+\sa)--(.5+\ca,1+\sa)--(.5,1);
\draw[orange] (.5,.5)--(.5+\ca,.5+\sa)--(1+\ca,.5+\sa)--(1,.5);
\draw[orange] (1+\ca,.5+\sa)--(1+\ca,\sa)--(1,0);
\draw[orange] (.5+\ca,1+\sa)--(.5+\ca,.5+\sa);
\draw[orange, opacity=.5, dashed] (\ca,1+\sa)--(\ca,\sa)--(0,0);
\draw[orange, opacity=.5, dashed] (\ca,\sa)--(1+\ca,\sa);
\end{scope}

\end{scope}

\end{tikzpicture}
\caption{Illustration of the setting from Remark ~\ref{rmk:weakantilochigher} with $p_1=0=p_2$, $p_3=1$. 
Choosing for instance $\{n_1,\dots, n_N\}=\{1,2,3,1\}$, we can conclude $u=0$ in $\Omega\cap (U+\cup_{k=1}^3Z_k)$.
}
\label{fig:weakantiloc3d}
\end{figure}

\begin{rmk}
We remark that in none of the results from this section, it played a role that $A^s_{\boldsymbol{p}}(D)$ was defined as the sum of one-dimensional fractional Laplacians with the same index $s$. All the above presented results remain valid for operators of the form
\begin{align*}
    A^{\boldsymbol{s}}_{\boldsymbol{p}}(D):=\sum_{k=1}^n A^{s_k}_{p_{k}}(D_{x_k}),
\end{align*}
where $\boldsymbol{s} = (s_1,\dots,s_n) \in (0,1)^n$. 
Harnack inequalities for these and similar operators have recently been studied in \cite{CK20}.
\end{rmk}

\subsection{Inverse operators inheriting directional-antilocality}

Here, as a model setting, we consider the following one-dimensional operator for $s\in\big(0,\frac 12\big)$
\begin{align*}
    B_{p}^s(D)f(x)
    &=
    \int_{-\infty}^\infty f(x+y)\big(p\chi_{\R_-}(y)+(1-p)\chi_{\R_+}(y)\big)\frac{dy}{|y|^{1-2s}}.
\end{align*}
This corresponds, up to a constant, to $(A_{p}^s(D))^{-1}$ (see Lemma ~\ref{lem:AspInv}).
Reducing the antilocality of $B_p^s(D)$ to that of $A_p^s(D)$, we can see that the two operators enjoy the same $Y_p$-antilocality.

\begin{lem}
\label{lem:inv_anti}
Let $p\in[0,1]$ and $s\in\big(0,\frac 12\big)$. Then the operator $B_{p}^s(D)$ is $Y_p$-antilocal, i.e. it is antilocal to the right if $p=0$, antilocal to the left if $p=1$ and antilocal (to both directions) for $p\in (0,1)$.
\end{lem}

\begin{proof}
Let $u\in C^\infty_c(\R)$ and  $U\subset \R$ be such that $u=0=B^s_p(D)u$ in $U$.
Let $v:=B^s_p(D) u$. By Lemma ~\ref{lem:AspInv} we thus have $A^s_p (D)v=\tilde c_s(p)^{-1}u$. Then, by definition, $v=0=A^s_p(D) v$ in $U$ and by the $Y_p$-antilocality of $A^s_p(D)$, we infer that $v=0$ in $U+Y_p$. By the domain dependence of the operator $A^s_p(D)$, we finally conclude  $u=\tilde c_s(p) A^s_p(D) v=0$ also in $U+Y_p$.
\end{proof}

 We remark that, while we proved the above result for the operator $B_p^s(D)$, the proof shows that the statement encodes a more general principle stating that  antilocality of a given operator also imply the antilocality of the inverse operator:

 \begin{prop}
 Let $\Gamma \subset \R^n$ be an open, non-empty, convex cone,
let $L(D)$ be an elliptic, $\Gamma$-antilocal operator and let $L(D)^{-1}$ denote its (whole-space) inverse operator. Then, $L(D)^{-1}$ is also $\Gamma$-antilocal.
 \end{prop}
 
  \begin{proof}
 The proof follows the same scheme as the one of Lemma \ref{lem:inv_anti}. 
 \end{proof}

\appendix

\section{Symbol Calculations and Inverses of the Operators}
\label{sec:symbol}

This appendix contains some results involving properties of the operators from the previous sections and their symbols.

\subsection{Fourier symbols of the operators}
We begin by computing the  symbols  of the operator $L$ in \eqref{eq:L} and the operator $A^s_{p,\Gamma}$ in  \eqref{eq:Aspgamma}.

\begin{lem}\label{lem:symbol}
The symbol of the operator $L$ defined in \eqref{eq:L} with $s\in(0,1)$ is given by:
\begin{align*}
    L(\xi)= 2c_s\int\limits_{\Ss^{n-1}}|\xi \cdot \theta|^{2s} a(\theta) d\theta, 
\end{align*}
where $c_s=-\Gamma(-2s)\cos(\pi s)$.
\end{lem}

\begin{proof}
If $u\in \mathcal D(\R^n)$, for any $\xi\in \R^n$ we have 
\begin{align*}
    \mathcal F\big(Lu\big)(\xi)=\int_{\R^n}\big(2-e^{iy\cdot\xi}-e^{-iy\cdot\xi}\big)\mathcal Fu(\xi) \frac{a(y/|y|)}{|y|^{n+2s}}dy=L(\xi) \mathcal Fu(\xi).
\end{align*}
Therefore
\begin{align*}
    L(\xi)
    &=2\int_{\R^n}\big(1-\cos(y\cdot\xi)\big)\ \frac{a(y/|y|)}{|y|^{n+2s}}dy
    =2\int_{\Ss^{n-1}}\int_0^\infty \big(1-\cos(r\theta\cdot\xi)\big)\frac{a(\theta)}{r^{2s+1}}drd\theta
    \\ &=2 \int_0^\infty \frac{1-\cos(\rho)}{\rho^{2s+1}} d\rho \int_{\Ss^{n-1}} |\xi\cdot\theta|^{2s}a(\theta) d\theta
    =2c_s \int_{\Ss^{n-1}} |\xi\cdot\theta|^{2s}a(\theta) d\theta.
\end{align*}
The constant $c_s$, which is finite for $s\in (0,1)$, can be computed as follows:
\begin{align}\label{eq:intcos}
\begin{split}
    c_s&=\int_0^\infty \frac{1-\cos(\rho)}{\rho^{2s+1}} d\rho 
    = -\frac{1}{2s}\int_0^\infty \big(1-\cos(\rho)\big)\left(\frac{1}{\rho^{2s}}\right)'d\rho
    =\frac{1}{2s}\int_0^\infty \frac{\sin(\rho)}{\rho^{2s}}d\rho
    \\
    &=\frac{1}{2s}\Gamma(1-2s)\sin\Big(\frac{\pi}{2}(1-2s)\Big)
    =-\cos(\pi s)\Gamma(-2s).
    \end{split}
\end{align}
The value of $\int_0^\infty \rho^{(1-2s)-1}\sin(\rho) d\rho$ is given e.g. in  \cite[5.9.7]{NIST}.
Notice that we have used the fact that $s\in(0,1)$ for both this identity and the integration by parts arguments.  
\end{proof}

\begin{lem}\label{lem:AspgammaSymb}
The Fourier symbol of the operator $A^s_{p,\Gamma}(D)$ given in \eqref{eq:Aspgamma} with $p\in[0,1]$ and $\Gamma\subset \R^n \backslash\{0\}$  an open, convex cone is
\begin{align*}
    A_{p,\Gamma}^s(\xi) =
    \begin{cases}
     c_{s} \int_{\Gamma\cap \Ss^{n-1}} \Big(1- i(1-2p)\tan(\pi s)\sign(\theta\cdot\xi)\Big)|\theta\cdot\xi|^{2s}d\theta & \mbox{ if } s\in\big(0,\frac 1 2\big)\cup\big(\frac 1 2, 1\big)\;,
    \\
    \int_{\Gamma\cap \Ss^{n-1}} \Big(\frac \pi 2- i(1-2p)\sign(\theta\cdot\xi)\log(|\theta\cdot\xi|)\Big)|\theta\cdot\xi|d\theta & \mbox{ if } s=\frac 1 2 ,\;
    \end{cases}
\end{align*}
where $c_{s}=-\Gamma(-2s)\cos(\pi s)$.

\end{lem}

\begin{proof} 
The symbol of the operator \eqref{eq:Aspgamma} in the case of $s\in\big(0,\frac 1 2\big)$ is given by
\begin{align*}
    A_{p,\Gamma}^s(\xi) & = \int_{\mathbb R^n} (1-e^{iy\cdot\xi})\frac{p\chi_{-\Gamma}(y)+(1-p)\chi_{\Gamma}(y)}{|y|^{n+2s}}dy 
    \\&=\frac{1}{2}\int_{(-\Gamma\cup\Gamma )\cap\Ss^{n-1}}\int_0^\infty (1-e^{ir\theta\cdot\xi})\frac{1+(2p-1)(\chi_{-\Gamma}(\theta)-\chi_{\Gamma}(\theta))}{r^{1+2s}}drd\theta
    \\ &= \frac{1}{2}\int_{ \Gamma\cap\Ss^{n-1}}\int_0^\infty\frac{2-e^{ir\theta\cdot\xi}-e^{-ir\theta\cdot\xi}}{r^{1+2s}}drd\theta
     +\frac 12 (2p-1)  \int_{\Gamma\cap\Ss^{n-1}}\int_0^\infty\frac{e^{ir\theta\cdot\xi}-e^{-ir\theta\cdot\xi}}{r^{1+2s}}drd\theta
    \\&=\int_{\Gamma\cap\Ss^{n-1}}\int_0^\infty\frac{1-\cos(r\theta\cdot\xi)}{r^{1+2s}}drd\theta
    -i(1-2p)\int_{\Gamma\cap\Ss^{n-1}}\int_0^\infty\frac{\sin(r\theta\cdot\xi)}{r^{1+2s}}drd\theta.
\end{align*}
Applying \eqref{eq:intcos} to the first integral, which is valid for $s\in(0,1)$ (with $-\Gamma(-2s)\cos(\pi s)=\frac \pi 2$ for $s=\frac 12$),  
and  \cite[(5.9.7)]{NIST} to the second integral (valid only for $s\in\big(0,\frac 12\big)$),  it  follows that
\begin{align*}
    A_{p,\Gamma}^s(\xi)
    &=-\Gamma(-2s)\int_{\Gamma\cap\Ss^{n-1}}
    \big(\cos(\pi s)-i(1-2p)\sin(\pi s)\sign(\theta\cdot\xi)\big)|\theta\cdot\xi|^{2s}
   d\theta.
\end{align*}

In the case of $s\in\big(\frac 1 2 ,1\big)$, we have 
\begin{align*}
    A_{p,\Gamma}^s(\xi) & = \int_{\mathbb R^n} (1-e^{iy\cdot\xi}+iy\cdot\xi)\frac{p\chi_{-\Gamma}(y)+(1-p)\chi_{\Gamma}(y)}{|y|^{n+2s}}dy.
\end{align*}
We observe that the symmetric contribution does not change, while the antisymmetric one is now given by 
\begin{align*}
    &i(2p-1)\int_{\Gamma\cap\Ss^{n-1}}\int_0^\infty\frac{\sin(r\theta\cdot\xi)-r\theta\cdot\xi}{r^{1+2s}}drd\theta
    \\ &\quad =i(2p-1)\int_{\Gamma\cap\Ss^{n-1}}\sign(\theta\cdot\xi)|\theta\cdot\xi|^{2s}\left(\int_0^\infty\frac{\sin(\rho)-\rho}{\rho^{1+2s}}d\rho d\right)\theta.
\end{align*}
The integral can be computed as follows, taking into account that  $2s-1\in(0,1)$  and finally applying  \eqref{eq:intcos}:
\begin{align*}
    \int_0^\infty\frac{\sin(\rho)-\rho}{\rho^{1+2s}}d\rho
    &=\frac{-1}{2s}\int_0^\infty \left(\frac{1}{\rho^{2s}}\right)'\big(\sin(\rho)-\rho\big)d\rho
    \\&=\frac{1}{2s}\int_0^\infty \frac{\cos(\rho)-1}{\rho^{2s}}d\rho
    =\frac{\sin(\pi s) \Gamma(1-2s)}{2s}
    =-\sin(\pi s) \Gamma(-2s).
\end{align*}
Therefore, the symbol of the antisymmetric part is the same as for $s\in\big(0,\frac 12)$. 

Finally, for $s=\frac 12$, we can argue similarly for the antisymmetric part. This leads to the integral 
\begin{align*}
    \int^{\infty}_0 \frac{\rho e^{-\frac{\rho}{|\theta\cdot\xi|}}-\sin(\rho)}{\rho^2}d\rho
    =\lim_{\epsilon\to 0^+} \left( \text{Ci}(\epsilon)-\frac {\sin(\epsilon)}{\epsilon}+\text{E}_1(\epsilon |\theta\cdot\xi|)\right),
\end{align*}
where $\text{Ci}$ and $\text{E}_1$ are the cosine integral and exponential integral, respectively, defined in \cite[(6.2.11), (6.2.1)]{NIST}.
By the power expansions in \cite[(6.6.6), (6.6.2)]{NIST}, 
\begin{align*}
    \int^{\infty}_0 \frac{\rho e^{-\frac{\rho}{e|\theta\cdot\xi|}}-\sin(\rho)}{\rho^2}d\rho
    =\lim_{\epsilon\to 0^+} \left( -\frac{\sin(\epsilon)}{\epsilon}+\log(\epsilon)-\log\Big(\frac{\epsilon}{ e|\theta\cdot\xi|}\Big)+O(\epsilon)\right)=
    \log(|\theta\cdot\xi|).
\end{align*}
\end{proof}

\subsection{Inverse of the operators}

This section is devoted to the derivation of an explicit real-space formula for the inverse operators of $A^s_{p}(D)$ in the one-dimensional case.

\begin{lem}[Inverse operator, 1D case]
\label{lem:AspInv}
The inverse of the operator $A^s_{p}(D)$ defined in \eqref{eq:example1D} for $s\in\big(0,\frac 1 2\big)$  and $p\in[0,1]$
is given by 
\begin{align*}
    (A_{p}^s(D))^{-1}f(x)
    &=\tilde c_s(p)
    \int_{-\infty}^\infty f(x+y)\big(p\chi_{\R_-}(y)+(1-p)\chi_{\R_+}(y)\big)\frac{dy}{|y|^{1-2s}},
\end{align*}
where
$$  \tilde c_s(p) := \frac{4s\sin(2\pi s)}{1+2p(1-p)\big(\cos(2\pi s)-1\big)}\;.$$

\end{lem}

\begin{proof}
By Lemma ~\ref{lem:AspgammaSymb}, the symbol of the inverse operator $(A_{p}^s(D))^{-1}$ is given by
\begin{align*}
    (A_{p}^s(\xi))^{-1}
    &=\frac{2s}{\Gamma(1-2s)}\frac{1}{\cos(\pi s)-i(1-2p)\sin(\pi s)\sign(\xi)} |\xi|^{-2s}\\
    &=\frac{2s}{\Gamma(1-2s)}\frac{\cos(\pi s)+i(1-2p)\sin(\pi s)\sign(\xi)}{\cos^2(\pi s)+(1-2p)^2\sin^2(\pi s)} |\xi|^{-2s}.
\end{align*}

Now, we notice  the following correspondence between  symbols and operators: 
\begin{align*}
    |\xi|^{-2s} f(\xi) &\quad \overset{\mathcal F^{-1}}{\longrightarrow} \quad 2\Gamma(1-2s)\sin(\pi s)\int_{-\infty}^\infty f(x-y)\frac{dy}{|y|^{1-2s}},\\
   - i\sign(\xi) |\xi|^{-2s}f(\xi) & \quad\overset{\mathcal F^{-1}}{\longrightarrow}\quad 2\Gamma(1-2s)\cos(\pi s)\int_{-\infty}^\infty f(x-y)\sign(y)\frac{dy}{|y|^{1-2s}}.
\end{align*}
The first one is just the Riesz potential, which holds for $s\in \big(0, \frac 12\big)$. Here the constant can be rewritten as $\frac{\pi^{\frac 12}}{2^{2s}}\frac{\Gamma\big(\frac 1 2-s\big)}{\Gamma(s)}$. Both correspondences can be obtained from \cite[(5.9.6)-(5.9.7)]{NIST}.

Therefore,
\begin{align*}
    \big(A_{p}^s(D)\big)^{-1}f(x)
    &=\frac{2s}{\Gamma(1-2s)}\frac{2\Gamma(1-2s)\cos(\pi s)\sin(\pi s)}{\cos^2(\pi s)+(1-2p)^2\sin^2(\pi s)} 
    \int_{-\infty}^\infty f(x-y)\frac{1-(1-2p)\sign(y)}{|y|^{1-2s}}dy
    \\&= \tilde c_s(p) \int_{-\infty}^\infty f(x-y)\big((1-p)\chi_{\R_-}(y)+p\chi_{\R_+}(y)\big)\frac{dy}{|y|^{1-2s}}
    \\&=\tilde c_s(p) \int_{-\infty}^\infty f(x+y)\big(p\chi_{\R_-}(y)+(1-p)\chi_{\R_+}(y)\big)\frac{dy}{|y|^{1-2s}},
\end{align*}
with
\begin{align*}
     \tilde c_s(p)
    &=\frac{8s\cos(\pi s)\sin(\pi s)}{\cos^2(\pi s)+(1-2p)^2\sin^2(\pi s)}
    =\frac{4s\sin(2\pi s)}{1+2p(1-p)\big(\cos(2\pi s)-1\big)}.
\end{align*}
\end{proof}

\section{A ``More Local'' Variant of the Problem from Section ~\ref{sec:exs1}}
\label{sec:loc_B}

In this appendix we present and discuss an alternative bilinear form associated with the problem \eqref{eq:AspgammaExt}. We compare the resulting notion with the one from Section ~\ref{sec:exs1}.

From now on, we assume $\Omega\subset \R^n$ is an  bounded, Lipschitz open domain, $\Gamma\subset\R^n\backslash\{0\}$ is an open, non-empty, convex cone and $s\in(0,1)\backslash\big\{\frac 12\big\}$.

\subsection{An alternative bilinear form}
As an alternative version of the bilinear form from Section ~\ref{sec:exs1}, we here introduce the following 
``more local'' bilinear form for $s\neq \frac 12$:
\begin{align}\label{eq:Bspgammaq}
\begin{split}
    B^s_{p,\Gamma;q}(u,v)
    & := \beta_s
    \int_{\Omega+\Gamma_p}\int_{\Gamma_p}\big(u(x)-u(x+y)\big)\big(v(x)-v(x-y)\big)\chi_{\Omega+\Gamma_p}(x+y)
    \nu^s_{p,\Gamma}(y)dydx
    \\
    &\quad +(qu,v)_{L^2(\Omega)},
\end{split}
\end{align}
where $\beta_s=(2-2^{2s})^{-1}$, $\chi_{\Omega+\Gamma_p}$ is the characteristic function of $\Omega+\Gamma_p$, and $\nu^s_{p,\Gamma}$ is given in \eqref{eq:nuspgamma}.
Equivalently, we may write
\begin{align*}
    B^s_{p,\Gamma;q}(u,v)
    & = \beta_s
    \int_{\Omega+\Gamma_p}\int_{\Omega+\Gamma_p}\big(u(x)-u(z)\big)\big(v(x)-v(2x-z)\big)
    \nu^s_{p,\Gamma}(z-x)dydx
    \\
    &\quad +(qu,v)_{L^2(\Omega)}.
\end{align*}

We observe that the first argument of $B^s_{p,\Gamma;q}$ is ``rather local", in the sense that for the first slot occupied by the function $u$ it suffices to be defined in the domain of dependence $\Omega+\Gamma_p$, while only the second one is required to be defined in $\R^n$ (due to the shift in the arguments of $v$).

\begin{rmk}
\label{rmk:consistent1}
We emphasize that the bilinear form $B^{s}_{p,\Gamma;q}(u,v)$ evaluated at a function $v$ supported in $\Omega$ is consistent with the weak form of the equation \eqref{eq:AspgammaExt} in the sense that it gives rise to a weak form of the interior equation in \eqref{eq:AspgammaExt}. In other words, it can be derived from the strong form of the (interior) equation in \eqref{eq:AspgammaExt} if this is tested  with  $v$, after a suitable change of variables and suitable recombinations of these expressions, which ultimately leads to the stated expression.
For the sake of simplicity, we discuss this for the case $p=1$  and $q=0$.
For $s\in \big(0,\frac 12)$, applying different changes of variables, we obtain the following identities for $v\in \widetilde H^s(\Omega)$: 
\begin{align*}
    \int_\Omega \big(A^s_{p,\Gamma}(D) u(x)\big) v(x)dx
    &=\int_{\Omega}\int_{\Gamma_p} \big(u(x)-u(x+y)\big)v(x)\nu^s_{p,\Gamma}(y)dydx
    \\&=\int_{\Omega+\Gamma_p}\int_{\Gamma_p} \big(v(x)-v(x-y)\big)u(x)\nu^s_{p,\Gamma}(y)dydx
    \\&=2^{-2s}\int_{\Omega+\Gamma_p}\int_{\Gamma_p} \big(v(x)-v(x-2y)\big)u(x)\nu^s_{p,\Gamma}(y)dydx
    \\&=2^{-2s}\int_{\Omega+\Gamma_p}\int_{\Gamma_p} \big(v(x+y)-v(x-y)\big)u(x+y)\chi_{\Omega+\Gamma_p}(x+y)\nu^s_{p,\Gamma}(y)dydx.
\end{align*}
We remark that in the last identity the characteristic function $\chi_{\Omega+\Gamma_p}(x+y)$ is not necessary due to the imposed support assumption for $v$, but it will be necessary in defining the full bilinear expression (for functions which are not necessarily supported in $\Omega$). 

By suitably combining the previous integral expressions, we infer
\begin{align*}
    (2-2^{2s})&\int_\Omega \big(A^s_{p,\Gamma}(D) u(x)\big)v(x)dx
    \\&=\int_{\Omega}\int_{\Gamma_p} \big(u(x)-u(x+y)\big)v(x)\nu^s_{p,\Gamma}(y)dydx
    \\
    &\quad +\int_{\Omega+\Gamma_p}\int_{\Gamma_p} \big(v(x)-v(x-y)\big)u(x)\nu^s_{p,\Gamma}(y)dydx
    \\&\quad-\int_{\Omega+\Gamma_p}\int_{\Gamma_p} \big(v(x+y)-v(x-y)\big)u(x+y)\chi_{\Omega+\Gamma_p}(x+y)\nu^s_{p,\Gamma}(y)dydx
    \\
    &=\int_{\Omega+\Gamma_p}\int_{\Gamma_p}\big(u(x)-u(x+y)\big)\big(v(x)-v(x-y)\big)\chi_{\Omega+\Gamma_p}(x+y)\nu^s_{p,\Gamma}(y)dydx.
\end{align*}
If $s\in \big(\frac 12,1\big)$, we repeat these arguments with additional terms of the form $\nabla u(x) v(x)$, but these finally cancel out.
\end{rmk}

\begin{rmk}
We highlight that the derivation of the bilinear form $B_{\sfrac{1}{2},\Gamma;q}$ differs from that of $B_q$ from the beginning of this article. In particular, this leads to a lack of symmetry of $\nu^s_{p,\Gamma}$, except if $p=\frac12$, but allows to treat the cases $B_{p,\Gamma;q}$ with $p\in [0,1]$ simultaneously.
The difference between \eqref{eq:Bspgammaq} for $p=\frac 12$ and \eqref{eq:bil} with $a(\theta)=\frac 12\chi_{\Ss^{n-1}\cap(-\Gamma\cup\Gamma)}(\theta)$ corresponds to
\begin{align*}
    &B^s_{\sfrac12,\Gamma;q}(u,v)-B_q(u,v)
    \\&\quad=\frac 12\int_{\Omega+\Gamma_{\sfrac12}}
    \int_{\Omega+\Gamma_{\sfrac12}} \big(u(x)-u(z)\big) \left(\frac{v(x)}{2^{1-2s}-1} -\frac{v(2x-z)}{1-2^{2s-1}}+v(z)\right) \nu^s_{\sfrac12,\Gamma}(x-z)dzdx.
\end{align*}
\end{rmk}

Comparing the two bilinear forms from \eqref{eq:Bspgammaq} and \eqref{eq:tBspgammaq}, we collect the following observations:

\begin{lem}
\label{lem:equiv1}
Let $s\in(0,1)\backslash\big\{\frac 12\big\}$ and $p\in[0,1]$. Let $u,v\in H^s(\R^n)$. Then
$B^s_{p,\Gamma;q}(u,v)=\tilde B^s_{p,\Gamma;q}(u,v)$
if one of the following holds:
\begin{itemize}
    \item $v\in \widetilde H^s(\Omega)$,
    \item $p\in(0,1)$ and $u\in \widetilde H^s(\Omega)$,
    \item $p\in\{0,1\}$, $u\in \widetilde H^s(\Omega)$ and $v=0$ in  $(\Gamma_{1-p}+\Omega)\backslash\overline{\Omega}$.
\end{itemize}
\end{lem}

\begin{proof}
We deal with the three cases separately. If $v\in\widetilde H^s(\Omega)$, by the arguments from Remark ~\ref{rmk:consistent1} we note that the bilinear forms agree (and can be reduced to the same strong form of the equation).

Let now $u\in\widetilde H^s(\Omega)$. According to  the definitions \eqref{eq:Bspgammaq} and \eqref{eq:tBspgammaq}, the difference of the bilinear forms is given by
\begin{align*}
    &\tilde B^s_{p,\Gamma;q}(u,v)-B^s_{p,\Gamma;q}(u,v)
    \\&\quad=-\beta_s\int_{\R^n\backslash(\overline{\Omega}+\Gamma_p)}\int_{\Gamma_p}u(x+y)\big(v(x)-v(x-y)\big)
    \nu_{p,\Gamma}(y)dydx
    \\ &\qquad + \beta_s\int_{\Omega}\int_{\Gamma_p}u(x)\big(v(x)-v(x-y)\big)
    \nu_{p,\Gamma}(y)\chi_{\R^n\backslash(\overline{\Omega}+\Gamma_p)}(x+y)dydx\\
    &\quad =: I_1+I_2.
\end{align*}
On the one hand, the second contribution $I_2$ always vanishes since $x+y\in\Omega+\Gamma_p$ for all $x\in\Omega$ and $y\in\Gamma_p$.
On the other hand, the integrand in $I_1$, in general, needs not vanish  if $x+y\in\Omega$ for $x\in \R^n\backslash(\overline{\Omega}+\Gamma_p)$ and $y\in \Gamma_p$. This situation cannot happen for $p\in(0,1)$ due to the two-sidedness of the cone.

However, if $p\in\{0,1\}$ and $x\in (\Omega+\Gamma_{1-p})\backslash\overline{\Omega}\subset\R^n\backslash(\overline{\Omega}+\Gamma_p) $, then $x+y\in \Omega$ for some $y\in\Gamma_p$.
Therefore, $I_1=0$ requires that $v(x)-v(x-y)=0$ a.e. for $x\in \Omega+\Gamma_{1-p}$ and $y\in\Gamma_p$, i.e. $v$ is constant a.e. in $(\Omega+\Gamma_{1-p})\backslash\overline{\Omega}$. But since the subset is unbounded and we require $v\in L^2(\R^n)$, it holds that $v=0$ a.e. in $(\Omega+\Gamma_{1-p})\backslash\overline{\Omega}$.
\end{proof}

\subsection{Function spaces}
\label{sec:app_functions}
We next discuss the well-posedness of the direct problem associated with the bilinear form $B^s_{p,\Gamma;q}$.

Heading towards an analysis of the exterior problem, we introduce the following space for $s\in(0,1)$, $\Gamma\subset\R^n\backslash\{0\}$ an open, non-empty, convex cone and $\Omega\subset\R^n$ open:
\begin{align*}
    \mathcal V^s_{p,\Gamma}(\Omega)
    :=\Big\{u:\Omega+\Gamma_p\to\R: \ u\in L^2(\Omega), \ \frac{u(x)-u(x+y)}{|y|^{\frac n2+s}}\in L^2(\Omega\times \Gamma_p)\Big\}
\end{align*}
endowed with the norm
\begin{align*}
    \|u\|_{\mathcal V^s_{p,\Gamma}(\Omega)}^2
    :=\|u\|_{L^2(\Omega)}^2+[u,u]_{\mathcal V^s_{p,\Gamma}(\Omega)},
\end{align*}
where 
\begin{align*}
    [u,v]_{\mathcal V^s_{p,\Gamma}(\Omega)}=
    \int_\Omega\int_{\Gamma_p}
    \big(u(x)-u(x+y)\big)\big(v(x)-v(x+y)\big)\frac{p\chi_{-\Gamma}(y)+(1-p)\chi_\Gamma(y)}{|y|^{n+2s}} dydx.
\end{align*}

Analogously to \eqref{eq:incl} and \eqref{eq:equiv}, it holds that
\begin{align*}
    H^s(\R^n)&=\mathcal V^s_{p,\Gamma}(\R^n),\\
   H^s(\Omega+\Gamma_p)&\subset \mathcal V^s_{p,\Gamma}(\Omega).
\end{align*}
The first statement follows as in Lemma ~\ref{lem:HVR}.
The second inclusion can be proved as  Lemma ~\ref{lem:HVOm}.

Lastly, we introduce the space for the exterior data:
\begin{align*}
    \mathcal V^s_{p,\Gamma}(\Omega)_e
    :=\raisebox{.3em}{$\mathcal V^s_{p,\Gamma}(\Omega)$}
    \Big/\raisebox{-.3em}{$\widetilde H^s(\Omega)$},\quad \mbox{equipped with the quotient topology}.
\end{align*}

Comparing the $H^s_{p,\Gamma}$ and $\mathcal V^s_{p,\Gamma}$ based spaces (see Section ~\ref{sec:functionspacesA} for the definition of $H^{s}_{p,\Gamma}$), we observe that 
\begin{align*}
   H^s_{p,\Gamma}(\Omega)&\subset \mathcal V^s_{p,\Gamma}(\Omega)_e.
\end{align*}

Building on this, analogously to Lemma ~\ref{lem:bilcont}, we infer a corresponding bound for $B^s_{p,\Gamma;q}$: 

\begin{lem}
\label{lem:bil_bound}
Let $s\in (0,1)\backslash\big\{\frac12\big\}$, $p\in [0,1]$ and $B^s_{p,\Gamma;q}$ be as in \eqref{eq:Bspgammaq}.
If $u\in \mathcal V^s_{p,\Gamma}(\Omega+\Gamma_p)$ and $v\in H^s(\R^n)$, then
\begin{align}\label{eq:BspgammaBound}
    |B^s_{p,\Gamma;q}(u,v)|
    &\leq C \|u\|_{\mathcal V^s_{p,\Gamma}(\Omega+\Gamma_p)}
    \|v\|_{H^{s}(\R^n)}.
\end{align}
\end{lem}

\begin{proof}
The argument for the estimate is analogous to the ones which were used for the bilinear form $B_q$ in Section ~\ref{sec:direct}. Indeed, we have
\begin{align*}
    |B^s_{p,\Gamma;0}(u,v)|
    &\leq  \beta_s [u,u]_{\mathcal V^s_{p,\Gamma}(\Omega+\Gamma_p)}^{\frac 12} [v,v]_{\mathcal V^s_{1-p,\Gamma}(\Omega+\Gamma_p)}^{\frac12}
    \leq \beta_s [u,u]_{\mathcal V^s_{p,\Gamma}(\Omega+\Gamma_p)}^{\frac 12} [v,v]_{\mathcal V^s_{1-p,\Gamma}(\R^n)}^{\frac12}
    \\&\leq C[u,u]_{\mathcal V^s_{p,\Gamma}(\Omega+\Gamma_p)}^{\frac 12} \|v\|_{H^s(\R^n)}.
\end{align*}
\end{proof}

\subsection{The exterior problem}

With the properties of the function spaces associated with the bilinear form $B^s_{p,\Gamma;q}$ in hand, we define our notion of a weak solution based on this bilinear form:

\begin{defi}
Let $s\in(0,1)\backslash\big\{\frac 12\big\}$, $p\in[0,1]$ and let $B^s_{p,\Gamma;q}$ be the bilinear form in \eqref{eq:Bspgammaq}. Given $f\in \mathcal V^{s}_{p,\Gamma}(\Omega)_e$,  a function $u\in  \mathcal V^s_{p,\Gamma}(\Omega)$ is a \emph{(weak) solution of \eqref{eq:AspgammaExt} (based on $B^s_{p,\Gamma;q}$)} if
\begin{align*}
    &B^s_{p,\Gamma;q}(u,v)=0\; \mbox{ for all } v\in \widetilde H^s(\Omega)\\
    &\mbox{and }\ \mathcal E_{\Omega+\Gamma_p}(u-\tilde f)\in \widetilde H^s(\Omega) \ \mbox{ for any } \tilde f\in V^{s}_{p,\Gamma}(\Omega) \mbox{ with }\tilde f\in [f].
\end{align*}
\end{defi}

We emphasize that the more local character of the bilinear form $B_{p,\Gamma;q}$ is of significance here in that in its first slot $B_{p,\Gamma;q}$ only needs to be defined on $\Omega+ \Gamma_p$.

We can now deduce the desired well-posedness result.

\begin{prop}\label{prop:AspgammaExt}
Let $s\in (0,1)\backslash\big\{\frac12\big\}$, $p\in[0,1]$,  $\Gamma\subset \R^n\backslash\{0\}$ be an open, convex cone, $\Omega\subset\R^n$ be open, bounded, Lipschitz and $q\in L^\infty(\Omega)$.
Then there is a countable set $\Sigma^s_{p,\Gamma;q}\subset\C$ such that if $\lambda\notin\Sigma^s_{p,\Gamma;q}$, for any $f\in  \mathcal V^{s}_{p,\Gamma}(\Omega)_e$, there is a unique solution $u\in  \mathcal V^s_{p,\Gamma}(\Omega)$ of
\begin{align*}
    \big(A^s_{p,\Gamma}(D)+q-\lambda\big)u&=0\; \mbox{ in } \Omega,\\
    u&=f \; \mbox{ in } (\Omega+\Gamma_p)\backslash\overline{\Omega}.
\end{align*}
In that case, 
\begin{align*}
    \|u\|_{\mathcal V^{s}_{p,\Gamma}(\Omega)}\leq C\|f\|_{\mathcal V^{s}_{p,\Gamma}(\Omega)_e}.
\end{align*}
\end{prop}

\begin{proof}
We reduce this problem  to the interior one in \eqref{eq:AspgammaInh}, solved in Proposition ~\ref{prop:AspgammaInh}.
Let $f\in \mathcal V^{s}_{p,\Gamma}(\Omega)_e$ and let $\tilde f\in \mathcal V^{s}_{p,\Gamma}(\Omega)$ such that $\tilde f\in[f]$, i.e. 
$\mathcal E_{\Omega+\Gamma_p}(\tilde f-f)\in \widetilde H^s(\Omega)$. 
We  construct $u=w+\tilde f$, where 
$w\in\widetilde H^s(\Omega)$ is the solution of the inhomogeneous problem \eqref{eq:AspgammaExt} with $g=-(A^s_{p,\Gamma}(D)+q-\lambda)\tilde f|_\Omega$. This requires proving that $A^s_{p,\Gamma} \tilde f|_\Omega\in H^{-s}(\Omega)$ (weakly) for $\tilde f\in \mathcal V^s_{p,\Gamma}(\Omega)$.
We present here the proof for $p=0$. As usually, we can derive similar results for $p=1$ and the general case follows by the fact that $A^s_{p,\Gamma}(D)=p A^s_{1,\Gamma}(D)+(1-p)A^s_{0,\Gamma}(D)$.
For any $v\in \widetilde H^s(\R^n)$, testing $A^s_{0,\Gamma}(D) \tilde f$ with $v$ we obtain:
\begin{align*}
    &\int_\Omega\int_\Gamma \frac{u(x)-u(x+y)}{|y|^{n+2s}}v(x)dydx
    \\&\qquad= \int_\Omega\int_{\Gamma \cap (-x+\Omega)} \frac{u(x)-u(x+y)}{|y|^{n+2s}}v(x)dydx
    +\int_\Omega\int_{\Gamma \backslash (-x+\Omega)} \frac{u(x)-u(x+y)}{|y|^{n+2s}}v(x)dydx
    \\&\qquad =I_1+I_2.
\end{align*}
Taking into account that $v$ is supported in $\Omega$ and making a suitable change of variables in $I_1$, we can rewrite both integrals as follows:
\begin{align*}
    I_1
    &=\frac 12 \int_\Omega\int_{\Gamma \cap (-x+\Omega)}  \frac{\big(u(x)-u(x+y)\big)\big(v(x)-v(x-y)\big)}{|y|^{n+2s}}dydx
    \\
    I_2 
     &= \int_\Omega\int_{\Gamma \backslash (-x+\Omega)} \frac{\big(u(x)-u(x+y)\big)\big(v(x)-v(x+y)\big)}{|y|^{n+2s}}dydx
\end{align*}
Then 
\begin{align*}
    |I_1|
    &\leq \frac 12 [u,u]_{\mathcal V^s_{0,\Gamma}(\Omega)}^{\frac 12}
    [v,v]_{\mathcal V^s_{1,\Gamma}(\Omega)}^{\frac 12}
    \leq  C\|u\|_{\mathcal V^s_{p,\Gamma}(\Omega)}\|v\|_{H^s(\R^n)},
    \\
    |I_2| 
    &\leq [u,u]_{\mathcal V^s_{0,\Gamma}(\Omega)}^{\frac 12}
    [v,v]_{\mathcal V^s_{0,\Gamma}(\Omega)}^{\frac 12}
   \leq  C\|u\|_{\mathcal V^s_{p,\Gamma}(\Omega)}\|v\|_{H^s(\R^n)}.
\end{align*}
This implies $g\in H^{-s}(\Omega)$ and 
$\|g\|_{H^{-s}(\Omega)}\leq C \|u\|_{\mathcal V^s_{p,\Gamma}(\Omega)}$.

By Proposition ~\ref{prop:AspgammaInh}, if $\lambda\notin \tilde\Sigma^s_{p,\Gamma;q}=\Sigma^s_{p,\Gamma;q}$, $w\in \widetilde H^s(\Omega)$ exists, is unique and satisfies
\begin{align*}
    \|w\|_{\mathcal V^s_{p,\Gamma}(\Omega)}
    \leq \|w\|_{\mathcal V^s_{p,\Gamma}(\R^n)}
    \leq C\|w\|_{H^s(\R^n)}
    \leq C\|\tilde f\|_{\mathcal V^s_{p,\Gamma}(\Omega)}.
\end{align*}

We emphasize here that $u$ is independent of the choice of $\tilde f\in \mathcal V^s_{p,\Gamma}(\Omega)$ with $\tilde f\in[f]$.
Indeed, let us consider  $\tilde f_1,\tilde f_2\in \mathcal V^s_{p,\Gamma}(\Omega)$ such that $\tilde f_1, \tilde f_2\in[f]$, i,e, $\mathcal E_{\Omega+\Gamma_p}(\tilde f_1-\tilde f_2)\in \widetilde H^s(\Omega)$. Let $u_1, u_2$ be constructed as above with $\tilde f_1,\tilde f_2$, respectively. 
Let $v=\mathcal E_{\Omega+\Gamma_p}(u_1-u_2)$. Since $\mathcal E_{\Omega+\Gamma_p}(\tilde f_1-\tilde f_2)\in \widetilde H^s(\Omega)$ and $w_1, w_2\in \widetilde H^s(\Omega)$, we have $v\in\widetilde H^s(\Omega)$. In addition, $\big(A^s_{p,\Gamma}(D)+q-\lambda)v=0$ and $v|_{(\Omega+\Gamma_p)\backslash\overline{\Omega}}=0$ by construction. Since $\lambda\notin\Sigma^s_{p,\Gamma;q}$,  by Proposition ~\ref{prop:AspgammaInh}, $v=0$.

Finally, we observe that
\begin{align*}
    \|u\|_{\mathcal V^s_{p,\Gamma}(\Omega)}
    \leq \|\tilde f\|_{\mathcal V^s_{p,\Gamma}(\Omega)}+\|w\|_{\mathcal V^s_{p,\Gamma}(\Omega)}
    \leq C\|\tilde f\|_{\mathcal V^s_{p,\Gamma}(\Omega)}.
\end{align*}
Taking the infimum among all possible $\tilde f$, the estimate  follows.
\end{proof}

We conclude the discussion about well-posedness with a direct corollary of Proposition ~\ref{prop:AspgammaExt} in the more restrictive function spaces related  to $H^s(\Omega+\Gamma_p)$.

\begin{cor}\label{cor:AspgammaExtHs}
Under the same conditions as in Proposition ~\ref{prop:AspgammaExt}, if $\lambda\notin\Sigma^s_{p,\Gamma;q}$ and $f\in H^s_{p,\Gamma}(\Omega)$, 
the unique solution satisfies $u\in H^s(\Omega+\Gamma_p)$ and
\begin{align*}
    \|u\|_{H^s(\Omega+\Gamma_p)}\leq C \|f\|_{H^s_{p,\Gamma}(\Omega)}.
\end{align*}
\end{cor}

\begin{proof}
Let $\tilde f\in H^s(\R^n)$ such that $\tilde f\in [f]$. 
Then $\tilde f|_{\Omega+\Gamma_p}\in \mathcal V^s_{p,\Gamma}(\Omega)$ and 
\begin{align*}
    \|\tilde f\|_{\mathcal V^s_{p,\Gamma}(\Omega)}
    \leq \|\tilde f\|_{\mathcal V^s_{p,\Gamma}(\R^n)}
    \leq C\|\tilde f\|_{H^s(\R^n)}.
\end{align*}

Arguing as in Proposition ~\ref{prop:AspgammaExt}, we construct $\tilde u\in H^s(\R^n)$ as $\tilde u=w+\tilde f$, where $w\in \widetilde H^s(\Omega)$ and
\begin{align*}
    \|w\|_{H^s(\R^n)}
    \leq  C\|\tilde f\|_{\mathcal V^s_{p,\Gamma}(\Omega)}
    \leq C\|\tilde f\|_{H^s(\R^n)}.
\end{align*}
This finally implies 
\begin{align*}
    \|\tilde u\|_{H^s(\R^n)}
    \leq C\|\tilde f\|_{H^s(\R^n)}.
\end{align*}
The result follows by taking the infimum among all possible $\tilde f$ and defining $u:=\tilde u|_{\Omega+\Gamma_p}$.

We emphasize again that $u$ does not depend on the choice of $\tilde f$ ($\tilde u$ does, but outside $(\Omega+\Gamma_p)\backslash\overline{\Omega}$) which follows as in Proposition ~\ref{prop:AspgammaExt}.
\end{proof}

\begin{lem}\label{lem:idtwosolsp}
Let $s\in (0,1)\backslash\big\{\frac12\big\}$, $p\in[0,1]$,  $\Gamma\subset \R^n\backslash\{0\}$ be an open, convex cone, $\Omega\subset\R^n$ be open, bounded, Lipschitz and $q\in L^\infty(\Omega)$ be such that $0\notin\Sigma^s_{p,\Gamma;q}$. Let $f\in  H^s_{p,\Gamma}(\Omega)$ and let $u\in   H^s(\Omega+\Gamma_p)$ be the solution to the exterior value problem from Corollary ~\ref{cor:AspgammaExtHs}. Then $u=\tilde P_{p,\Gamma;q}f$  from Definition ~\ref{def:Poissonspgamma}.
\end{lem}

\begin{proof}
The proof works as the one of Lemma ~\ref{lem:idtwosols}. 
It suffices to note that the constructions of $u$ in Corollary ~\ref{cor:AspgammaExtHs} and Proposition  ~\ref{prop:tildeB_exist} agree. 
\end{proof}

\subsection{The Dirichlet-to-Neumann operator}

Based on the bilinear form $B_{q,\Gamma;q}^s$, we define the Dirichlet-to-Neumann operator as follows:

\begin{defi}[Dirichlet-to-Neumann operator $\Lambda_{p,\Gamma;q}$]
\label{defi:DtNspgamma}
Let $s\in (0,1)\backslash\big\{\frac12\big\}$, $p\in [0,1]$, $q\in L^\infty(\Omega)$ be such that $0\notin\Sigma^s_{p,\Gamma;q}$ and $B^s_{p,\Gamma;q}$ be as in \eqref{eq:Bspgammaq}.
We set
\begin{align*}
    \Lambda_{p,\Gamma;q}:  H^s_{p,\Gamma}(\Omega)\to H^{-s}(\R^n): f \mapsto \Lambda_{p,\Gamma;q} f,
\end{align*}
with  
\begin{align}
\label{eq:DtNnD}
    \langle \Lambda_{p,\Gamma;q} f, h\rangle:=B^s_{p,\Gamma;q}(u_f, h) \; \mbox{ for any } h\in H^s(\R^n),
\end{align}
where $u_f=\tilde P_{p,\Gamma;q}f\in H^s(\Omega+\Gamma_p)$.
\end{defi}

We remark that in this definition, because of the domains of dependence in \eqref{eq:Bspgammaq}, there is no need to consider an extension of $u_f$. 
This makes possible to define $\Lambda_{p,\Gamma;q}$ on $H^s_{p,\Gamma}(\Omega)$ and not only on $\widetilde H^s((\Omega+\Gamma_p)\backslash\overline{\Omega})$ as for $\tilde \Lambda_{p,\Gamma;q}$ in Definition ~\ref{defi:tDtNspgamma}.
We also notice that the operator $\Lambda_{p,\Gamma;q}$ is bounded and well-defined, which follows from \eqref{eq:BspgammaBound}.

We next observe that it is possible to relate the Dirichlet-to-Neumann operators $\tilde \Lambda_{p,\Gamma;q}$ and $\Lambda_{p,\Gamma;q}$ under appropriate geometric assumptions. 

\begin{lem}\label{lem:idtwoDtNpgamma}
Let $s\in (0,1)\backslash\big\{\frac12\big\}$, $p\in (0,1)$, $q\in L^\infty(\Omega)$ be such that $0\notin\Sigma^s_{p, \Gamma;q}$ and $\Lambda_{p,\Gamma;q}$, $\tilde \Lambda_{p,\Gamma;q}$ be as in Definitions ~\ref{defi:DtNspgamma} and ~\ref{defi:tDtNspgamma}, respectively.
Let $W_1, W_2\subset \mathcal  (\Omega+\Gamma_p)\backslash\overline{\Omega}$ and  let $f\in \widetilde H^s(W_1)$. Then, 
\begin{align*}
    \Lambda_{p,\Gamma;q} f|_{W_2}= \tilde \Lambda_{p,\Gamma;q} f|_{W_2}
\end{align*}
if one of the following conditions holds:
\begin{itemize}
    \item $W_1\subset (\Omega\pm\Gamma)\backslash(\overline{\Omega}\mp\Gamma)$ and $W_2\subset (\Omega\mp\Gamma)\backslash(\overline{\Omega}\pm\Gamma)$, 
    \item $W_j+\Gamma\subset \Omega+\Gamma$ for $j\in\{1,2\}$,
    \item $W_1\cap (W_2+\Gamma)=\emptyset$.
\end{itemize}
\end{lem}

\begin{proof}[Proof of Lemma \ref{lem:idtwoDtNpgamma}]
We seek to prove that
$\tilde B^s_{p,\Gamma;q}(u_f, h)= B^s_{p,\Gamma;q}(u_f, h)$ for $h\in \widetilde H^s(W_2)$ under the stated geometric conditions on $W_1,W_2$.
Indeed, by Proposition ~\ref{prop:tildeB_exist}
\begin{align*}
    \tilde B^s_{p,\Gamma;q}(u_f,h)- B^s_{p,\Gamma;q}(u_f,h)
    &=B^s_{p,\Gamma;q}(u_f-f,h)- B^s_{p,\Gamma;q}(u_f-f,h)+
    \tilde B^s_{p,\Gamma;q}(f,h)- B^s_{p,\Gamma;q}(f,h).
\end{align*}
Since $u_f-f\in \widetilde H^s(\Omega)$, by  Lemma ~\ref{lem:equiv1}, 
for $p\in(0,1)$ it holds
\begin{align*}
    \tilde B^s_{p,\Gamma;q}(u_f,h)- B^s_{p,\Gamma;q}(u_f,h)
    &=
    \tilde B^s_{p,\Gamma;q}(f,h)- B^s_{p,\Gamma;q}(f,h)
\end{align*}
Now, using the definitions \eqref{eq:Bspgammaq} and \eqref{eq:tBspgammaq},
\begin{align*}
    &\tilde B^s_{p,\Gamma;q}(f,h)- B^s_{p,\Gamma;q}(f,h)
    \\&\quad=\beta_s\int_{\R^n\backslash(\overline{\Omega}+\Gamma_p)}\int_{\Gamma_p}f(x+y)h(x-y)
    \nu_{p,\Gamma}(y)dydx
    \\ &\qquad + \beta_s\int_{\Omega+\Gamma_p}\int_{\Gamma_p}f(x)\big(h(x)-h(x-y)\big)
    \nu_{p,\Gamma}(y)\chi_{\R^n\backslash(\overline{\Omega}+\Gamma_p)}(x+y)dydx
    \\ &\quad = \beta_s\int_{W_1}\int_{\Gamma_p}f(x)\big(h(x)-h(x-y)\big)
    \nu_{p,\Gamma}(y)\chi_{\R^n\backslash(\overline{\Omega}+\Gamma_p)}(x+y)dydx
    \\&\quad = \beta_s\int_{W_1\cap W_2}\int_{\Gamma_p}f(x)h(x)
    \nu_{p,\Gamma}(y)\chi_{\R^n\backslash(\overline{\Omega}+\Gamma_p)}(x+y)dydx
    \\&\qquad-\beta_s\int_{W_1}\int_{\Gamma_p}f(x)h(x-y)
    \nu_{p,\Gamma}(y)\chi_{\R^n\backslash(\overline{\Omega}+\Gamma_p)}(x+y)dydx\;.
\end{align*}
The first term vanishes if $W_1\cap W_2=\emptyset$ or if $(W_1\cap W_2)+\Gamma_p\subset \Omega+\Gamma_p$ (so $x+y\in \Omega+\Gamma_p$, see Figure ~\ref{figure:settingidtwoDtN}).

It is clear that the second term vanishes in any of the following cases:
\begin{itemize}
    \item  $W_1+\Gamma_p\subset \Omega+\Gamma_p$, since again in this case $x+y\in \Omega+\Gamma_p$.
    \item $W_2+\Gamma_p\subset \Omega+\Gamma_p$, since $(x-y)+2y\in \Omega+\Gamma_p$.
    \item  $W_1\cap(W_2+\Gamma_p)=\emptyset$ or $W_2\cap(W_1+\Gamma_{1-p})=\emptyset$, since then $f(x)h(x-y)=0$ for all $x\in W_1$ and $y\in \Gamma_p$.
\end{itemize}
In general, the integral  vanishes if for all $x\in W_1$ and $y\in \Gamma_p$ such that $x+y\notin \Omega+\Gamma_p$, one has $x-y\in W_2$.
This is in particular  satisfied in the following further setting:
\begin{itemize}
   \item $W_1\subset (\Omega+\Gamma)\backslash(\overline{\Omega}-\Gamma)$ and $W_2\subset (\Omega-\Gamma)\backslash(\overline{\Omega}+\Gamma)$ (or the other way around). In this case, if $x+y\notin\Omega+\Gamma_p$, then $y\in-\Gamma$ and so $x-y\in \Omega+\Gamma$, which does not intersect $W_2$.
\end{itemize}
\end{proof}

\begin{rmk}
We remark that the conditions given in Lemma ~\ref{lem:idtwoDtNpgamma} are not exhaustive. As the proof illustrates, it would for instance have been possible to impose any condition ensuring that the two integrals giving the difference of the bilinear forms vanish.
\end{rmk}

\begin{rmk}
In the case $p\in\{0,1\}$, we do not have a similar identification as above if $W_2\subset (\Omega+\Gamma_{1-p})\backslash\overline{\Omega}$, which is the relevant region containing information about $u_f$. Indeed, by the previous discussion, if $p\in\{0,1\}$, $W_1\subset (\Omega+\Gamma_p)\backslash\overline{\Omega}$ and $W_2\subset (\Omega+\Gamma_{1-p})\backslash\overline{\Omega}$, then $\tilde B^s_{p,\Gamma;q}(f,h)= B^s_{p,\Gamma;q}(f,h)$. 
However, by Lemma ~\ref{lem:equiv1},  $\tilde B^s_{p,\Gamma;q}(u_f-f,h)\neq B^s_{p,\Gamma;q}(u_f-f,h)$ for $h\neq 0$ supported in $W_2$.
\end{rmk}

Last but not least, we remark that as for the bilinear form $\tilde{B}_{p,\Gamma;q}$ one could also study the inverse problem for $B_{p,\Gamma;q}$. In the cases in which $\Lambda_{p,\Gamma;q} = \tilde{\Lambda}_{p,\Gamma;q}$ which for instance hold for $W_1$ and $W_2$ as in Lemma ~\ref{lem:idtwoDtNpgamma}, we may invoke the antilocality results from above. In general, the settings however differ and we lack a good Alessandrini identity. We do not pursue this any further here.

\section{Technical Results on Domains}
\label{appendix:domains}

This appendix contains some technical results involving  two-sided cones that are used in the article. 
Throughout the whole section, we assume that $\mathcal C\subset\R^n\backslash\{0\}$ is a convex, non-empty, open cone.

\begin{lem}\label{lem:C1dom}
Let  $\Omega\subset \R^n$ be a bounded differentiable domain. Then
\begin{align*}
    \Omega\Subset \mathcal C(\Omega).
\end{align*}

\end{lem}

\begin{proof}
We argue by contradiction. Assume that $\partial\Omega$ is differentiable but $\Omega\not\Subset \mathcal C(\Omega)$. Since $\Omega\subset\mathcal C(\Omega)$ but $\overline{\Omega}\not\subset \mathcal C(\Omega)$, there must exist a point $x\in\partial\Omega\cap\partial \mathcal C(\Omega)$. 
Since $\p\Omega$ is differentiable at $x$, the tangent space $T$ to $\Omega$ through $x$ is a well-defined hyperplane. 

We claim that there exists $\theta\in\Ss^{n-1}$ such that $\theta\in -\mathcal C\cup\mathcal C$ but $\{x+t\theta, t\in\R\}\not\subset T$.
Indeed, the opposite statement would mean $\mathcal C(\{x\})\subseteq T$, which cannot be true because of the openness of $\mathcal C$.
This implies  that
there exists $y\in \Omega$ such that $y=x+t\theta$ for some $t\in\R$.
However, since we are considering two-sided cones, we have  $x\in \mathcal C(\{y\})\subset \mathcal C(\Omega)$, which contradicts the assumption  $x\in \partial\mathcal C(\Omega)$.
\end{proof}

\begin{rmk}
The two-sidedness of $\mathcal C(\{y\})$ is essential for the validity of the above lemma. In fact, for a one-sided cone one can easily cook up examples in which $\Omega \Subset \mathcal C(\Omega)$ does not hold, as shown in Figure ~\ref{figure:onecone}.
\end{rmk}

\begin{lem}\label{lem:Lipdom}
Let $\Omega\subset \R^n$ be a bounded Lipschitz  domain. Then
\begin{align*}
    \mathcal H^{n-1}\big(\overline{\Omega}\cap (\R^n\backslash\mathcal C(\Omega))\big)=0.
\end{align*}
\end{lem}

\begin{proof}
Since $\Omega\subset \mathcal C(\Omega)$, we have $\overline{\Omega}\cap (\R^n\backslash\mathcal C(\Omega))= \partial\Omega\cap\partial \mathcal C(\Omega)$. 
We have already seen in the proof of Lemma ~\ref{lem:C1dom} that if $\p\Omega$ is differentiable at $x$, then $x\notin\p\mathcal C(\Omega)$. This implies
\begin{align*}
    \partial\Omega\cap\partial \mathcal C(\Omega)\subset \{x\in\p\Omega: \p\Omega \mbox{ is not differentiable at } x\}.
\end{align*}
Finally, since $\p\Omega$ is Lipschitz, by the Rademacher's theorem it holds 
$$\mathcal H^{n-1}(\{x\in\p\Omega: \p\Omega \mbox{ is not differentiable at } x\})=0,$$
which concludes the proof.
\end{proof}

\begin{lem}\label{lem:coveringdouble}
Let $\Omega\subset \R^n$ be a connected, bounded, differentiable domain.
For any $U_0\subset \Omega$, let $U_{j+1}:=\Omega\cap\mathcal C(U_j)$ for $j\in\N_0$.
Then there exists an integer $N\in\N$ such that $U_j=\Omega$ for $j\geq N$.
\end{lem}

\begin{proof}

We argue by contradiction. Let us assume that there is an open set $U_0\subsetneq \Omega$ such that $U_j\subsetneq \Omega$ for all $j\in\N$, where $U_{j}=\mathcal C(U_{j-1})\cap\Omega$.
Then, either
\begin{itemize}
    \item[$(i)$]there is an integer $J\in\N$ such that $U_{j}=U_J\subsetneq \Omega$ for all $j\geq J$, or 
    \item[$(ii)$] $U_j\subsetneq U_{j+1}\subsetneq \Omega$ for all $j\in\N_0$.
\end{itemize}

\emph{Step 1.} If condition $(i)$ is satisfied, then $
   \mathcal C(U_J)\cap \Omega = U_J \subsetneq \Omega. $
Since $\Omega$ is connected, $\p U_J\cap\Omega\neq\emptyset$.
Let $z\in \p U_J\cap\Omega$ and $\tau>0$ be such that $B_\tau(z)\subset \Omega$.
For any $x\in B_{\sfrac{\tau}2}(z)\backslash U_J$,
let us define
\begin{align*}
    \epsilon:=\sup\{\delta>0: B_{\delta}(x) \subset \Omega\backslash\overline{U}_J\} >0.
\end{align*}
We know that $B_{\epsilon'}(x)\subset\mathcal C(B_{\epsilon}(x))\cap \Omega$ for  $\epsilon'>\epsilon$ small enough. Because of the choice of $\epsilon$, there must be $y\in B_{\epsilon'}(x)\cap U_J$. 
However, this implies $\mathcal C(\{y\})\cap B_\epsilon (x)\neq \emptyset $.
In addition, $ \mathcal C(\{y\})\cap B_\epsilon (x) \subset \mathcal C(U_J)\cap \Omega = U_J$, which implies  $B_\epsilon(x)\cap U_J\neq \emptyset$. Since this is absurd, condition $(i)$ does not hold.

\emph{Step 2.} 
If condition $(ii)$ holds, we can find $x_j\in U_{j+1}\backslash\overline{U}_j$ for all $j\in\N$.
Due to the boundedness of $\overline{\Omega}$, there exists a subsequence $\{x_{j_k}\}_{k=1}^\infty$ converging to $ x\in\overline{\Omega}$. Firstly, we observe that $ \mathcal C(\{x\})\cap \Omega\neq\emptyset$. This is immediate if $x\in\Omega$, while if $x\in\p\Omega$ it follows from the differentiability of $\p\Omega$ at $x$ and Lemma ~\ref{lem:C1dom}. Secondly, we notice  that, by construction, $x\notin \overline{U}_j$ for all $j\in\N_0$. In addition,  $x_j\notin \mathcal C(\{x\})\cap \Omega$. Indeed, if $x_j\in \mathcal C(\{x\})\cap \Omega$,  then $x\in \mathcal C(\{x_j\})\cap\overline{\Omega} \subset \mathcal C(U_{j+1})\cap\overline{\Omega}=\overline{U}_{j+2}$, which is false.

If $y\in \mathcal C(\{x\})\cap \Omega$, then
it holds $y\notin U_j$ for all $j\in\N$, otherwise $x\in \mathcal C(\{y\})\cap \overline{\Omega}\subset  \overline{U}_{j+1}$.
Moreover, $x_j\notin \mathcal C(\{y\})$, since  the contrary would imply  $y\in  \mathcal C(\{x_j\})\cap  \Omega\subset   U_{j+2}$.
Finally, let $\epsilon>0$ be so small that  $B_\epsilon(x)\subset\mathcal C(\{y\})$.
Then we conclude  that $x_j\notin B_\epsilon(x)$ for all $j\in\N$, contradicting that $\lim_{k\to\infty}x_{j_k}=x$.
\end{proof}

\begin{rmk}\label{rmk:coveringnoC1}
If $\Omega$ is Lipschitz but not differentiable, Lemma ~\ref{lem:coveringdouble} may not hold. Indeed, let us consider the following 2-dimensional counterexample (see also Figure ~\ref{figure:nocoveringcone}):
\begin{align*}
    \mathcal C&=\{(t\cos\theta, t\sin\theta):\  t>0,\ \theta\in(-\alpha,\alpha) \},\\
    \Omega&=\{(t\sin\theta, t\cos\theta):\  t>0,\ \theta\in(-\beta,\beta), \ t\cos\theta<h \}.
\end{align*}
Let $U_0$ be such that 
\begin{align*}
    U_1=\mathcal C(U_0)\cap\Omega=\big\{(x,y)\in\Omega: y> \min\{y_0+ (\tan\alpha) x, y_0- (\tan\alpha) x\}\big\}
\end{align*}
for some $y_0<h$. 
By simple geometric calculations, we can infer 
\begin{align*}
    U_2=\mathcal C(U_1)\cap\Omega=\big\{(x,y)\in\Omega: y> \min\{y_1+ (\tan\alpha) x, y_1- (\tan\alpha) x\}\big\}, 
\end{align*}
where
\begin{align*}
    y_1=y_0\left(\frac{(\tan\beta)^{-1}-\tan \alpha}{(\tan\beta)^{-1}+\tan \alpha}\right).
\end{align*}
Iterating this argument, it follows
\begin{align*}
    U_{j+1}=\big\{(x,y)\in\Omega: y> \min\{y_j+ (\tan\alpha) x, y_j- (\tan\alpha) x\}\big\}, 
\end{align*}
with
\begin{align*}
    y_j=y_0\left(\frac{(\tan\beta)^{-1}-\tan \alpha}{(\tan\beta)^{-1}+\tan \alpha}\right)^j.
\end{align*}
If $(\tan\beta)^{-1}>\tan \alpha$, i.e. $\alpha+\beta<\frac{\pi}{2}$, we have $\lim_{j\to \infty}y_j= 0$, but $0$ can not be achieved in a finite number of steps
\end{rmk}

\begin{figure}[t]
\begin{tikzpicture}

\pgfmathsetmacro{\anglealpha}{20}
\pgfmathsetmacro{\anglebeta}{30}

\pgfmathsetmacro{\calpha}{cos(\anglealpha)}
\pgfmathsetmacro{\salpha}{sin(\anglealpha)}

\pgfmathsetmacro{\cbeta}{cos(\anglebeta)}
\pgfmathsetmacro{\sbeta}{sin(\anglebeta)}

\pgfmathsetmacro{\A}{tan(\anglealpha)}
\pgfmathsetmacro{\B}{1/tan(\anglebeta)}

\pgfmathsetmacro{\height}{5}
\pgfmathsetmacro{\r}{.5}
\pgfmathsetmacro{\c}{\height*\cbeta-\r/\calpha}

\pgfmathsetmacro{\yzero}{\c-\r/\calpha}
\pgfmathsetmacro{\intzerox}{\yzero/(\A+\B)}
\pgfmathsetmacro{\intzeroy}{\B*\yzero/(\A+\B)}

\pgfmathsetmacro{\yone}{\yzero*(\B-\A)/(\B+\A)}
\pgfmathsetmacro{\intonex}{\yone/(\A+\B)}
\pgfmathsetmacro{\intoney}{\B*\yone/(\A+\B)}

\pgfmathsetmacro{\ytwo}{\yone*(\B-\A)/(\B+\A)}
\pgfmathsetmacro{\inttwox}{\ytwo/(\A+\B)}
\pgfmathsetmacro{\inttwoy}{\B*\ytwo/(\A+\B)}

\pgfmathsetmacro{\ythree}{\ytwo*(\B-\A)/(\B+\A)}
\pgfmathsetmacro{\intthreex}{\ythree/(\A+\B)}
\pgfmathsetmacro{\intthreey}{\B*\ythree/(\A+\B)}

\pgfmathsetmacro{\yfour}{\ytwo*(\B-\A)/(\B+\A)}
\pgfmathsetmacro{\intfourx}{\yfour/(\A+\B)}
\pgfmathsetmacro{\intfoury}{\B*\yfour/(\A+\B)}

\pgfmathsetmacro{\yfive}{\ythree*(\B-\A)/(\B+\A)}
\pgfmathsetmacro{\intfivex}{\yfive/(\A+\B)}
\pgfmathsetmacro{\intfivey}{\B*\yfive/(\A+\B)}

\begin{scope}[scale=1]

\draw[black!50, dashed] (-3,0)--(3,0);
\draw[black!50, dashed] (0, -.3)--(0,\height);

\draw[orange] (0,0)--(-\height*\sbeta, \height*\cbeta)--(\height*\sbeta, \height*\cbeta)--cycle;
\node[orange] at (.5,.2) {$\Omega$};

\fill[red!50, opacity=.2] (\intzerox, \intzeroy)--(0, \yzero)--(-\intzerox, \intzeroy)--(-\height*\sbeta, \height*\cbeta)--(\height*\sbeta, \height*\cbeta)--cycle;
\node[red] at (1,\yzero) {\footnotesize $U_1$};

\node[red] at (0, \yone-.2) {\scriptsize $(0,y_1)$};

\draw[blue,line width=.7pt] (0, \c) circle [radius=\r];
\node[blue] at (.3, \c) {\footnotesize $U_0$};

\node[blue] at (0, \yzero-.2) {\scriptsize $(0,y_0)$};

\draw[blue, dashed, opacity=.5, line width=.7pt, domain=0:2] plot (0+\x*\calpha, \yzero-\x*\salpha);
\draw[blue, dashed, opacity=.5, line width=.7pt, domain=0:2] plot (0-\x*\calpha, \yzero-\x*\salpha);

\fill[pattern color= green!90, pattern= crosshatch dots, opacity=.5]  (\intonex, \intoney)--(0, \yone)--(-\intonex, \intoney)--(-\height*\sbeta, \height*\cbeta)--(\height*\sbeta, \height*\cbeta)--cycle;
\node[green] at (1,\yone) {\footnotesize $U_2$};

\draw[red, dashed, opacity=.5, domain=0:3.2, line width=.7pt] plot (\intzerox-\x*\calpha, \intzeroy-\x*\salpha);
\draw[red, dashed, opacity=.5, domain=0:3.2, line width=.7pt] plot (-\intzerox+\x*\calpha, \intzeroy-\x*\salpha);

\draw[green, dashed, opacity=.5, domain=0:2, line width=.7pt] plot (\intonex-\x*\calpha, \intoney-\x*\salpha);
\draw[green, dashed, opacity=.5, domain=0:2, line width=.7pt] plot (-\intonex+\x*\calpha, \intoney-\x*\salpha);

\draw[black!50] (0, \ytwo)--(\inttwox, \inttwoy);
\draw[black!50] (0, \ytwo)--(-\inttwox, \inttwoy);

\draw[black!50] (0, \ythree)--(\intthreex, \intthreey);
\draw[black!50] (0, \ythree)--(-\intthreex, \intthreey);

\draw[black!50] (0, \yfour)--(\intfourx, \intfoury);
\draw[black!50] (0, \yfour)--(-\intfourx, \intfoury);

\draw[black!50] (0, \yfive)--(\intfivex, \intfivey);
\draw[black!50] (0, \yfive)--(-\intfivex, \intfivey);

\end{scope}

\begin{scope}[xshift=-6cm, yshift=3cm, scale=.6]
\node[black!50] at (-.5,0) {$\mathcal C$};
\draw[black!50,line width=.8pt, domain=0:2.2] plot ({\x*\calpha},{-\x * \salpha});
\draw[black!50,line width=.8pt, domain=0:2.2] plot ({\x*\calpha},{\x * \salpha});
\draw[black!50,line width=.7pt, domain=-.2:2.2, dashed, opacity=.5] plot ({\x},{0});

\draw [cyan,thick,domain=0:\anglealpha, ->] plot ({1.8*cos(\x)}, {1.8*sin(\x)});
\node [cyan] at ({2.2*cos(\anglealpha/2)},{2.2*sin(\anglealpha/2)}) {$\alpha$};

\end{scope}

\begin{scope}[xshift=-5.2cm, yshift=0cm, scale=.6]

\draw[orange,line width=.8pt, domain=0:2.2] plot ({\x*\sbeta},{\x * \cbeta});
\draw[orange,line width=.8pt, domain=0:2.2] plot ({-\x*\sbeta},{\x * \cbeta});
\draw[black!50,line width=.7pt, domain=-.2:2.2, dashed, opacity=.5] plot ({0},{\x});

\draw [black!50,thick,domain=90-\anglebeta:90, ->] plot ({1.8*cos(\x)}, {1.8*sin(\x)});
\node [black!50] at ({2.2*cos(90-\anglebeta/2)},{2.2*sin(90-\anglebeta/2)}) {$\beta$};

\end{scope}

\end{tikzpicture}
\caption{Counterexample of Lemma ~\ref{lem:coveringdouble} for Lipschitz doimains, as explained in Remark ~\ref{rmk:coveringnoC1}.}
\label{figure:nocoveringcone}
\end{figure}
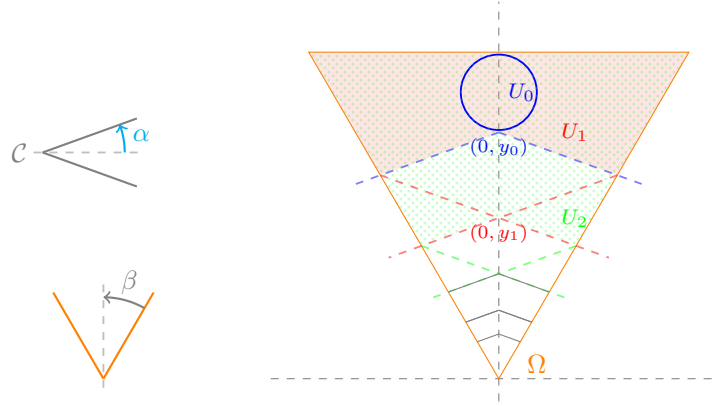

\bibliographystyle{alpha}
\bibliography{citationsHT4}

\end{document}